%% file: main.tex
\documentclass[article]{amsart}
\pdfoutput=1
\usepackage{main}
\usepackage{standalone}

\title{Symmetries of colored Khovanov--Rozansky homology}
\author{Alexis Guérin}
\address{Université Clermont Auvergne, LMBP, Campus des Cézeaux, 3 place Vasarely, TSA 60026, CS 60026, 63178 Aubière Cedex, France}
\email{\href{mailto:alexis.guerin@uca.fr}{alexis.guerin@uca.fr}}

\author{Felix Roz}
\address{Department of Mathematics Columbia University Room 509, MC 4406 2990 Broadway New York, NY 10027}
\email{\href{mailto:fr2588@columbia.edu}{fr2588@columbia.edu}}
\date{}

\begin{document}

\begin{abstract}

We construct an action of the positive half of the Witt algebra on colored $\mathfrak{gl}_N$-Khovanov--Rozansky homology and prove that this action is compatible with the functoriality of colored link homology. To establish this result, we use Bar-Natan's framework of canopolies to formulate and prove the relevant statements at the level of tangles.

\end{abstract}


\pagenumbering{arabic}

\maketitle

\setcounter{tocdepth}{2}
\tableofcontents

\addtocontents{toc}{\protect\setcounter{tocdepth}{1}}
\section{Introduction}

\input{Z_Introduction}

\subsection{Notations}

\input{Z_Notations}

\addtocontents{toc}{\protect\setcounter{tocdepth}{2}}

\section{The Witt action on foams}
\label{sec_Witt&foams}

\subsection{Webs and foams}
\label{sec_web_foam}

\input{Z_Web_and_Foams}

\subsection{The Witt algebra}

\input{Z_Witt}

\subsection{The action on foams}
\input{Z_action}

\section{Locality}
\label{sec_can}

\subsection{The canopolis}
\input{Z_canopoly}

\subsection{Between topology and algebra}

\input{Z_BtwTA}

\section{Red dots}
\label{sec_reddots}

\subsection{Twisted actions}

\input{Z_twist_summary}

\subsection{The relative homotopy category}
\label{RHC}

\input{Z_RHC_2}

% moved to a summarized version 
% \section{Twist the action}
% \input{Z_more_equi}

% \subsection{Maurer--Cartan elements}
% \input{Z_Maurer-Cartan}

% \subsection{Maurer--Cartan elements in $\algtop$}
% \input{Z_twists_on_webs}

% \subsection{Red dots}
% \input{Z_red_dot}

% \subsection{Useful morphisms}
% \input{Z_useful_morph}

\section{Colored Khovanov--Rozansky homology}
\label{secConstr}

\subsection{The complex}

\input{Z_complex}

\subsection{cKR-simplicity}

\input{Z_simplicity}

\subsection{The blist hypothesis}
\label{sec_blist}

\input{Z_blist}

\subsection{Forkslide moves}

\input{Z_forkslide}

\subsection{Forkslide isomorphisms}
\input{Z_forkiso}

\section{Invariance}
\label{sec_invar}

\subsection{The uncolored case}
\label{sec_uncol}

\input{Z_uncolored}

\subsection{Reidemeister 2 and 3}
\input{Z_R23}

\subsection{Reidemeister 1}

\input{Z_R1}

\section{Functoriality}
\label{sec_func}
This section is devoted to the functoriality of the construction introduced above. We first establish the sliding lemma stated in Proposition~\ref{slid_rd}, together with the various intermediate results needed for its proof. The functoriality statement is then formulated in Subsection~\ref{funct}, using this proposition.

\subsection{Sliding lemmas}
\label{sliding_sec}

\input{Z_sliding_g}

\subsection{Extended canopoly}
\label{ext_can}

\input{Z_ext_can}

\subsection{Witt and extended foams}
\label{ext_f_w}
\input{Z_witt_ext}

\subsection{Orange dot sliding}
\label{slid_or}

\input{Z_sliding_orange}

\subsection{Functoriality}
\label{funct}

\input{Z_functoriality}

\newpage

\newpage 
\bibliographystyle{alphaurl}
\bibliography{biblio}

\end{document}

%% file: Z_Introduction.tex
\subsection{Goal of the paper}

Khovanov's categorification of the Jones polynomial~\cite{Kho0} led to the development of different homological theories. Khovanov and Rozansky introduced a link homology theory that categorifies the $\mathfrak{gl}_N$ Reshetikhin--Turaev invariant \cite{KR08a,KR08b}. The colored version of this homology was introduced by Yonezawa \cite{Yo11} and Wu \cite{Wu07}, where the color refers to the $k$-th quantum exterior power of the standard $U_q(\mathfrak{gl}_N)$-representation assigned to each component of the link. The construction of Lie algebra actions on this homology was initiated by a conjecture of Gorsky, Oblomkov, and Rasmussen~\cite{GOR13}, which relates colored link homologies to characters of representations of affine Lie algebras. In this paper, we establish the following theorem. 
\begin{thm}
    The colored $\mathfrak{gl}_N$-Khovanov--Rozansky homology is a representation of the positive part of the Witt algebra and this structure is invariant under links isotopies.
\end{thm}
Moreover this construction gives a functor from a category of colored links with additional decorations called red dots to the category of Witt-modules. 
\[\left \{
\begin{array}{c}
    \text{Colored links in $\R^3$ with red dots}   \\
    \text{Cobordisms}  
\end{array}
\right \}
\longrightarrow
\left \{
\begin{array}{c}
    \text{Witt-modules}\\
    \text{Witt equivariant maps}
\end{array} 
\right \}
\]

\subsection{Webs and foams}

Colored Khovanov--Rozansky homology is a categorification of the quantum $\mathfrak{gl}_N$ Reshetikhin--Turaev link invariant colored by exterior powers. Since its original construction in terms of matrix factorizations \cite{KR08a,KR08b}, Khovanov--Rozansky homology has been studied from several complementary perspectives, including higher representation theory, algebraic geometry, and symplectic geometry \cite{SS06,Man07,CK08,CK09}. In particular, it can be formulated via a complex of webs and foams, making the local structure of the theory and its functorial properties manifest \cite{ETW17}. 

The foam formulation of link homology was first introduced by Khovanov for $\mathfrak{sl}_3$ homology \cite{Kho04}. This construction was subsequently extended to higher-rank link homologies by several authors \cite{KR07,KR08a,KR08b,MSV09,QR16}. The construction used in this paper is based on the Robert--Wagner foam evaluation \cite{RW20a}
which categorifies MOY evaluation \cite{MOY98},
and uses the universal construction of \cite{BHMV92}.
 
\subsection{Locality via canopolies}

The foam formulation is closely related to the cobordism-based approach to Khovanov homology developed by Bar-Natan \cite{Bar05} and can be used to provide a local, tangle-based framework for studying link homology.
A convenient formalism for local considerations is given by the language of canopolies, which categorify Jones's notion of planar algebra \cite{Jon21}. This framework was used in the context of $\mathfrak{su}_3$ homology by Morrison and Nieh \cite{MN08}, for colored Khovanov homology by Carmen \cite{Ca14} and for odd Khovanov by Ebert and Schelstraete \cite{ES25}.
It was subsequently adapted to the matrix factorization formulation of Khovanov--Rozansky homology by Webster \cite{W07}.
In the setting of colored Khovanov--Rozansky homology, Ehrig, Tubbenhauer, and Wedrich developed the corresponding tangle and foam version of this framework and used it to establish the functoriality of colored Khovanov--Rozansky homology \cite{ETW17}.

The canopoly framework is essential for the formulation of the so-called \textit{blist hypothesis} of Proposition \ref{blisthyp}. This hypothesis was first established by Qi, Sussan, Robert, and Wagner \cite{QRSW21} in the setting of Soergel bimodules. We follow the same general strategy, using the blist hypothesis to extend the arguments of the uncolored case developed in \cite{GR25} to the colored setting.

\subsection{Symmetries of link homology}

Various symmetries have been constructed on link homologies, starting with the action of a positive part of the graded Witt algebra on HOMFLYPT homology introduced by Khovanov and Rozansky \cite{KR16}. Further symmetries have since been discovered \cite{GHM24,GLW18,CG24,KS25,ES25}. In \cite{QRSW23}, Qi, Robert, Sussan, and Wagner constructed action of the $\sld$ on $\mathfrak{gl}_N$-Khovanov--Rozansky homology, providing $p$-DG structures in positive characteristic.

Here, we study the action of the positive part of the Witt algebra, denoted by $\Witt$, following the work of Qi, Robert, Sussan, and Wagner \cite{QRSW24}. This algebra is also known as a Cartan-type Lie algebra, often denoted by $W_1$ for example in \cite{Ru74}. It acts naturally on symmetric polynomials, and this action can be extended to foams via the Robert--Wagner evaluation \cite{QRSW24}, and subsequently to the uncolored $\mathfrak{gl}_N$-Khovanov--Rozansky homology \cite{GR25}. 

The Lie algebra $\Witt$ is generated by a countable family of operators $(\wgen_k)_{k\geq -1}$, several of which have already appeared in the study of link homologies and related structures. Building on the results of \cite{KR16}, the generator $\wgen_1$ was used to define a $p$-DG structure on HOMFLY-PT homology and $\gl_{-2}$-homology, leading to categorifications of the Jones and colored Jones polynomials at roots of unity \cite{QRSW21,qi_p-differential_2022}. Similarly, the generator $\wgen_{-1}$ in \cite{QRSW24} coincides with the operator $\nabla$ introduced by Wang \cite{Wang24}, extending earlier work of Shumakovitch \cite{Shumakovitch_2014}.

\subsection{Structure of the paper}

In Section~\ref{sec_Witt&foams}, we recall the definitions of webs and foams and define a collection of categories denoted by $\catf$. We also recall the Witt action on foams, originally defined in \cite{QRSW24}.

In Section~\ref{sec_defcan}, we define the notion of a canopoly following the work of \cite{Bar05}, and in Section~\ref{BtwT&A}, we construct the canopoly $\algtop$.This canopoly has the same objects as $\catf$, but we only retain foams that behave “nicely” with respect to the Witt action.

Section~\ref{sec_reddots} summarizes the notions of Maurer--Cartan elements and their application to foams, which we call red dots. These elements allow us to construct more $\Witt$-equivariant foams. This leads to the definition of the canopolies $\ralgtop$ and the relative homotopy canopoly.

Sections~\ref{secConstr} and~\ref{sec_invar} are devoted to constructing the complex that we associate with each tangle. To prove the invariance under Reidemeister moves of this complex, we use the blist hypothesis from Section~\ref{sec_blist} to reduce the argument to the uncolored case treated in \cite{GR25}.

Finally, the functoriality statement of Section~\ref{funct} is preceded by four sections devoted to proving the sliding lemma of Proposition~\ref{slid_rd}. This result allows us to slide red dots along tangle diagrams and construct the canopoly $\rTD$. To prove this lemma, we need to extend the original canopoly of webs and foams by adding dots and trajectories. This new canopoly, denoted by $\extcatf$, allows us to work with new variables that are not involved in the decoration of foams.

\subsection{Acknowledgment}

We would like to thank Louis-Hadrien Robert, Joshua Sussan and Emmanuel Wagner for advising us on this project.
% as well as Matt Hogancamp for helpful discussions related to twisted representations.
We thank the authors of \cite{QRSW23} and \cite{QRSW24} for allowing us to use the diagrams from their papers.
The second author was partially supported by NSF grant DMS-2503714.

\subsection{AI usage}
AI was used solely to correct language errors in this paper.

%% file: Z_Notations.tex
Let us denote by $\N$ the non-negative integers, $\N^*:=\N \setminus \{0\}$, and $\N_{-1}:=\N \cup \{-1\}$.
Let \(N\) be a natural number which is fixed throughout the entire paper.

Let $\scalars$ be a unital ring in which 2 is invertible.
For all $x \in \scalars$, set $\Bar{x} := 1 - x$. 

% Let $\bnd$ be a finite sequence of integers such that exactly its first $n_\bnd$ terms are nonzero: $\bnd=(\bnd_1,\ldots,\bnd_{n_\bnd}) \in \Z_{\neq 0}^{n_\bnd}$. We allow $\bnd$ to be the empty sequence. We denote by $\Disk$ an oriented disk with a marked point on its boundary.

The symbol $\interval$ denotes the interval $[0,1]$. 

For \(k \in \N\), let us denote by $Sym_k := \scalars[x_1, \ldots, x_k]^{S_k}$ the ring of symmetric polynomials with $\scalars$ coefficients in $k$ variables. When working in such a ring, we will let \(e_i\), \(h_i\), and $p_i$ be the $i$th elementary, complete homogeneous, and power sum symmetric polynomials respectively without reference to the variables. The ring $Sym_k$ is graded by imposing $deg(x_i) = 2$.

We often work with polynomials in \(N\) variables, where \(N\) is our fixed natural number.
In this case, we write \(\scalars_N:= Sym_N\), and refer to the variables, the elementary, complete homogeneous, and power sum symmetric polynomials with capital letters: \(X_i\), \(E_i\), \(H_i\), and \(P_i\) respectively.

For $n \in \Z$, and $q$ a formal variable, let $[n] := \frac{q^n-q^{-n}}{q-q^{-1}} = \sum_{i=0}^{n-1}q^{n-1-2i}$, for $k \in \N$, we let $[k]! = \prod_{j=1}^k[j]$. Finally, for $m \in \Z$ and $a \in \N$, define: 
\[ \croch{m}{a} = \prod_{i=1}^{a}\frac{[m+1-i]}{[i]} \in \Z[q,q^{-1}].\]
Note that if $m$ is non-negative, one has $\croch{m}{a}=\frac{[m]!}{[a]![m-a]!}$.

For a \(\Z\)-graded module \(M\), let \(M_i\) denote the subspace in degree $i$. Let $q^nM$ denotes the $\Z$-graded module where $(q^nM)_i = M_{i-n}$. 

Complexes will be taken to be cohomologically graded.
In other words, the $i$th differential goes from $C_i$ to $C_{i+1}$. For a complex $C$, we let $t^iC$ denote the complex with \((t^iC)_j = C_{j-i}.\)

%% file: Z_Web_and_Foams.tex
\begin{dfn}
    A \textit{web} is a finite, oriented, trivalent graph $\web$ properly embedded in $\Disk$ together with a $\N^*$-labeling of edges satisfying a flow condition at every vertex:
    \begin{equation}
        \label{vertex1}
         \NB{\tikz[scale=0.6]{\input{\imagesfolder/sch/vertex_merge}}} \hspace{5mm} \text{and} \hspace{5mm} \NB{\tikz[scale=0.6]{\input{\imagesfolder/sch/vertex_split}}}.
    \end{equation}
    The boundary of a web is its intersection with the boundary of $\Disk$.
    Vertices on the boundary must have index 1.
    A \textit{closed web} is a web with empty boundary.
\end{dfn}

\begin{rmk}
    In web diagrams, edges of thickness~1 will be represented by an unlabeled edge (\NB{\tikz{\draw (0,0) -- (0.5,0); }}) and edges of thickness~2 will be represented by unlabeled double edges (\NB{\tikz{\draw[double] (0,0) -- (0.5,0);}}). See, for example, Figure \ref{fig:ex_web}.
\end{rmk}

The marked point and orientation of $\Disk$ determines an order on the points of $\partial\Disk$, called the \textit{oriented order}.

\begin{figure}[h!]
    \centering
    \NB{\tikz[]{\input{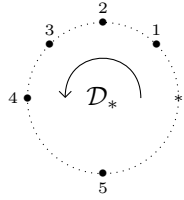}}}
    \caption{Example of the oriented order on five points of $\partial \Disk$. The numbering here refers only to the order, not to the thickness of an edge.}
\end{figure}

With this order the boundary of a web $\web$ can be described by
a finite sequence of non-zero integers,
$\bnd:=(\bnd_1,\ldots,\bnd_{n})$
of length $n$.
The integer $\bnd_i$ determines the label and the orientation of the edge of $\web$ connected to the \(i\)th boundary point.
Negative labels correspond to outgoing edges and positive labels correspond to incoming edges.
Empty sequences are allowed and are used to represent the boundary of closed webs.

\begin{figure}[h!]
    \centering
    \NB{\tikz[scale=.8]{\input{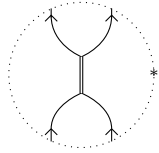}}}
    \caption{Example of a web with boundary $\bnd = (-1,-1,1,1)$.}
    \label{fig:ex_web}
\end{figure}

\begin{dfn}
The oriented $2$-dimensional CW-complexes in $\Disk \times \interval$ shown in Figure~\ref{localmod}
are called \textit{basic foams}.
The $2$-cells are smoothly embedded oriented surfaces called \textit{facets}.
Each facets $f$ is assigned a thickness $\ell(f) \in \N^*$ and
a polynomial $ Q_f \in Sym_{\ell(f)}^\scalars \otimes Sym_{N-\ell(f)}^\scalars$,
called a \textit{decoration}.
The interior $1$-cells are smoothly embedded arcs with index 3 called \textit{bindings}.
The orientations of bindings are induced by those on the thin facets and by the opposite of that on the thick facet.
The interior $0$-cells are called \textit{singular vertices}.
\end{dfn}

\begin{rmk}
     Every basic foam has a \textit{degree}, denoted $\deg_N$, given in Figure~\ref{localmod}.
\end{rmk}

\begin{figure}[ht]
   \centering
   \begin{tikzpicture}[xscale=2.4, yscale =-3.2]
     \node (pol) at (0,0) {\NB{\tikz[font=\tiny,scale=1.5]{\input{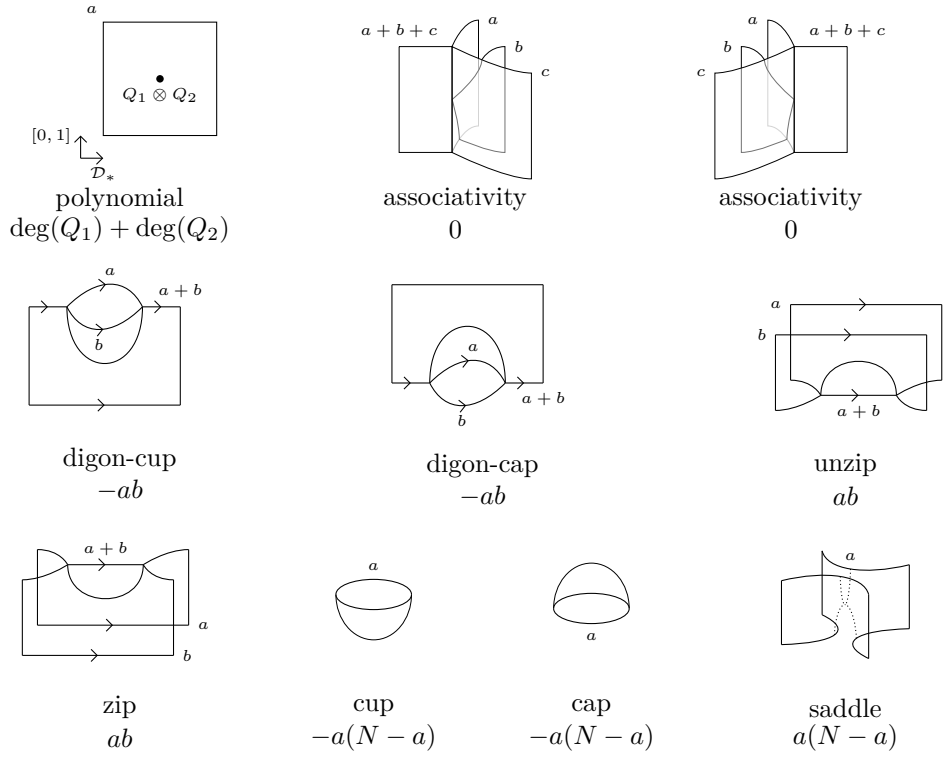}}}};
     \node[yshift= -1.4cm] at (pol) {polynomial};
     \node[yshift = -1.8cm] at (pol) {$\deg(Q_1)+\deg(Q_2)$};
     \node (asso) at(1.85,0) {\NB{\tikz[font=\tiny,scale=0.7]{\input{\imagesfolder/sch/foam_associativity}}}};
     \node[yshift = -1.4cm] at (asso) {associativity};
     \node[yshift= -1.8cm] at (asso) {$0$};
     \node (coasso) at (3.7,0) {\NB{\tikz[font=\tiny, xscale=-0.7, yscale=0.7]{\input{\imagesfolder/sch/foam_coassociativity}}}};
     \node[yshift = -1.4cm] at (coasso) {associativity};
     \node[yshift= -1.8cm] at (coasso) {$0$};
     \node (digcup) at (0,1) {\NB{\tikz[font=\tiny]{\input{\imagesfolder/sch/foam_digoncup_ab_arrows}}}};
     \node[yshift = -1.64cm] at (digcup) {digon-cup};
     \node[yshift= -2.04cm] at (digcup) {$-ab$};
     \node (digcap) at (2,1.1) {\NB{\tikz[font=\tiny]{\input{\imagesfolder/sch/foam_digoncap_ab_arrows}}}};
     \node[yshift = -1.4cm] at (digcap) {digon-cap};
     \node[yshift= -1.8cm] at (digcap) {$-ab$};
     \node (zip) at (4,1.1) {\NB{\tikz[font=\tiny]{\input{\imagesfolder/sch/foam_unzip_ab_arrows}}}};
     \node[yshift = -1.4cm] at (zip) {unzip};
     \node[yshift= -1.8cm] at (zip) {$ab$};
     \node (unzip) at (0,2.1) {\NB{\tikz[font=\tiny]{\input{\imagesfolder/sch/foam_zip_ab_arrows}}}};
     \node[yshift = -1.4cm] at (unzip) {zip};
     \node[yshift= -1.8cm] at (unzip) {$ab$};
     \node (cup) at (1.4,2.1) {\NB{\tikz[font=\tiny]{\input{\imagesfolder/sch/foam_cup_a_empty}}}};
     \node[yshift = -1.4cm] at (cup) {cup};
     \node[yshift= -1.8cm] at (cup) {$-a(N-a)$};
     \node (cap) at (2.6,2.1) {\NB{\tikz[font=\tiny]{\input{\imagesfolder/sch/foam_cap_a_empty}}}};
     \node[yshift = -1.4cm] at (cap) {cap};
     \node[yshift= -1.8cm] at (cap) {$-a(N-a)$};
     \node (saddle) at (4,2.1) {\NB{\tikz[font=\tiny]{\input{\imagesfolder/sch/foam_saddle_a_empty}}}};
     \node[yshift = -1.4cm] at (saddle) {saddle};
     \node[yshift= -1.8cm] at (saddle) {$a(N-a)$};
   \end{tikzpicture}
   \caption{The degree of a basic foam is given below its name.}\label{localmod}
 \end{figure}

\begin{dfn}
    Let $\web$ be a web and let $h \colon \Disk \times \interval \rightarrow \Disk$ be a smooth, boundary-preserving isotopy of $\Disk$.
    The subspace $F(h)_t := h(\web,t)$ in $\Disk\times \interval$ is called a \textit{trace of an isotopy} of $\web$.
    The $\N^*$-labeling and the orientation of the web define the thicknesses and orientations for the facets and bindings of \(F(h)\).
    The degree of the trace of any isotopy is $0$.
\end{dfn}

\begin{dfn}
    A \textit{foam} $\foam \subset \Disk \times \interval$ is a finite collection of basic foams and traces of isotopies glued together along their boundaries with common labels. The intersection of the foam with $\partial \Disk \times \interval$ is a collection of vertical arcs.
    Decorations are multiplied and degrees are added under gluing.
\end{dfn}

The boundary of a foam $\foam \subset \Disk \times \interval$ contains
a web in $\Disk \times \{0\}$
and a web in $\Disk \times \{1\}$.
A foam can be viewed as a ``singular cobordism'' between two webs
and the following family of categories is defined:

\begin{dfn}
For a given boundary data $\bnd$, let $\widehat{\catf}(\bnd)$ be the category whose objects are webs in $\Disk$ with boundary $\bnd$
and whose morphisms are foams in \(\Disk \times \interval\) modulo ambient, boundary-preserving isotopy.
Composition is given by stacking and identities are given by cylinders on webs.
The category \(\widehat{\catf}(\bnd)\) is formally extended to the category \(\catf(\bnd)\) by taking an additive, \(\scalars_N\)-linear closure, adding formal \(\Z\)-grading shifts and imposing the local foam relations defined by Robert--Wagner evaluation \cite{RW20a}.
\end{dfn}

% defined by:
% \[
% \text{Hom}(\web_0, \web_1) =  \sfrac{\{\text{foams} \foam { in } \Disk \times \interval \text{ with } \foam_0 = -\web_0 \text{ and } \foam_1 = \web_1 \}}{\sim} ,
% \]
% where $-\web_0$ is the same web as $\web_0$ but with all edges oriented in the opposite direction. The "$\sim$" means up to ambient isotopy that preserve the boundary of $\Disk \times \interval$ and local relations are defined by the Robert--Wagner evaluation. We choose not to discuss these relations here and refer the curious reader to \cite{RW20a}.

\begin{notation}
    For two webs $\web$ and $\web'$ with the same boundary, we denote by $\morf(\web,\web')$ the space of morphisms from $\web$ to $\web'$, and by $\morf(\web)$ the endomorphism of $\web$ in $\catf(\bnd)$.
\end{notation}

% A power sum decoration on a facet of thickness \(a\) will be denoted by
% \[
% \ttkz[scale=.8]{sch/notation_deco2}:= \ttkz[scale=.8]{sch/notation_deco1} \hspace{5mm} \text{and} \hspace{5mm} \ttkz[scale=.8]{sch/notation_deco3}:= \ttkz[scale=.8]{sch/notation_antideco}
% \]
% where \(p_i = p_i \otimes 1 \in Sym_a^\scalars \otimes Sym_{N-a}^\scalars\) and \(\widehat{p}_i = 1 \otimes p_i \in Sym_a^\scalars \otimes Sym_{N-a}^\scalars\).
% In particular, \(\deco_0=a\) and \(\antideco_0 = N-a\).

%% file: sch/vertex_merge.tex
\begin{scope}[font=\tiny]
    \draw [->-] (0,0) -- (90:1) node[pos =1, above] {$a+b$};
    \draw [-<-] (0,0) -- (-30:1) node[pos =1, below] {$b$};
    \draw [-<-] (0,0) -- (-150:1) node[pos =1, below] {$a$};
    \draw[dotted] (0,0) circle (1); 
    \draw (1,0) node[font=\tiny] {$*$};
\end{scope}

%% file: sch/vertex_split.tex
\begin{scope}[font=\tiny]
    \draw [-<-] (0,0) -- (-90:1) node[pos =1, below] {$a+b$};
    \draw [->-] (0,0) -- (30:1) node[pos =1, above] {$b$};
    \draw [->-] (0,0) -- (150:1) node[pos =1, above] {$a$};
    \draw[dotted] (0,0) circle (1);
    \draw (1,0) node[font=\tiny] {$*$};
\end{scope}

%% file: Z_Witt.tex
\begin{dfn}
    Let $\Wittt$ be the Lie algebra over $\scalars$ generated by symbols $(\wgen_n)_{n \in \Z}$ with bracket
    \begin{equation} \label{witt} [\wgen_n,\wgen_m] := (n-m) \wgen_{n+m}. \end{equation}
    This algebra is called \textit{Witt algebra}. 
    Let $\Witt$ denote the Lie subalgebra generated by symbols $(\wgen_n)_{n \geqslant -1}$. 
\end{dfn}
The Lie algebra $\Witt$ acts on polynomial rings.
For any $Q \in \scalars[x_1,\ldots,x_k]$, let
\begin{equation}
    \wgen_n\cdot Q := -\sum_{i=1}^k x_i^{n+1}\frac{\partial Q}{\partial x_i}
\end{equation}
Note that the subspace $\scalars[x_1,\ldots,x_k]^{S_k}$ is fixed by this action.
% The map
% \begin{equation} 
%     i : \left | \begin{array}{lll}
%     e &\mapsto \wgen_{-1} \\
%     h &\mapsto 2\wgen_{0} \\
%     f &\mapsto -\wgen_1 
% \end{array} \right. 
% \end{equation}
% defines an injective morphism of Lie algebras from $\sld$ to $\Witt$.

%% file: Z_action.tex
\label{sec_action}

Fix $s, \lambda \in \scalars$ and let
$\lambda_n := \lambda(n+1)$. % \footnote{In \cite{QRSW24}, authors use the terminology of Witt sequences. When $2$ is invertible in the ground ring, these sequences are of the form $(\lambda(n+1))_{n\in \N_{-1}}$ with $\lambda \in \scalars$.}.
Using these distinguished elements, the following operators $\Wgen_n$ are defined on basic foams:

\begin{align}
    \label{polyn}
    \Wgen_n \left (~\NB{\tikz[scale=1.7, font=\tiny]{\input{\imagesfolder/sch/foam_poly}}}~\right) :=  \NB{\tikz[scale=1.7, font=\tiny]{\input{\imagesfolder/sch/foam_polynomial_L}}}
\end{align}
\begin{align}
    \Wgen_n \left (  \NB{\tikz[xscale=.9,yscale=.7,font=\tiny]{\input{\imagesfolder/sch/foam_associativity}}} \right ) := \Wgen_n \left (  \NB{\tikz[xscale=-.9, yscale=.7,font=\tiny]{\input{\imagesfolder/sch/foam_coassociativity}}} \right ) := 0
\end{align}
\begin{align}
    \Wgen_n \left ( ~~ \NB{\tikz[scale=1.2,font=\tiny]{\input{\imagesfolder/sch/foam_digoncup_ab_arrows}}}  \right ) := ~~ &\lambda_n ~~ \NB{\tikz[scale=1.2,font=\tiny]{\input{\imagesfolder/sch/foam_digoncup_ab_0n}}}+ \lambda_n ~~ \NB{\tikz[scale=1.2,font=\tiny]{\input{\imagesfolder/sch/foam_digoncup_ab_n0}}} \\ \nonumber &+ s\sum_{k+\ell=n} \NB{\tikz[scale=1.2,font=\tiny]{\input{\imagesfolder/sch/foam_digoncup_ab_kl}}} 
\end{align}
\begin{align}
    \Wgen_n \left (~~ \NB{\tikz[scale=1.2,font=\tiny]{\input{\imagesfolder/sch/foam_digoncap_ab_arrows}}} \right ) := ~~ & -\lambda_n ~~ \NB{\tikz[scale=1.2,font=\tiny]{\input{\imagesfolder/sch/foam_digoncap_ab_0n}}} - \lambda_n ~~ \NB{\tikz[scale=1.2,font=\tiny]{\input{\imagesfolder/sch/foam_digoncap_ab_n0}}} \\ \nonumber &+ \bar{s}\sum_{k+\ell=n} \NB{\tikz[scale=1.2,font=\tiny]{\input{\imagesfolder/sch/foam_digoncap_ab_kl}}} 
\end{align}
\begin{align}
    \Wgen_n \left ( ~~ \NB{\tikz[scale=1.2,font=\tiny]{\input{\imagesfolder/sch/foam_zip_ab_arrows}}}  \right ) := ~~ &\lambda_n ~~ \NB{\tikz[scale=1.2,font=\tiny]{\input{\imagesfolder/sch/foam_zip_ab_0n}}}+ \lambda_n ~~ \NB{\tikz[scale=1.2,font=\tiny]{\input{\imagesfolder/sch/foam_zip_ab_n0}}} \\ \nonumber &- \bar{s}\sum_{k+\ell=n} \NB{\tikz[scale=1.2,font=\tiny]{\input{\imagesfolder/sch/foam_zip_ab_kl}}}
\end{align}
\begin{align}
    \Wgen_n \left ( \NB{\tikz[scale=1.2,font=\tiny]{\input{\imagesfolder/sch/foam_unzip_ab_arrows}}} ~~ \right ) := ~~ & -\lambda_n  \NB{\tikz[scale=1.2,font=\tiny]{\input{\imagesfolder/sch/foam_unzip_ab_0n}}} - \lambda_n  \NB{\tikz[scale=1.2,font=\tiny]{\input{\imagesfolder/sch/foam_unzip_ab_n0}}} \\ \nonumber &-s\sum_{k+\ell=n} \NB{\tikz[scale=1.2,font=\tiny]{\input{\imagesfolder/sch/foam_unzip_ab_kl}}}
\end{align}
\begin{align}
    \Wgen_n \left ( \NB{\tikz[scale=1.5,font=\tiny]{\input{\imagesfolder/sch/foam_cup_a_empty}}} \right ) := \frac{1}{2} \sum_{k+\ell=n} \NB{\tikz[scale=1.5,font=\tiny]{\input{\imagesfolder/sch/foam_cup_a_kl}}} \hspace{1.2cm} \Wgen_n \left ( \NB{\tikz[scale=1.5,font=\tiny]{\input{\imagesfolder/sch/foam_cap_a_empty}}} \right ) := \frac{1}{2} \sum_{k+\ell=n} \NB{\tikz[scale=1.5,font=\tiny]{\input{\imagesfolder/sch/foam_cap_a_kl}}}
\end{align}

\begin{align}
    \label{saddle}
    \Wgen_n \left ( \NB{\tikz[scale=1.3,font=\tiny]{\input{\imagesfolder/sch/foam_saddle_a_empty}}} \right ) := - \frac{1}{2} \sum_{k+\ell=n}~~ \NB{\tikz[scale=1.3,font=\tiny]{\input{\imagesfolder/sch/foam_saddle_a_kl}}}
\end{align}

These operators are defined to be 0 on traces of isotopy and extended to gluings by the Leibniz rule.

\begin{prop}[{\cite[Prop 4.3]{QRSW23}}]
    Let $\foam$ be a closed foam, then for all $n \in \N_{-1}$
    \begin{equation}
        \langle \Wgen_n \cdot \foam \rangle_N = \wgen_n \cdot \langle \foam \rangle_N. 
    \end{equation}
\end{prop}

\begin{prop}[{\cite[Prop 3.6]{GR25}}]
    \label{action}
    Let $\web, \web'$ be two webs with the same boundary.
    The operators $\Wgen_n$ define a $\Witt$-action on the $\scalars_N$-module $\morf(\web,\web')$.
    Moreover, for any foam $\foam$:
    \[\deg_N(\wgen_n \cdot \foam) = 2n + \deg_N(\foam).\]
\end{prop}

\begin{rmk}
In \cite{QRSW23} and \cite{GR25}, the action of $\Witt$ on foams depends on three parameters $(\lambda,\mu,s)$. Here we set $\lambda = \mu$ for reasons detailed in Remark \ref{parameters}.
\end{rmk}

%% file: sch/foam_poly.tex
\begin{scope}
  \draw (0,0) rectangle (1,1) coordinate [midway] (A);
  \fill (A) circle (0.33mm) node[below] {$P_1 \otimes P_2$};
  \node[above] at (.2,.7) {$a$};
\end{scope}

%% file: sch/foam_polynomial_L.tex
\begin{scope}
  \draw (0,0) rectangle (1,1) coordinate [midway] (A);
  \fill (A) circle (0.29mm) node[below] {$\wgen_n\cdot (P_1 \otimes P_2)$};
  \node[above] at (.2,.7) {$a$};
\end{scope}

%% file: sch/foam_associativity.tex
\begin{scope}
  \coordinate (ABCt) at (0,2);
  \coordinate (Ot) at (1,2);
  \coordinate (At) at (1.5,2.5);
  \coordinate (Bt) at (2,2);
  \coordinate (Ct) at (2.5,1.5);
  \begin{scope}[yshift= -2cm]
    \coordinate (ABCb) at (0,2);
    \coordinate (Ob) at (1,2);
    \coordinate (Ab) at (1.5,2.5);
    \coordinate (Bb) at (2,2);
    \coordinate (Cb) at (2.5,1.5);
  \end{scope}
  \begin{scope}[fill = white, fill opacity = 0.45, draw = black,
    draw opacity =1, very thin]
    \filldraw (Ob) .. controls +(0,0) and +(-0.3, 0) .. (Ab)
    coordinate[pos=0.5] (ABb) -- (At)
    .. controls +(-0.3, 0) and +(0,0) .. (Ot) -- (Ob);
    \fill (Ob) .. controls +(0,0) and +(-0.5, 0) .. (Cb) -- (Ct)
    .. controls +(-0.5, 0) and +(0,0) .. (Ot) coordinate[pos=0.5]
    (BCt) -- (Ob) coordinate[pos=0.5] (Om);
    \filldraw (ABb) .. controls +(0,0) and +(-0.3, 0) .. (Bb)-- (Bt)
    .. controls +(-0.3, 0) and +(0,0) .. (BCt) .. controls +(0, -0.3) and
    +(0,0) .. (Om) .. controls +(0,0) and +(0, 0.3)  .. (ABb);
    \filldraw (Ob) .. controls +(0,0) and +(-0.5, 0) .. (Cb) -- (Ct)
    .. controls +(-0.5, 0) and +(0,0) .. (Ot) -- (Ob);
    \filldraw (ABCt) -- (Ot) -- (Ob) -- (ABCb) -- (ABCt);
  \end{scope}
  
  \begin{scope}
    \node[above] at (ABCt) {$a+b+c$};
    \node[right] at (At) {$a$};
    \node[right] at (Bt) {$b$};
    \node[right] at (Ct) {$c$};
  \end{scope}
\end{scope}

%% file: sch/foam_coassociativity.tex
\begin{scope}
  \coordinate (ABCt) at (0,2);
  \coordinate (Ot) at (1,2);
  \coordinate (At) at (1.5,2.5);
  \coordinate (Bt) at (2,2);
  \coordinate (Ct) at (2.5,1.5);
  \begin{scope}[yshift= -2cm]
    \coordinate (ABCb) at (0,2);
    \coordinate (Ob) at (1,2);
    \coordinate (Ab) at (1.5,2.5);
    \coordinate (Bb) at (2,2);
    \coordinate (Cb) at (2.5,1.5);
  \end{scope}
  \begin{scope}[fill = white, fill opacity = 0.45, draw = black,
    draw opacity =1, very thin]
    \filldraw (Ob) .. controls +(0,0) and +(-0.3, 0) .. (Ab)
    coordinate[pos=0.5] (ABb) -- (At)
    .. controls +(-0.3, 0) and +(0,0) .. (Ot) -- (Ob);
    \fill (Ob) .. controls +(0,0) and +(-0.5, 0) .. (Cb) -- (Ct)
    .. controls +(-0.5, 0) and +(0,0) .. (Ot) coordinate[pos=0.5]
    (BCt) -- (Ob) coordinate[pos=0.5] (Om);
    \filldraw (ABb) .. controls +(0,0) and +(-0.3, 0) .. (Bb)-- (Bt)
    .. controls +(-0.3, 0) and +(0,0) .. (BCt) .. controls +(0, -0.3) and
    +(0,0) .. (Om) .. controls +(0,0) and +(0, 0.3)  .. (ABb);
    \filldraw (Ob) .. controls +(0,0) and +(-0.5, 0) .. (Cb) -- (Ct)
    .. controls +(-0.5, 0) and +(0,0) .. (Ot) -- (Ob);
    \filldraw (ABCt) -- (Ot) -- (Ob) -- (ABCb) -- (ABCt);
  \end{scope}
  
  \begin{scope}
    \node[above] at (ABCt) {$a+b+c$};
    \node[left] at (At) {$a$};
    \node[left] at (Bt) {$b$};
    \node[left] at (Ct) {$c$};
  \end{scope}
\end{scope}

%% file: sch/foam_digoncup_ab_arrows.tex
\begin{scope}
  \begin{scope}
    \coordinate (L) at (0,0);
    \coordinate (R) at (2,0);
    \coordinate (ML) at (0.5, 0);
    \coordinate (MR) at (1.5, 0);
    \draw[->-] (L) -- (ML);
    \draw[->-] (MR) -- (R) node[above] {$a+b$};
    \draw[->-] (ML).. controls + (0.4, 0.4) and +(-0.2, 0.4) .. (MR)
    node[above, midway] {$a$};
    \draw[->-] (ML).. controls + (0.2, -0.4) and +(-0.4, -0.4) .. (MR)
    node[below, midway] {$b$};
  \end{scope}  
 \begin{scope}[yshift = -1.3cm]
    \coordinate (LB) at (0,0);
    \coordinate (RB) at (2,0);
    \draw[->-] (LB) -- (RB);
  \end{scope}  
  \draw (R) -- (RB);
  \draw (L) -- (LB);
  \draw (ML) .. controls +(0, -1) and +(0, -1) .. (MR);
\end{scope}

%% file: sch/foam_digoncup_ab_0n.tex
\begin{scope}
  \begin{scope}
    \coordinate (L) at (0,0);
    \coordinate (R) at (2,0);
    \coordinate (ML) at (0.5, 0);
    \coordinate (MR) at (1.5, 0);
    \draw[->-] (L) -- (ML);
    \draw[->-] (MR) -- (R) node[above] {$a+b$};
    \draw[->-] (ML).. controls + (0.4, 0.4) and +(-0.2, 0.4) .. (MR)
    node[above, midway] {$a$} node[pos=0.4,below] {$\deco_0$};
    \draw[->-] (ML).. controls + (0.2, -0.4) and +(-0.4, -0.4) .. (MR)
    node[below, pos=0.35] {$b$} node[pos=0.75,below] {$\deco_n$};
  \end{scope}  
 \begin{scope}[yshift = -1.3cm]
    \coordinate (LB) at (0,0);
    \coordinate (RB) at (2,0);
    \draw[->-] (LB) -- (RB);
  \end{scope}  
  \draw (R) -- (RB);
  \draw (L) -- (LB);
  \draw(ML) .. controls +(0, -1) and +(0, -1) .. (MR);
\end{scope}

%% file: sch/foam_digoncup_ab_n0.tex
\begin{scope}
  \begin{scope}
    \coordinate (L) at (0,0);
    \coordinate (R) at (2,0);
    \coordinate (ML) at (0.5, 0);
    \coordinate (MR) at (1.5, 0);
    \draw[->-] (L) -- (ML);
    \draw[->-] (MR) -- (R) node[above] {$a+b$};
    \draw[->-] (ML).. controls + (0.4, 0.4) and +(-0.2, 0.4) .. (MR)
    node[above, midway] {$a$} node[pos=0.4,below] {$\deco_n$};
    \draw[->-] (ML).. controls + (0.2, -0.4) and +(-0.4, -0.4) .. (MR)
    node[below, pos=0.35] {$b$} node[pos=0.75,below] {$\deco_0$};
  \end{scope}  
 \begin{scope}[yshift = -1.3cm]
    \coordinate (LB) at (0,0);
    \coordinate (RB) at (2,0);
    \draw[->-] (LB) -- (RB);
  \end{scope}  
  \draw (R) -- (RB);
  \draw (L) -- (LB);
  \draw (ML) .. controls +(0, -1) and +(0, -1) .. (MR);
\end{scope}

%% file: sch/foam_digoncup_ab_kl.tex
\begin{scope}
  \begin{scope}
    \coordinate (L) at (0,0);
    \coordinate (R) at (2,0);
    \coordinate (ML) at (0.5, 0);
    \coordinate (MR) at (1.5, 0);
    \draw[->-] (L) -- (ML);
    \draw[->-] (MR) -- (R) node[right] {$a+b$};
    \draw[->-] (ML).. controls + (0.4, 0.4) and +(-0.2, 0.4) .. (MR)
    node[above, midway] {$a$} node[pos=0.4,below] {$\deco_k$};
    \draw[->-] (ML).. controls + (0.2, -0.4) and +(-0.4, -0.4) .. (MR)
    node[below, pos=0.35] {$b$} node[pos=0.75,below] {$\deco_l$};
  \end{scope}  
 \begin{scope}[yshift = -1.3cm]
    \coordinate (LB) at (0,0);
    \coordinate (RB) at (2,0);
    \draw[->-] (LB) -- (RB);
  \end{scope}  
  \draw (R) -- (RB);
  \draw (L) -- (LB);
  \draw (ML) .. controls +(0, -1) and +(0, -1) .. (MR);
\end{scope}

%% file: sch/foam_digoncap_ab_arrows.tex
\begin{scope}
  \begin{scope}
    \coordinate (L) at (0,0);
    \coordinate (R) at (2,0);
    \coordinate (ML) at (0.5, 0);
    \coordinate (MR) at (1.5, 0);
    \draw[->-] (L) -- (ML);
    \draw[->-] (MR) -- (R) node[below] {$a+b$};
    \draw[->-] (ML).. controls + (0.4, 0.4) and +(-0.2, 0.4) .. (MR)
    node[above, midway] {$a$};
    \draw[->-] (ML).. controls + (0.2, -0.4) and +(-0.4, -0.4) .. (MR)
    node[below, midway] {$b$};
  \end{scope}  
 \begin{scope}[yshift = 1.3cm]
    \coordinate (LB) at (0,0);
    \coordinate (RB) at (2,0);
    \draw (LB) -- (RB);
  \end{scope}  
  \draw (R) -- (RB);
  \draw (L) -- (LB);
  \draw (ML) .. controls +(0, 1) and +(0, 1) .. (MR);
\end{scope}

%% file: sch/foam_digoncap_ab_0n.tex
 \begin{scope}
  \begin{scope}
    \coordinate (L) at (0,0);
    \coordinate (R) at (2,0);
    \coordinate (ML) at (0.5, 0);
    \coordinate (MR) at (1.5, 0);
    \draw[->-] (L) -- (ML);
    \draw[->-] (MR) -- (R) node[below] {$a+b$};
    \draw[->-] (ML).. controls + (0.4, 0.4) and +(-0.2, 0.4) .. (MR)
    node[above, pos=0.2] {$a$} node[pos=0.6,above] {$\deco_0$};
    \draw[->-] (ML).. controls + (0.2, -0.4) and +(-0.4, -0.4) .. (MR)
    node[below, pos=0.5] {$b$} node[pos=0.6, above] {$\deco_n$};
  \end{scope}  
 \begin{scope}[yshift = 1.3cm]
    \coordinate (LB) at (0,0);
    \coordinate (RB) at (2,0);
    \draw (LB) -- (RB);
  \end{scope}  
  \draw (R) -- (RB);
  \draw (L) -- (LB);
  \draw (ML) .. controls +(0, 1) and +(0, 1) .. (MR);
\end{scope}

%% file: sch/foam_digoncap_ab_n0.tex
 \begin{scope}
  \begin{scope}
    \coordinate (L) at (0,0);
    \coordinate (R) at (2,0);
    \coordinate (ML) at (0.5, 0);
    \coordinate (MR) at (1.5, 0);
    \draw[->-] (L) -- (ML);
    \draw[->-] (MR) -- (R) node[below] {$a+b$};
    \draw[->-] (ML).. controls + (0.4, 0.4) and +(-0.2, 0.4) .. (MR)
    node[above, pos=0.2] {$a$} node[pos=0.6,above] {$\deco_n$};
    \draw[->-] (ML).. controls + (0.2, -0.4) and +(-0.4, -0.4) .. (MR)
    node[below, pos=0.5] {$b$} node[pos=0.6, above] {$\deco_0$};
  \end{scope}  
 \begin{scope}[yshift = 1.3cm]
    \coordinate (LB) at (0,0);
    \coordinate (RB) at (2,0);
    \draw (LB) -- (RB);
  \end{scope}  
  \draw (R) -- (RB);
  \draw (L) -- (LB);
  \draw (ML) .. controls +(0, 1) and +(0, 1) .. (MR);
\end{scope}

%% file: sch/foam_digoncap_ab_kl.tex
 \begin{scope}
  \begin{scope}
    \coordinate (L) at (0,0);
    \coordinate (R) at (2,0);
    \coordinate (ML) at (0.5, 0);
    \coordinate (MR) at (1.5, 0);
    \draw[->-] (L) -- (ML);
    \draw[->-] (MR) -- (R) node[below] {$a+b$};
    \draw[->-] (ML).. controls + (0.4, 0.4) and +(-0.2, 0.4) .. (MR)
    node[above, pos=0.2] {$a$} node[pos=0.6,above] {$\deco_k$};
    \draw[->-] (ML).. controls + (0.2, -0.4) and +(-0.4, -0.4) .. (MR)
    node[below, pos=0.5] {$b$} node[pos=0.6, above] {$\deco_l$};
  \end{scope}  
 \begin{scope}[yshift = 1.3cm]
    \coordinate (LB) at (0,0);
    \coordinate (RB) at (2,0);
    \draw (LB) -- (RB);
  \end{scope}  
  \draw (R) -- (RB);
  \draw (L) -- (LB);
  \draw (ML) .. controls +(0, 1) and +(0, 1) .. (MR);
\end{scope}

%% file: sch/foam_zip_ab_arrows.tex
\begin{scope}
  \begin{scope}
    \coordinate (L1) at (0.2,0.4);
    \coordinate (L2) at (0,0);
    \coordinate (R1) at (2.2,0.4);
    \coordinate (R2) at (2,0);
    \coordinate (ML) at (0.6, 0.2);
    \coordinate (MR) at (1.6, 0.2);
    \draw[->-] (ML) -- (MR) node[above, midway] {$a+b$};
    \draw (MR) .. controls +(0, 0) and +(-0.3,0) .. (R1) ;
    \draw (MR) .. controls +(0, 0) and +(-0.3,0) .. (R2);
    \draw (L1) .. controls +( 0.3, 0) and +(0,0) .. (ML);
    \draw (L2) .. controls +( 0.3, 0) and +(0,0) .. (ML);
  \end{scope}  
 \begin{scope}[yshift = -1cm]
    \coordinate (L1B) at (0.2,0.4);
    \coordinate (L2B) at (0,0);
    \coordinate (R1B) at (2.2,0.4);
    \coordinate (R2B) at (2,0);
    \draw[->-] (L1B) .. controls +( 0, 0) and +(0,0) .. (R1B) node [right, pos
    = 1] {$a$};
    \draw[->-] (L2B) .. controls +( 0, 0) and +(0,0) .. (R2B) node [right, pos
    = 1] {$b$};
 \end{scope}  
  \draw (R1) -- (R1B);
  \draw (R2) -- (R2B);
  \draw (L1) -- (L1B);
  \draw (L2) -- (L2B);
  \draw (ML) .. controls +(0, -0.6) and +(0, -0.6) .. (MR);
\end{scope}

%% file: sch/foam_zip_ab_0n.tex
\begin{scope}
  \begin{scope}
    \coordinate (L1) at (0.2,0.4);
    \coordinate (L2) at (0,0);
    \coordinate (R1) at (2.2,0.4);
    \coordinate (R2) at (2,0);
    \coordinate (ML) at (0.6, 0.2);
    \coordinate (MR) at (1.6, 0.2);
    \draw[->-] (ML) -- (MR) node[above, midway] {$a+b$};
    \draw (MR) .. controls +(0, 0) and +(-0.3,0) .. (R1) ;
    \draw (MR) .. controls +(0, 0) and +(-0.3,0) .. (R2);
    \draw (L1) .. controls +( 0.3, 0) and +(0,0) .. (ML);
    \draw (L2) .. controls +( 0.3, 0) and +(0,0) .. (ML);
  \end{scope}  
 \begin{scope}[yshift = -1cm]
    \coordinate (L1B) at (0.2,0.4);
    \coordinate (L2B) at (0,0);
    \coordinate (R1B) at (2.2,0.4);
    \coordinate (R2B) at (2,0);
    \draw[->-] (L1B) .. controls +( 0, 0) and +(0,0) .. (R1B) node [right, pos = 1] {$a$} node[above,pos=0.6] {$\deco_0$};
    \draw[->-] (L2B) .. controls +( 0, 0) and +(0,0) .. (R2B) node [right, pos = 1] {$b$} node[above,pos=0.4] {$\deco_n$};
 \end{scope}  
  \draw (R1) -- (R1B);
  \draw (R2) -- (R2B);
  \draw (L1) -- (L1B);
  \draw (L2) -- (L2B);
  \draw (ML) .. controls +(0, -0.6) and +(0, -0.6) .. (MR);
\end{scope}

%% file: sch/foam_zip_ab_n0.tex
\begin{scope}
  \begin{scope}
    \coordinate (L1) at (0.2,0.4);
    \coordinate (L2) at (0,0);
    \coordinate (R1) at (2.2,0.4);
    \coordinate (R2) at (2,0);
    \coordinate (ML) at (0.6, 0.2);
    \coordinate (MR) at (1.6, 0.2);
    \draw[->-] (ML) -- (MR) node[above, midway] {$a+b$};
    \draw (MR) .. controls +(0, 0) and +(-0.3,0) .. (R1) ;
    \draw (MR) .. controls +(0, 0) and +(-0.3,0) .. (R2);
    \draw (L1) .. controls +( 0.3, 0) and +(0,0) .. (ML);
    \draw (L2) .. controls +( 0.3, 0) and +(0,0) .. (ML);
  \end{scope}  
 \begin{scope}[yshift = -1cm]
    \coordinate (L1B) at (0.2,0.4);
    \coordinate (L2B) at (0,0);
    \coordinate (R1B) at (2.2,0.4);
    \coordinate (R2B) at (2,0);
    \draw[->-] (L1B) .. controls +( 0, 0) and +(0,0) .. (R1B) node [right, pos = 1] {$a$} node[above,pos=0.6] {$\deco_n$};
    \draw[->-] (L2B) .. controls +( 0, 0) and +(0,0) .. (R2B) node [right, pos = 1] {$b$} node[above,pos=0.4] {$\deco_0$};
 \end{scope}  
  \draw (R1) -- (R1B);
  \draw (R2) -- (R2B);
  \draw (L1) -- (L1B);
  \draw (L2) -- (L2B);
  \draw (ML) .. controls +(0, -0.6) and +(0, -0.6) .. (MR);
\end{scope}

%% file: sch/foam_zip_ab_kl.tex
\begin{scope}
  \begin{scope}
    \coordinate (L1) at (0.2,0.4);
    \coordinate (L2) at (0,0);
    \coordinate (R1) at (2.2,0.4);
    \coordinate (R2) at (2,0);
    \coordinate (ML) at (0.6, 0.2);
    \coordinate (MR) at (1.6, 0.2);
    \draw[->-] (ML) -- (MR) node[above, midway] {$a+b$};
    \draw (MR) .. controls +(0, 0) and +(-0.3,0) .. (R1) ;
    \draw (MR) .. controls +(0, 0) and +(-0.3,0) .. (R2);
    \draw (L1) .. controls +( 0.3, 0) and +(0,0) .. (ML);
    \draw (L2) .. controls +( 0.3, 0) and +(0,0) .. (ML);
  \end{scope}  
 \begin{scope}[yshift = -1cm]
    \coordinate (L1B) at (0.2,0.4);
    \coordinate (L2B) at (0,0);
    \coordinate (R1B) at (2.2,0.4);
    \coordinate (R2B) at (2,0);
    \draw[->-] (L1B) .. controls +( 0, 0) and +(0,0) .. (R1B) node [right, pos = 1] {$a$} node[above,pos=0.6] {$\deco_k$};
    \draw[->-] (L2B) .. controls +( 0, 0) and +(0,0) .. (R2B) node [right, pos = 1] {$b$} node[above,pos=0.4] {$\deco_l$};
 \end{scope}  
  \draw (R1) -- (R1B);
  \draw (R2) -- (R2B);
  \draw (L1) -- (L1B);
  \draw (L2) -- (L2B);
  \draw(ML) .. controls +(0, -0.6) and +(0, -0.6) .. (MR);
\end{scope}

%% file: sch/foam_unzip_ab_arrows.tex
\begin{scope}
  \begin{scope}
    \coordinate (L1) at (0.2,0.4);
    \coordinate (L2) at (0,0);
    \coordinate (R1) at (2.2,0.4);
    \coordinate (R2) at (2,0);
    \coordinate (ML) at (0.6, 0.2);
    \coordinate (MR) at (1.6, 0.2);
    \draw[->-] (ML) -- (MR) node[below, midway] {$a+b$};
    \draw (MR) .. controls +(0, 0) and +(-0.3,0) .. (R1) ;
    \draw (MR) .. controls +(0, 0) and +(-0.3,0) .. (R2);
    \draw (L1) .. controls +( 0.3, 0) and +(0,0) .. (ML);
    \draw (L2) .. controls +( 0.3, 0) and +(0,0) .. (ML);
  \end{scope}  
 \begin{scope}[yshift = 1cm]
    \coordinate (L1B) at (0.2,0.4);
    \coordinate (L2B) at (0,0);
    \coordinate (R1B) at (2.2,0.4);
    \coordinate (R2B) at (2,0);
    \draw[->-] (L1B) .. controls +( 0, 0) and +(0,0) .. (R1B) node [left, pos
    = 0] {$a$};
    \draw[->-] (L2B) .. controls +( 0, 0) and +(0,0) .. (R2B) node [left, pos
    = 0] {$b$};
 \end{scope}  
  \draw (R1) -- (R1B);
  \draw (R2) -- (R2B);
  \draw (L1) -- (L1B);
  \draw (L2) -- (L2B);
  \draw (ML) .. controls +(0, 0.6) and +(0, 0.6) .. (MR);
\end{scope}

%% file: sch/foam_unzip_ab_0n.tex
\begin{scope}
  \begin{scope}
    \coordinate (L1) at (0.2,0.4);
    \coordinate (L2) at (0,0);
    \coordinate (R1) at (2.2,0.4);
    \coordinate (R2) at (2,0);
    \coordinate (ML) at (0.6, 0.2);
    \coordinate (MR) at (1.6, 0.2);
    \draw[->-] (ML) -- (MR) node[below, midway] {$a+b$};
    \draw (MR) .. controls +(0, 0) and +(-0.3,0) .. (R1) ;
    \draw (MR) .. controls +(0, 0) and +(-0.3,0) .. (R2);
    \draw (L1) .. controls +( 0.3, 0) and +(0,0) .. (ML);
    \draw (L2) .. controls +( 0.3, 0) and +(0,0) .. (ML);
  \end{scope}  
 \begin{scope}[yshift = 1cm]
    \coordinate (L1B) at (0.2,0.4);
    \coordinate (L2B) at (0,0);
    \coordinate (R1B) at (2.2,0.4);
    \coordinate (R2B) at (2,0);
    \draw[->-] (L1B) .. controls +( 0, 0) and +(0,0) .. (R1B) node [left, pos
    = 0] {$a$} node[pos=0.75,below] {$\deco_0$};
    \draw[->-] (L2B) .. controls +( 0, 0) and +(0,0) .. (R2B) node [left, pos
    = 0] {$b$} node[pos=0.35,below] {$\deco_n$};
 \end{scope}  
  \draw (R1) -- (R1B);
  \draw (R2) -- (R2B);
  \draw (L1) -- (L1B);
  \draw (L2) -- (L2B);
  \draw (ML) .. controls +(0, 0.6) and +(0, 0.6) .. (MR);
\end{scope}

%% file: sch/foam_unzip_ab_n0.tex
\begin{scope}
  \begin{scope}
    \coordinate (L1) at (0.2,0.4);
    \coordinate (L2) at (0,0);
    \coordinate (R1) at (2.2,0.4);
    \coordinate (R2) at (2,0);
    \coordinate (ML) at (0.6, 0.2);
    \coordinate (MR) at (1.6, 0.2);
    \draw[->-] (ML) -- (MR) node[below, midway] {$a+b$};
    \draw (MR) .. controls +(0, 0) and +(-0.3,0) .. (R1) ;
    \draw (MR) .. controls +(0, 0) and +(-0.3,0) .. (R2);
    \draw (L1) .. controls +( 0.3, 0) and +(0,0) .. (ML);
    \draw (L2) .. controls +( 0.3, 0) and +(0,0) .. (ML);
  \end{scope}  
 \begin{scope}[yshift = 1cm]
    \coordinate (L1B) at (0.2,0.4);
    \coordinate (L2B) at (0,0);
    \coordinate (R1B) at (2.2,0.4);
    \coordinate (R2B) at (2,0);
    \draw[->-] (L1B) .. controls +( 0, 0) and +(0,0) .. (R1B) node [left, pos
    = 0] {$a$} node[pos=0.75,below] {$\deco_n$};
    \draw[->-] (L2B) .. controls +( 0, 0) and +(0,0) .. (R2B) node [left, pos
    = 0] {$b$} node[pos=0.35,below] {$\deco_0$};
 \end{scope}  
  \draw (R1) -- (R1B);
  \draw (R2) -- (R2B);
  \draw (L1) -- (L1B);
  \draw (L2) -- (L2B);
  \draw (ML) .. controls +(0, 0.6) and +(0, 0.6) .. (MR);
\end{scope}

%% file: sch/foam_unzip_ab_kl.tex
\begin{scope}
  \begin{scope}
    \coordinate (L1) at (0.2,0.4);
    \coordinate (L2) at (0,0);
    \coordinate (R1) at (2.2,0.4);
    \coordinate (R2) at (2,0);
    \coordinate (ML) at (0.6, 0.2);
    \coordinate (MR) at (1.6, 0.2);
    \draw[->-] (ML) -- (MR) node[below, midway] {$a+b$};
    \draw (MR) .. controls +(0, 0) and +(-0.3,0) .. (R1) ;
    \draw (MR) .. controls +(0, 0) and +(-0.3,0) .. (R2);
    \draw (L1) .. controls +( 0.3, 0) and +(0,0) .. (ML);
    \draw (L2) .. controls +( 0.3, 0) and +(0,0) .. (ML);
  \end{scope}  
 \begin{scope}[yshift = 1cm]
    \coordinate (L1B) at (0.2,0.4);
    \coordinate (L2B) at (0,0);
    \coordinate (R1B) at (2.2,0.4);
    \coordinate (R2B) at (2,0);
    \draw[->-] (L1B) .. controls +( 0, 0) and +(0,0) .. (R1B) node [left, pos
    = 0] {$a$} node[pos=0.75,below] {$\deco_k$};
    \draw[->-] (L2B) .. controls +( 0, 0) and +(0,0) .. (R2B) node [left, pos
    = 0] {$b$} node[pos=0.35,below] {$\deco_l$};
 \end{scope}  
  \draw (R1) -- (R1B);
  \draw (R2) -- (R2B);
  \draw (L1) -- (L1B);
  \draw (L2) -- (L2B);
  \draw (ML) .. controls +(0, 0.6) and +(0, 0.6) .. (MR);
\end{scope}

%% file: sch/foam_cup_a_empty.tex
\begin{scope}
  \draw (0,0) arc (180 :0: 0.5cm and 0.2cm) 
  node[above, pos=0.5] {$a$};
  \draw[very thin] (0,0) arc (180 :0: 0.5cm and -0.6cm) node[pos=0.5, above] {};
  \draw (0,0) arc (180 :0: 0.5cm and -0.2cm);
  \node at (.5,-.6) {};
\end{scope}

%% file: sch/foam_cup_a_kl.tex
\begin{scope}
  \draw (0,0) arc (180 :0: 0.5cm and 0.2cm) 
  node[above, pos=0.5] {$a$};
  \draw[very thin] (0,0) arc (180 :0: 0.5cm and -0.6cm);
  \draw (0,0) arc (180 :0: 0.5cm and -0.2cm) node[pos=0.5, below] {$\deco_k \antideco_l$};
\end{scope}

%% file: sch/foam_cap_a_empty.tex
\begin{scope}[-]
  \draw (0,0) arc (180 :0: 0.5cm and 0.2cm);
  \draw[very thin] (0,0) arc (180 :0: 0.5cm and 0.6cm) node[pos=0.5,
  below] {};
  \draw (0,0) arc (180 :0: 0.5cm and -0.2cm)node[ below, pos =0.5] {$a$};
  \node at (.5,.6) {};
\end{scope}

%% file: sch/foam_cap_a_kl.tex
\begin{scope}[-]
  \draw (0,0) arc (180 :0: 0.5cm and 0.2cm);
  \draw[very thin] (0,0) arc (180 :0: 0.5cm and 0.6cm) node[pos=0.5, yshift=-2mm,below] {$\deco_k \antideco_l$};
  \draw (0,0) arc (180 :0: 0.5cm and -0.2cm)node[ below, pos =0.5] {$a$};
\end{scope}

%% file: sch/foam_saddle_a_empty.tex
\begin{scope}[scale=0.6]
\tdplotsetmaincoords{70}{25}
\begin{scope}[scale = 1.5, tdplot_main_coords]
  \tikzset{yxplane/.style={canvas is xy plane at z=#1}}
  \begin{scope}[yxplane=1]
    \coordinate (AT) at ({cos(  45)}, {sin(  45)});
    \coordinate (BT) at ({cos(135)}, {sin(135)});
    \coordinate (CT) at ({cos(225)}, {sin(225)});
    \coordinate (DT) at ({cos(315)}, {sin(315)});
    \coordinate (aT) at (  0.3, 0);
    \coordinate (bT) at ( -0.3, 0);
    \draw (AT) .. controls (aT) and (bT) .. (BT)
    coordinate[pos=0.5] (eT) node [pos =0.5, above] {$a$};
    \draw (DT) .. controls (aT) and (bT) .. (CT) coordinate[pos=0.5] (fT);
  \end{scope}

  \begin{scope}[yxplane=0]
    \coordinate (AM) at ({cos(  45)}, {sin(  45)});
    \coordinate (BM) at ({cos(135)}, {sin(135)});
    \coordinate (CM) at ({cos(225)}, {sin(225)});
    \coordinate (DM) at ({cos(315)}, {sin(315)});
    \coordinate (aM) at (0, 0.3);
    \coordinate (bM) at (0, -0.3);
    \draw (AM) .. controls (aM) and (bM) .. (DM) coordinate[pos=0.5] (eM);
    \draw (BM) .. controls (aM) and (bM) .. (CM) coordinate[pos=0.5] (fM);
\end{scope}
  \coordinate (OT) at (0,0, 0.5);
\draw (AM) -- (AT);
  \draw (BM) -- (BT);
  \draw (CM) -- (CT);
  \draw (DM) -- (DT);
  \draw[densely dotted] (eM) ..controls +(0,0,0.2) and + (0.1,0,0).. (OT);
  \draw[densely dotted] (fM) ..controls +(0,0,0.2) and + (-0.1,0,0).. (OT);
  \draw[densely dotted] (eT) ..controls +(0,0,-0.2) and + (0,0.1,0).. (OT);
  \draw[densely dotted] (fT) ..controls +(0,0,-0.2) and + (0,-0.1,0).. (OT);
\end{scope}  
\end{scope}

%% file: sch/foam_saddle_a_kl.tex
\begin{scope}[scale=0.6]
\tdplotsetmaincoords{70}{25}
\begin{scope}[scale = 1.5, tdplot_main_coords]
  \tikzset{yxplane/.style={canvas is xy plane at z=#1}}
  \begin{scope}[yxplane=1]
    \coordinate (AT) at ({cos(  45)}, {sin(  45)});
    \coordinate (BT) at ({cos(135)}, {sin(135)});
    \coordinate (CT) at ({cos(225)}, {sin(225)});
    \coordinate (DT) at ({cos(315)}, {sin(315)});
    \coordinate (aT) at (  0.3, 0);
    \coordinate (bT) at ( -0.3, 0);
    \draw (AT) .. controls (aT) and (bT) .. (BT)
    coordinate[pos=0.5] (eT) node [pos =0.5, above] {$a$};
    \draw (DT) .. controls (aT) and (bT) .. (CT) coordinate[pos=0.5] (fT);
  \end{scope}

  \begin{scope}[yxplane=0]
    \coordinate (AM) at ({cos(45)}, {sin(45)});
    \coordinate (BM) at ({cos(135)}, {sin(135)});
    \coordinate (CM) at ({cos(225)}, {sin(225)});
    \coordinate (DM) at ({cos(315)}, {sin(315)});
    \coordinate (aM) at (0, 0.3);
    \coordinate (bM) at (0, -0.3);
    \draw (AM) .. controls (aM) and (bM) .. (DM) coordinate[pos=0.5] (eM);
    \draw (BM) .. controls (aM) and (bM) .. (CM) coordinate[pos=0.5] (fM);
\end{scope}
  \coordinate (OT) at (0,0, 0.5);
    \draw (AM) -- (AT) node[pos=.5,left] {$\antideco_\ell$};
    \draw (BM) -- (BT) ;
    \draw (CM) -- (CT) node[pos=.5,right] {$\deco_k$};
    \draw (DM) -- (DT);
  \draw[densely dotted] (eM) ..controls +(0,0,0.2) and + (0.1,0,0).. (OT);
  \draw[densely dotted] (fM) ..controls +(0,0,0.2) and + (-0.1,0,0).. (OT);
  \draw[densely dotted] (eT) ..controls +(0,0,-0.2) and + (0,0.1,0).. (OT);
  \draw[densely dotted] (fT) ..controls +(0,0,-0.2) and + (0,-0.1,0).. (OT);
\end{scope}  
\end{scope}

%% file: Z_canopoly.tex
\label{sec_defcan}
The canopoly structure that we introduce in this section was first introduced by Bar-Natan in~\cite{Bar05} and can be seen as a categorification of Jones's notion of planar algebra \cite{Jon21}.

\begin{dfn}
    A $d$-input \textit{colored planar arc diagram} $\cpd$ consists of a disk called the \textit{output} with $d$ disjoint disks called \textit{inputs} removed and a finite collection of oriented arcs properly embedded in the complement of the inputs.
    % numbered $1$ through $d$ 
    The input disks and arcs are labeled by nonnegative integers.
    Arcs begin and end on the boundary of $\cpd$.
    % This information is considered only up to planar isotopy that preserve the boundary. 
\end{dfn}

\begin{figure}[h!]
    \centering
    \NB{\tikz[scale=1.5]{\input{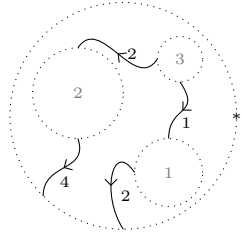}}}
    \caption{Example of a colored planar arc diagram.}
\end{figure}

\begin{dfn}
    A \textit{colored planar algebra} is a collection of sets $(\mathcal{P}(\bnd))_{\bnd}$ indexed by finite sequences of integers together with a function
    \[\cpd : \mathcal{P}(\bnd_1) \times \ldots \times \mathcal{P}(\bnd_d) \rightarrow \mathcal{P}(\bnd),\]
    for each \(d\)-input colored planar arc diagram \(\cpd\) with the boundaries of the input disks described by \(\bnd_i\), and the boundary of the output disk described by \(\bnd\).
\end{dfn}

% The curious reader can find a lot of examples of planar algebras in \cite{Bar05}. To build our first example we need the two following definitions.

\begin{dfn}
    A \textit{colored tangle diagram} is an oriented tangle diagram $T$ in $\Disk$ with components labeled by positive non-negative integers and boundary described by a finite sequence of non-zero integers $\bnd=(\bnd_1,\ldots,\bnd_{n})$.
    % where each of its components is labeled by an element of $\N^*$, with $n_\bnd$ points on $\partial\Disk$ labeled according to the oriented order from $1$ to $n_\bnd$. The integer $\bnd_i$ wears the information of the label and the orientation of the component of $T$ connected to the ith boundary point. If  $-\bnd_i \in \N$ then the orientation of the component is going outside $\Disk$ and is colored by $-\bnd_i$ else the component orientation goes inside $\Disk$ and is labeled by $\bnd_i \in \N$.
\end{dfn}

\begin{figure}[h!]
    \centering
    \NB{\tikz[scale=1.5]{\input{\imagesfolder/sch/ex_col_tang_diag}}}
    \caption{Example of a colored tangle diagram with $\bnd=(2,4,-1,-2,-1,1,1,-4)$.}
\end{figure}

\begin{dfn}
    \label{wt_def}
    A \textit{colored web–tangle diagram} with boundary $\bnd$ is a colored tangle diagram with boundary $\bnd$ from which a collection of disjoint open disks $(\mathcal{D}_k)_k$ has been removed and replaced by webs $(\web_k)_k$ having the same boundaries as the removed tangles and the color of the tangle matches the labels of the webs. The data $(\mathcal{D}_k,\web_k)_k$ is called the \textit{web part}. It is considered up to the equivalence relation on web parts generated by: \[ (\mathcal{D}_k,\web_k)_{k\in K} \sim (\mathcal{D}_j,\web_j)_{j\in J} \] if for all $k \in K$ there exists $J' \subset J$ such that \[ \bigcup_{j \in J'} \mathcal{D}_j \subset \mathcal{D}_k \quad \text{and} \quad \web_k = (\mathcal{D}_k \setminus \bigcup_{j \in J'} \mathcal{D}_j)(\web_j, j\in J'),\]  where $\mathcal{D}_k \setminus \cup_{j \in J'} \mathcal{D}_j$ is regarded as a colored planar arc diagram applied to the web $(\web_j)_{j\in J'}$. 
    %We say that two web-tangle diagrams are \textit{isotopic} if they have the same web parts and if the tangles in $\mathcal{D} \setminus \cup_k \mathcal{D}_k$ represent isotopic tangles.
\end{dfn}

\begin{figure}[H]
    \centering
    \NB{\tikz[scale=2]{\input{\imagesfolder/sch/ex_web-tang}}}
    \caption{A colored web-tangle diagram with $\bnd=(3,1,-1,-3,-1,1,1,-1)$ with its web part in blue.}
\end{figure}

\begin{exa}
    \label{ex_pa_1}
    Let $\wt(\bnd)$ be the collection of all colored web-tangle diagrams with a boundary sequence $\bnd$.
    The collection $(\wt(\bnd))_\bnd$ indexed over all boundary sequences \(u\) is a colored planar algebra.
    % $\bnd=(\bnd_1,\ldots,\bnd_{n_\bnd}) \in \Z_{\neq 0}^{n_\bnd}$,
    The action of a colored planar arc diagram $\cpd$ is given by the horizontal composition of web–tangle diagrams with the matching labeled boundary components of $\cpd$. 
\end{exa}

\begin{exa}
    \label{ex_pa_2}
   Similarly, the collection $(\cweb(\bnd))_\bnd$ of webs in $\Disk$ forms a planar algebra.
\end{exa}

\begin{figure}[H]
    \centering
    \includegraphics[]{sch/ex_comp_plandiag.tex}
    \caption{Map induced by a colored planar arc diagram that goes from $\wt(1,-1) \times \wt(2,-1,-1) \times \wt(-2,1,1) $ to $\wt(1,1,-1,-1)$.}
\end{figure}

\begin{dfn}[\cite{Bar05}]
    A \textit{colored canopoly} is a collection of small categories $C := (C(\bnd))_\bnd$ such that, 
    \begin{enumerate}
        \item $(\mathrm{Obj}(C(\bnd)))_\bnd$ is a colored planar algebra,
        \item $(\mathrm{Morph}(C(\bnd)))_\bnd$ is a colored planar algebra,
        \item the planar algebra operations commute with the compositions in the various categories.  
    \end{enumerate} 
     % A composition of morphisms is called \textit{vertical} if it comes from the category structure, and \textit{horizontal} if it comes from the planar algebra structure.
\end{dfn}

\begin{exa}
    Building on the Example \ref{ex_pa_1} of the colored planar algebra $\wt(\bnd)$, define the colored canopoly $\WT := (\WT(u))_u$:
    \begin{itemize}
        \item The objects of \(\WT(u)\) are the elements of $\wt(\bnd)$.
        \item The morphisms of \(\WT(u)\) between two colored web-tangles with the same web part $(\mathcal{D}_k,\web_k)_k$, are cobordisms between tangles where a collection of disjoint cylinder $(\mathcal{D}_k\times \interval)_k$ has been removed and replaced by $(\web_k\times \interval)_k$.
        \item The categorical composition is given by gluing morphisms vertically
        \item The planar algebra operations is given by gluing objects or morphisms horizontally.
    \end{itemize}
\end{exa}

\begin{exa}
    \label{ex:foam_can}
    A colored canopoly associated to Example~\ref{ex_pa_2} is the collection of categories $(\catf(\bnd))_\bnd$ denoted $\catf$. The action of a planar diagram at the level of morphisms (on foams) is given by vertical composition, which is the identity on all edges involved in the planar diagram.
\end{exa}

The proof of Proposition~\ref{prop:htpcanopoly} is left as an exercise. Nevertheless, it follows the same ideas as those used in Example~8.4 of \cite{Bar05} and Theorem~2.1 of \cite{W07}.

\begin{prop}
    \label{prop:htpcanopoly}
    Let $(C(\bnd))_\bnd$ be a colored canopoly, then the collection of the homotopy category of the categories $C(\bnd)$ is also a colored canopoly. 
\end{prop}

\begin{dfn}
    A \textit{morphism of colored canopoly} between $(C(\bnd))_\bnd$ and $(C'(\bnd))_\bnd$ is a collection of functors $(F_\bnd)_\bnd$, where $F_\bnd : C(\bnd) \to C'(\bnd)$, which commute with the operators defined by colored planar diagrams.
\end{dfn}

%% file: Z_BtwTA.tex
\label{BtwT&A}

If we restrict our work to links (and not tangles), only closed webs appear so we can use the Robert--Wagner evaluation defined in \cite{RW20a} to associate to each closed web $\web$ a graded $\scalars_N$-module. This module is denoted by $\ssp_N(\web)$ and called the \textit{$\gl_N$-state space} of $\web$. This construction is functorial:
\begin{align*}
    \ssp_N : \catf(\varnothing)  &\longrightarrow{} \Rmod \\
    \web &\longmapsto \ssp_N(\web) \\
    \left (\web \xrightarrow{\foam} \web' \right)  &\longmapsto  \left (\ssp_N(\web) \xrightarrow{F \circ _{-}} \ssp_N(\web') \right ).
\end{align*}

This functor is compatible with the $\Witt$-action on foams.

\begin{thm}[{\cite[Theorem 4.4]{QRSW24}}]
    For any closed web $\web$, the operators $(\Wgen_n)_{n\geq -1}$ induce an action of $\Witt$ on the state space $\ssp_N(\web)$.
\end{thm}

For a foam $\foam$ to induce a morphism of \(\Witt\)-modules between state spaces
it is explained in \cite{GR25} that the Witt action on $\foam$ needs to be trivial.
% This action can be computed before applying $\ssp_N$.
This remark leads us to the following definition. 

\begin{dfn}
    Let $\algtop(\bnd)$ be the subcategory of $\catf(\bnd)$ whose objects are the same but whose morphisms are only those $\scalars_N$-linear combinations of foams $\foam$ such that for all $n \in \N_{-1}$:
    \[\wgen_n \cdot \foam = 0.\]
    Such a foam is called \textit{$\Witt$-equivariant}.
\end{dfn}

\begin{rmk}
    In the case of $\bnd = \varnothing$, we can restrict the functor $\ssp_N$ such that: 
    \[\ssp_N : \algtop(\varnothing) \rightarrow \scalars_N \# \Witt\text{-mod}, \]
    where the arrival category is a category of $\Witt$-module describe in \cite{GR25}. In other words we have the following commutative diagram: 
    \begin{equation}
    \NB{
    \begin{tikzpicture}
        \node (a) at (0,0) {$\algtop(\varnothing)$};
        \node (b) at (3,0) {$\catf(\varnothing)$};
        \node (c) at (0,-1.5) {$\scalars_N \# \Witt\text{-mod}$};
        \node (d) at (3,-1.5) {$\scalars_N\text{-mod}$};
        \draw[->] (a) -- (b) node[pos=0.5,above,scale=0.8]{$i$};
        \draw[->] (b) -- (d) node[pos=0.5,right,scale=0.8]{$\ssp_N$};
        \draw[->] (a) -- (c) node[pos=0.5, scale=0.8, left]{$\ssp_N$}; 
        \draw[->] (c) -- (d) node[pos=.5,scale=.8,above] {$\text{For}$};
    \end{tikzpicture}}   
    \end{equation}
    where For is the forgetful functor.
\end{rmk}

\begin{prop}
    \label{cano_algtop}
    The collection of categories $(\algtop(\bnd))_\bnd$ is a canopoly denoted $\algtop$.
\end{prop}

\begin{proof}
    The canopoly structure is inherited by the one on $(\catf(\bnd))_\bnd$. The only thing to check is that the horizontal composition of $\Witt$-equivariant foams produces a $\Witt$-equivariant foam. This is straightforward as the Witt action respects the Leibniz rule for both vertical and horizontal compositions of foams. 
\end{proof}

% The following corollary is a consequence of Proposition \ref{action}. 

% \begin{cor}
%     \label{coro_enriched}
%     The category $\algtop(\bnd)$ is enriched in the category of $\Witt$-modules.
% \end{cor}

%% file: Z_twist_summary.tex
This section reviews the notions of Maurer--Cartan elements and red dots.
See \cite{GR25} for a more extensive discussion.

By the proposition of Section \ref{sec_action}, the categories \(\catf(\bnd)\) are \textit{enriched in $\Witt$-modules}. That is, \(\Witt\) acts on every hom space by derivations with respect to composition.

\begin{dfn}
\label{defMC}
A \textit{Maurer--Cartan element} for an object \(V\) of \(\catf(\bnd)\) is an \(\scalars\)-linear map \(\alpha: \Witt \to \morf(V,V)\) such that, for all \(X,Y \in \Witt\):
\[
\alpha_{[X,Y]} = [\alpha_X, \alpha_Y] + X\alpha_Y - Y\alpha_X,
\]
where \(\alpha_X\) denotes the endomorphism that \(\alpha\) assigns to \(X\) and \([~,~]\) in \( \morf(V) \) is the commutator.
\end{dfn}

A Maurer--Cartan element ``twists'' the \(\Witt\) action in the following sense.

\begin{prop}[{\cite[Lem 2.6]{GR25}}]
    \label{prop_MC}
    Let \((V, \alpha)\) and \((W, \beta)\) be pairs of objects of \(\catf(\bnd)\), and Maurer--Cartan elements.
    The following defines a new \(\Witt\)-action on
    \( f \in \morf(V,W) \):
    \begin{equation}
    \label{act_mc}
    X \cdot f := Xf + \beta_X \circ f - f \circ \alpha_X
    \end{equation}
\end{prop}

For notational convenience we introduce a diagrammatic formalism to track specific Maurer--Cartan elements.

\begin{lem}[{\cite[Lem 4.5]{GR25}}]
\label{lem_reddot}
Suppose \(\web\) is a web with an $a$-labeled edge \(e\).
Let \(p_k^{(e)}\) and \(\widehat{p}_k^{(e)}\) denote the
foams \(\web \times \interval\) decorated by the power sums \(p_k \in Sym_a\) and \(\widehat{p}_k \in Sym_{N-a}\) respectively on the thickness $a$ facet.
The following three maps define Maurer-Cartan elements for \(\Gamma\).
	\begin{itemize}
		\item A hollow red dot 
		\[\rd: L_n \mapsto -\sum_{k=0}^np_k^{(e)}p_{n-k}^{(e)}.\]
		\item A solid red dot
		\[\srd: L_n \mapsto -\sum_{k=0}^np^{(e)}_k\widehat{p}^{(e)}_{n-k}.\]
		\item A triangular red dot
		\[\trd: L_n \mapsto -(n+1) p^{(e)}_n\]
	\end{itemize}
\end{lem}

\begin{dfn}
\label{def_dots}
A \textit{red-dotted} web is a web \(\web\) endowed with a finite collection \(\cold\) of \textit{red dots} labelled by elements of \(\scalars\) on the interiors of the edges of \(\web\) representing linear combinations of the Maurer--Cartan elements defined in Lemma \ref{lem_reddot}.
\end{dfn}

The twists induced by red dots can be expressed schematically:

\begin{align}
    \wgen_n * \, \NB{\tikz[font=\small, scale=1.5]{\input{\imagesfolder/sch/def_dots_circle_hollow}}} \, &=
    \wgen_n \cdot \NB{\tikz[font=\small, scale=1.5]{\input{\imagesfolder/sch/def_dots_blank}}}
    -\omega \sum_{k=0}^n \NB{\tikz[font=\small, scale=1.5]{\input{\imagesfolder/sch/def_dots_hollow_blank}}}
    \\ 
    \wgen_n * \, \NB{\tikz[font=\small, scale=1.5]{\input{\imagesfolder/sch/def_dots_circle_solid}}} \, &=
    \wgen_n \cdot \NB{\tikz[font=\small, scale=1.5]{\input{\imagesfolder/sch/def_dots_blank}}}
    -\omega \sum_{k=0}^n \NB{\tikz[font=\small, scale=1.5]{\input{\imagesfolder/sch/def_dots_solid_blank}}}
    \\
    \wgen_n * \, \NB{\tikz[font=\small, scale=1.5]{\input{\imagesfolder/sch/def_dots_trian_hollow}}} \, &=
    \wgen_n \cdot \NB{\tikz[font=\small, scale=1.5]{\input{\imagesfolder/sch/def_dots_blank}}}
    -\omega(n+1) ~ \NB{\tikz[font=\small, scale=1.5]{\input{\imagesfolder/sch/def_dots_trian_hollow_blank}}} 
    %\\ \wgen_n \cdot \, \ttkz[font=\small, scale=1.5]{sch/def_dots_trian_solid} \, &= \wgen_n \cdot \ttkz[font=\small, scale=1.5]{sch/def_dots_blank} - \lambda(n+1) ~ \ttkz[font=\small, scale=1.5]{sch/def_dots_trian_solid_n}
\end{align} 

\begin{rmk}
Following Proposition \ref{prop_MC}, the twists from red dots on the target web of a foam are subtracted, while the twists on the source web are added. 
By the Leibniz rule, red-dotted twists are additive.
\end{rmk}

\begin{dfn}
To simplify computations, we introduce the following notation: 
\[
\NB{\tikz[font=\small, scale=0.5]{\input{\imagesfolder/sch/notation_reddot2}}} :=
\NB{\tikz[font=\small, scale=0.5]{\input{\imagesfolder/sch/notation_reddot1}}} \hspace{3mm}
\text{and} \hspace{3mm}
\NB{\tikz[font=\small, yscale=-1, scale=0.5]{\input{\imagesfolder/sch/notation_reddot4}}} :=
\NB{\tikz[font=\small, yscale=-1, scale=0.5]{\input{\imagesfolder/sch/notation_reddot3}}}
.\]
\[
\NB{\tikz[font=\small, scale=0.5]{\input{\imagesfolder/sch/notation_reddot2_solid}}} :=
\NB{\tikz[font=\small, scale=0.5]{\input{\imagesfolder/sch/notation_reddot1_solid}}} \hspace{3mm}
\text{and} \hspace{3mm}
\NB{\tikz[font=\small, yscale=-1, scale=0.5]{\input{\imagesfolder/sch/notation_reddot4_solid}}} :=
\NB{\tikz[font=\small, yscale=-1, scale=0.5]{\input{\imagesfolder/sch/notation_reddot3_solid}}}
.\]
\end{dfn}

The following is direct consequence of dot migration \cite[(11)]{RW20a}.
\begin{lem}
    \label{lemm_rd_3}
    The following objects are isomorphic in $\ralgtop((-b,-a,a+b))$: 
    \begin{equation}
        \NB{\tikz[font=\small, scale=0.5]{\input{\imagesfolder/sch/notation_reddot2_solid}}} \simeq \NB{\tikz[font=\small, scale=0.5]{\input{\imagesfolder/sch/notation_reddot2}}}
	\text{and} \hspace{3mm}
        \NB{\tikz[font=\small, yscale=-1, scale=0.5]{\input{\imagesfolder/sch/notation_reddot4_solid}}} \simeq \NB{\tikz[font=\small, yscale=-1, scale=0.5]{\input{\imagesfolder/sch/notation_reddot4}}}.
    \end{equation}
\end{lem}

\begin{lem}
    \label{circ_tri}
    The following objects are isomorphic in $\ralgtop((-1,1))$:  
    \begin{equation}
        \NB{\tikz[font=\small]{\input{\imagesfolder/sch/lemm_circ_tri_1}}} = \NB{\tikz[font=\small]{\input{\imagesfolder/sch/lemm_circ_tri_2}}}.
    \end{equation}
\end{lem}

\begin{prop}[{\cite[Cor 4.2]{GR25}}]
    \label{flatness}
    Let $(\web,\cold)$ and $(\web',\cold')$ be two red dotted webs with the same boundary.
    Let \(\tau\) and \(\tau'\) denote the Maurer--Cartan elements determined by \(\cold\) and \(\cold'\), respectively.
    Then $\foam \in \morf(\web,\web')$
    is $\Witt$-equivariant on state spaces if and only if, for all $n \in \N_{-1}$ 
    \begin{equation}
        \label{eq_action}
        \wgen_n * \foam := \wgen_n \cdot \foam + \tau_{\wgen_n} \circ \foam - \foam \circ \tau^\prime_{\wgen_n} = 0.
    \end{equation}
\end{prop}

\begin{exa}
    The two following foams:
    \[\NB{\tikz[font=\tiny, scale=1]{\input{\imagesfolder/sch/foam_zip_a_empty}}} \hspace{5.5mm} \text{and} \hspace{5.5mm} \NB{\tikz[font=\tiny, scale=1]{\input{\imagesfolder/sch/foam_unzip_a_empty}}}\]
    induce two $\Witt$-equivariant morphisms:
    \[ \mapzip: \NB{\tikz[font=\tiny, scale=0.7]{\input{\imagesfolder/sch/web_two_arrows_ab}}} \rightarrow \NB{\tikz[font=\tiny, scale=0.7]{\input{\imagesfolder/sch/web_zip_end}}} \hspace{5.5mm} \text{and} \hspace{5.5mm} \mapunzip: \NB{\tikz[font=\tiny, scale=0.7]{\input{\imagesfolder/sch/web_unzip_beg}}} \rightarrow \NB{\tikz[font=\tiny, scale=0.7]{\input{\imagesfolder/sch/web_two_arrows_ab}}}.\]
    More examples of this type can be found in Section~4.3 of \cite{GR25}. We recall that, compared with the uncolored case, we set $\lambda = \mu$.
\end{exa}

\begin{dfn}
Let \(\mathrm{r}\catf(\bnd)\) denote the \(\Witt\)-enriched category consisting of pairs of objects and Maurer--Cartan elements in \(\catf(\bnd)\) with the \(\Witt\)-actions on morphisms defined by (\ref{act_mc}).
     Let $\ralgtop(\bnd)$ be the subcategory of \(\mathrm{r}\catf(\bnd)\) whose morphisms are only those which are killed by \(\Witt\).
\end{dfn}

The proof of Proposition \ref{prop:canord} is the same as that of Proposition \ref{cano_algtop}.
\begin{prop}
    \label{prop:canord}
    The collection of categories \((\ralgtop(\bnd))_\bnd\) denoted $\ralgtop$ is a colored canopoly. 
\end{prop}

%% file: sch/def_dots_circle_hollow.tex
\begin{scope}
  \draw[thin] (0,0) rectangle (1,0.7);
  \draw[thick] (0,0.7) -- (1,0.7)  coordinate [pos=0.5] (A);
  \filldraw[draw= \crd, fill = white] (A) circle (.7mm)
  node[below, yshift=-0.5mm] {\textcolor{\crd}{$\omega$}};
\end{scope}

%% file: sch/def_dots_blank.tex
\begin{scope}
  \draw[thin] (0,0) rectangle (1,0.7);
  \draw[thick] (0,0.7) -- (1,0.7)  coordinate [pos=0.5] (A);
  % \filldraw[draw= \crd, fill = white] (A) circle (.7mm)
  % node[below, yshift=-0.5mm] {$\omega$};
\end{scope}

%% file: sch/def_dots_hollow_blank.tex
\begin{scope}
  \draw[thin] (0,0) rectangle (1,0.7);
  \draw[thick] (0,0.7) -- (1,0.7)  coordinate [pos=0.5] (A);
  % \filldraw[draw= \crd, fill = white] (A) circle (.7mm)
  % node[below, yshift=-0.5mm] {$\omega$};
  \node (B) at (.52,.3) {$\deco_k \deco_{n-k}$};
\end{scope}

%% file: sch/def_dots_circle_solid.tex
\begin{scope}
  \draw[thin] (0,0) rectangle (1,0.7);
  \draw[thick] (0,0.7) -- (1,0.7)  coordinate [pos=0.5] (A);
  \filldraw[draw= \crd, fill = \csrd] (A) circle (.7mm)
  node[below, yshift=-0.5mm] {\textcolor{\crd}{$\omega$}};
\end{scope}

%% file: sch/def_dots_solid_blank.tex
\begin{scope}
  \draw[thin] (0,0) rectangle (1,0.7);
  \draw[thick] (0,0.7) -- (1,0.7)  coordinate [pos=0.5] (A);
  % \filldraw[draw= \crd, fill = \csrd] (A) circle (.7mm)
  % node[below, yshift=-0.5mm] {$\omega$};
  \node (B) at (.52,.3) {$\deco_k\antideco_{n-k}$};
\end{scope}

%% file: sch/def_dots_trian_hollow.tex
\begin{scope}
  \draw[thin] (0,0) rectangle (1,0.7);
  \draw[thick] (0,0.7) -- (1,0.7)  coordinate [pos=0.5] (A); 
  \node (C) at (0.5,0.7) [draw=\crd, fill=white,regular polygon, regular polygon sides=3,inner sep=1.5pt]{};
  \node (D) at (0.5,0.7) [below,yshift=-0.5mm]{\textcolor{\crd}{$\omega$}};   
\end{scope}

%% file: sch/def_dots_trian_hollow_blank.tex
\begin{scope}
  \draw[thin] (0,0) rectangle (1,0.7);
  \draw[thick] (0,0.7) -- (1,0.7)  coordinate [pos=0.5] (A); 
  % \node (C) at (0.5,0.7) [draw=\crd, fill=white,regular polygon, regular polygon sides=3,inner sep=1.5pt]{};
  % \node (D) at (0.5,0.7) [below,yshift=-0.5mm]{$\omega$}; 
  \node (B) at (.55,.3) {$\deco_n$};
\end{scope}

%% file: sch/lemm_circ_tri_1.tex
\begin{scope}
  \coordinate (m) at (0,0);
  \coordinate (t) at (2,0);
  \draw[->] (m) -- (t) node[pos = 1, right,font=\tiny] {$1$} coordinate[pos=0.5] (ga) node[pos= 0.5 , draw=\crd, fill=white,regular polygon, regular polygon sides=3,inner sep=1.2pt] {} node[pos=0.5, above, yshift=0.5mm]{\color{\crd}$\omega$};
\end{scope}

%% file: sch/lemm_circ_tri_2.tex
\begin{scope}
  \coordinate (m) at (0,0);
  \coordinate (t) at (2,0);
  \draw[->] (m) -- (t) node[pos = 1, right,font=\tiny] {$1$} coordinate[pos=0.5] (ga) coordinate[pos=0.7] (gb);
  \filldraw[draw= \crd, fill = white] (ga) circle (1mm) node[yshift= 3mm] {\color{\crd}$\omega$};
\end{scope}

%% file: sch/foam_zip_a_empty.tex
\begin{scope}
  \begin{scope}
    \coordinate (L1) at (0.2,0.4);
    \coordinate (L2) at (0,0);
    \coordinate (R1) at (2.2,0.4);
    \coordinate (R2) at (2,0);
    \coordinate (ML) at (0.6, 0.2);
    \coordinate (MR) at (1.6, 0.2);
    \draw[->-] (ML) -- (MR) node[above, midway] {$a+b$};
    \draw (MR) .. controls +(0, 0) and +(-0.3,0) .. (R1) ;
    \draw (MR) .. controls +(0, 0) and +(-0.3,0) .. (R2);
    \draw (L1) .. controls +( 0.3, 0) and +(0,0) .. (ML);
    \draw (L2) .. controls +( 0.3, 0) and +(0,0) .. (ML);
  \end{scope}  
 \begin{scope}[yshift = -1cm]
    \coordinate (L1B) at (0.2,0.4);
    \coordinate (L2B) at (0,0);
    \coordinate (R1B) at (2.2,0.4);
    \coordinate (R2B) at (2,0);
    \draw[->-] (L1B) .. controls +( 0, 0) and +(0,0) .. (R1B) node [right, pos
    = 1] {$a$};
    \draw[->-] (L2B) .. controls +( 0, 0) and +(0,0) .. (R2B) node [right, pos
    = 1] {$b$};
 \end{scope}  
  \draw (R1) -- (R1B);
  \draw (R2) -- (R2B);
  \draw (L1) -- (L1B);
  \draw (L2) -- (L2B);
  \draw (ML) .. controls +(0, -0.6) and +(0, -0.6) .. (MR);
\end{scope}

%% file: sch/foam_unzip_a_empty.tex
\begin{scope}
  \begin{scope}
    \coordinate (L1) at (0.2,0.4);
    \coordinate (L2) at (0,0);
    \coordinate (R1) at (2.2,0.4);
    \coordinate (R2) at (2,0);
    \coordinate (ML) at (0.6, 0.2);
    \coordinate (MR) at (1.6, 0.2);
    \draw[->-] (ML) -- (MR) node[below, midway] {$a+b$};
    \draw (MR) .. controls +(0, 0) and +(-0.3,0) .. (R1) ;
    \draw (MR) .. controls +(0, 0) and +(-0.3,0) .. (R2);
    \draw (L1) .. controls +( 0.3, 0) and +(0,0) .. (ML);
    \draw (L2) .. controls +( 0.3, 0) and +(0,0) .. (ML);
  \end{scope}  
 \begin{scope}[yshift = 1cm]
    \coordinate (L1B) at (0.2,0.4);
    \coordinate (L2B) at (0,0);
    \coordinate (R1B) at (2.2,0.4);
    \coordinate (R2B) at (2,0);
    \draw[->-] (L1B) .. controls +( 0, 0) and +(0,0) .. (R1B) node [left, pos
    = 0] {$a$};
    \draw[->-] (L2B) .. controls +( 0, 0) and +(0,0) .. (R2B) node [left, pos
    = 0] {$b$};
 \end{scope}  
  \draw (R1) -- (R1B);
  \draw (R2) -- (R2B);
  \draw (L1) -- (L1B);
  \draw (L2) -- (L2B);
  \draw (ML) .. controls +(0, 0.6) and +(0, 0.6) .. (MR);
\end{scope}

%% file: sch/web_two_arrows_ab.tex
\begin{scope}
    \coordinate (bl) at (-0.5, -1);
    \coordinate (br) at ( 0.5, -1);
    \coordinate (tl) at (-0.5,  1);
    \coordinate (tr) at ( 0.5,  1);
    \draw[->] (bl) -- (tl) node[pos=0,below] {$a$};
    \draw[->] (br) -- (tr) node[pos=0,below] {$b$};
\end{scope}

%% file: Z_RHC_2.tex
% We consider the homotopy category \(\mathcal{C}(\ralgtop(\bnd))\) of chain complexes whose objects and morphisms are those of \(\ralgtop(\bnd)\).

\begin{dfn}
    Let \(\mathrm{For}\) denote the following two functors that forget red dots and the \(\Witt\)-action.
\begin{align*}
    \mathrm{For}:\ralgtop(\bnd) &\longrightarrow \catf(\bnd), %\\
    % (\web,\alpha) &\longmapsto \web \\
    % \foam &\longmapsto \foam
\end{align*}
\[\mathrm{For}: \mathcal{C}(\ralgtop(\bnd)) \longrightarrow \mathcal{C}(\catf(\bnd)).\]
\end{dfn}

\begin{dfn}
    The \textit{relative homotopy category} is the Verdier quotient
    \[\relcat{\bnd}:= \frac{\mathcal{C}(\ralgtop(\bnd))}{\Ker(\mathrm{For})}.\]
\end{dfn}

The universal property of the Verdier quotient provides the following commutative diagram:
\begin{equation}
    \NB{
    \begin{tikzpicture}
        \node (a) at (0,0) {$\mathcal{C}(\ralgtop(\bnd))$};
        \node (b) at (5,0) {$\mathcal{C}(\catf(\bnd))$};
        \node (c) at (2.5,-1.5) {$\relcat{\bnd}$};
        \draw[->] (a) -- (b) node[pos=0.5,above,scale=0.8]{$\mathrm{For}$};
        \draw[dashed,->] (c) -- (b) node[pos=0.5,right,scale=0.8]{$\forg$};
        \draw[->] (a) -- (c) node[pos=0.5, scale=0.8, left, xshift=-1.5mm]{$\Funiv$}; 
    \end{tikzpicture}},  
    \end{equation}
% where $\Funiv$ comes from the definition of the Verdier quotient. 

% % We work here with categories of chain of complexes over additive categories. These categories are equipped with termwise split exact sequences.
% \begin{dfn}
%     Let $\mathcal{A}$ be an additive category, and $\mathcal{C}(\mathcal{A})$ be its homotopy category.
%     Let $A$, $B$ and $C$ be object in $\mathcal{C}(\mathcal{A})$. A \textit{termwise split exact sequence} of $\mathcal{C}(\mathcal{A})$ is a complex 
%     \[0 \rightarrow A \xrightarrow{f} B \xrightarrow{g} C \rightarrow 0,\] isomorphic to 
%     \[0 \rightarrow \Tilde{A} \xrightarrow{i} \Tilde{A} \oplus \Tilde{C} \xrightarrow{p} \Tilde{C} \rightarrow 0,\] 
%     where $i$ and $p$ are the inclusion maps and the projection maps.
% \end{dfn}

% As the homotopy category of an additive category, $\mathcal{C}(\ralgtop(\bnd))$ is a triangulated category whose distinguished triangles are those isomorphic to
% \[
% M \xrightarrow{f} N \rightarrow L \rightarrow t^{-1}M.
% \]
% where $M \rightarrow N \rightarrow L$ is a termwise split exact sequence.  

The proof of Proposition~\ref{prop:RHcanop} relies on the definition of Verdier quotient categories given in \cite[Chap.~2]{neeman2014triangulated}, in particular on the construction of morphisms in these categories. The key point is that the homotopy categories of complexes inherit an action of colored planar arc diagrams.

\begin{prop}
    \label{prop:RHcanop}
    The collection of categories $(\relcat{\bnd})_\bnd$ denoted $\relcatt$, is a colored canopoly. 
\end{prop}

\begin{prop}[{\cite[Lem~2.3]{QRSW21}}]
	The Verdier quotient $\relcat{\bnd}$ is a triangulated category
    such that a triangle is distinguished in $\relcat{\bnd}$ if and only if it is isomorphic to
    \[\Funiv(M) \rightarrow \Funiv(N) \rightarrow \Funiv(L) \rightarrow t^{-1}\Funiv(M),\]
    where $\mathrm{For}(M \rightarrow N \rightarrow L)$ is a termwise split exact sequence in $\mathcal{C}(\catf(\bnd))$. 
\end{prop}

\begin{cor}
    \label{lemme_RHC}
    Let $M \xrightarrow{f} N \xrightarrow{g} L$ be in $\mathcal{C}(\ralgtop(\bnd))$ such that $\mathrm{For}(M \rightarrow N \rightarrow L)$ is a termwise split exact sequence in $\mathcal{C}(\catf(\bnd))$. If $L$ is contractible, then $\Funiv(f)$ is an isomorphism in $\relcat{\bnd}$. Similarly, if $M$ is contractible then $\Funiv(g)$ is an isomorphism in $\relcat{\bnd}$.
\end{cor}

%% file: Z_complex.tex
To define the colored Khovanov--Rozansky homology, we first associate a complex in $\relcatt$ to each crossing.
Let $a, b \in \N^*$ and $m:=\min(a,b)$:

Let
\begin{equation}
    C_i^-:=\NB{\tikz[font=\fontsize{1.6mm}{0mm},scale=1.25]{\input{\imagesfolder/sch/col_com_neg_i}}}, \hspace{10mm}C_{-j}^+:= \NB{\tikz[font=\fontsize{1.6mm}{0mm},scale=1.25]{\input{\imagesfolder/sch/col_com_pos_i}}}
\end{equation}
and
\begin{equation}
    d_i^-:=\mymoviefive[scale=.6]{\NB{\tikz[font=\fontsize{1.6mm}{0mm},scale=.55]{\input{\imagesfolder/sch/diff_neg_1}}}}{\NB{\tikz[font=\fontsize{1.6mm}{0mm},scale=.55]{\input{\imagesfolder/sch/diff_neg_2}}}}{\NB{\tikz[font=\fontsize{1.6mm}{0mm},scale=.55]{\input{\imagesfolder/sch/diff_neg_3}}}}{\NB{\tikz[font=\fontsize{1.6mm}{0mm},scale=.55]{\input{\imagesfolder/sch/diff_neg_4}}}}{\NB{\tikz[font=\fontsize{1.6mm}{0mm},scale=.55]{\input{\imagesfolder/sch/diff_neg_5}}}}
\end{equation}
\begin{equation}
    d_j^+:=\mymoviefive[scale=.6]{\NB{\tikz[font=\fontsize{1.5mm}{0mm},scale=.55]{\input{\imagesfolder/sch/diff_pos_1}}}}{\NB{\tikz[font=\fontsize{1.5mm}{0mm},scale=.55]{\input{\imagesfolder/sch/diff_pos_2}}}}{\NB{\tikz[font=\fontsize{1.5mm}{0mm},scale=.55]{\input{\imagesfolder/sch/diff_pos_3}}}}{\NB{\tikz[font=\fontsize{1.5mm}{0mm},scale=.55]{\input{\imagesfolder/sch/diff_pos_4}}}}{\NB{\tikz[font=\fontsize{1.25mm}{0mm},scale=.55]{\input{\imagesfolder/sch/diff_pos_5}}}}
\end{equation}.

Define the following pair of complexes for the negative and positive crossings.
\begin{equation}
      \NB{\tikz[xscale = -0.6]{\input{\imagesfolder/sch/negative_crossing_col}}} := \uwave{q^{m} C_0^-} \xrightarrow{d_0^-} \dots q^{m-i}C_i^- \xrightarrow{d_i^-} q^{m-i-1}C_{i+1}^-\ldots \xrightarrow{d_{m-1}^-} C_m^- 
\end{equation}
\begin{equation}
      \NB{\tikz[xscale =0.6]{\input{\imagesfolder/sch/positive_crossing_col}}} :=  C_{-m}^+ \xrightarrow{d_{-m}^+} \dots q^{-m-j}C_{-j}^+ \xrightarrow{d_{-j}^+} q^{-m-j+1}C_{-j+1}^+\ldots \xrightarrow{d_{-1}^+} \uwave{q^{-m}C_0^+}
\end{equation}
In general, for a colored web-tangle diagram \(\tangle\) in \(\WT(\bnd)\), we associate a complex by assigning locally to each crossing the corresponding complex, and to each web part the one-term complex consisting of that web. We denote by $\colcom_N(\tangle)$ the complex associated to $\tangle$ in $\relcat{\bnd}$.

\begin{rmk}
    We use the conventions of \cite{ETW17} for the complexes associated to each crossing, while adopting the quantum and homological grading conventions from \cite{GR25}, which are those of the framed version of Khovanov--Rozansky homology equipped with a $\Witt$-action.
\end{rmk}

\begin{rmk}
    \label{parameters} In comparison with \cite{QRSW23} and \cite{GR25}, We set $\lambda = \mu$ in the definition of the $\Witt$–action on foams; otherwise, we would not obtain a recursive expression for the red dots on the complex. 
\end{rmk}

%% file: sch/col_com_neg_i.tex
\begin{scope}
  \def\la{$b-i$}
  \def\cir{$\color{\crd}-s/2$}
  \def\tri{$\color{\crd}-\lambda(b-i-1)+s/2$}
  \def\lab{$a+b-i$}
  \def\circ{$\color{\crd}s/2$}
  \def\tria{$\color{\crd}\lambda(a+b-i-1)-s/2$}
  \def\labe{$a-i$}
  \def\circl{$\color{\crd}-s/2$}
  \def\trian{$\color{\crd}-\lambda(a-i-1)+s/2$}
  \def\label{$i$}
  \def\circle{$\color{\crd}\bar{s}/2$}
  \def\triang{$\color{\crd}-\lambda(i-1)-\bar{s}/2$}
  \coordinate (bl) at (-1, -1.5);
  \coordinate (br) at ( 1, -1.5);
  \coordinate (tl) at (-1,  1.5);
  \coordinate (tr) at ( 1,  1.5);
  \draw[>->] (bl) -- (tl) node[pos = 0, below] {$a$} node[pos = 1, above] {$b$}  node[pos=.5, left] {\lab}
  node[pos= 0.6 , draw=\crd, fill=white,regular polygon, regular polygon sides=3,inner sep=0.9pt] {} node[pos=0.6, right]{\tria} 
  node[pos=.4, draw=\crd, fill= white, circle , inner sep=1.5pt]{} node[pos=.4, left]{\circ}
  coordinate[pos = 0.75] (ml) 
  coordinate[pos = 0.25] (nl);
  \draw[>->] (br) -- (tr) node[pos = 0, below] {$b$} node[pos = 1, above] {$a$}  node[pos=.5,right] {\label}
  node[pos= 0.35 , draw=\crd, fill=white,regular polygon, regular polygon sides=3,inner sep=0.9pt] {} node[pos=0.35, right]{\triang} 
  node[pos=.65, draw=\crd, fill= white, circle , inner sep=1.5pt]{} node[pos=.65, right]{\circle}
  coordinate[pos = 0.85] (mr)  
  coordinate[pos = 0.15] (nr);
  \draw[->-] (ml) -- (mr) node [pos= 0.4, above,yshift=.5mm] {\labe} node[pos=.8, draw=\crd, fill= white, circle , inner sep=1.5pt]{} node[pos=.82, above]{\circl} 
  node[pos= 0.25 , draw=\crd, fill=white,regular polygon, regular polygon sides=3,inner sep=0.9pt] {} 
  node[pos=.45, below,yshift=-.5mm] {\trian};
  \draw[->-] (nr) -- (nl) node [pos= 0.55, below,yshift=-.5mm] {\la} node[pos=.2, draw=\crd, fill= white, circle , inner sep=1.5pt]{} node[pos=.15, below]{\cir} 
  node[pos= 0.75 , draw=\crd, fill=white,regular polygon, regular polygon sides=3,inner sep=0.9pt] {} 
  node[pos=.55, above,yshift=.2mm] {\tri};
\end{scope}

%% file: sch/col_com_pos_i.tex
\begin{scope}
  \def\la{$b-j$}
  \def\cir{$\color{\crd}\bar{s}/2$}
  \def\tri{$\color{\crd}-\lambda(b-j-1)-\bar{s}/2$}
  \def\lab{$a+b-j$}
  \def\circ{$\color{\crd}-\bar{s}/2$}
  \def\tria{$\color{\crd}\lambda(a+b-j-1)+\bar{s}/2$}
  \def\labe{$a-j$}
  \def\circl{$\color{\crd}\bar{s}/2$}
  \def\trian{$\color{\crd}-\lambda(a-j-1)-\bar{s}/2$}
  \def\label{$j$}
  \def\circle{$\color{\crd}-s/2$}
  \def\triang{$\color{\crd}-\lambda(j-1)+s/2$}
  \coordinate (bl) at (-1, -1.5);
  \coordinate (br) at ( 1, -1.5);
  \coordinate (tl) at (-1,  1.5);
  \coordinate (tr) at ( 1,  1.5);
  \draw[>->] (bl) -- (tl) node[pos = 0, below] {$a$} node[pos = 1, above] {$b$}  node[pos=.5, left] {\lab}
  node[pos= 0.6 , draw=\crd, fill=white,regular polygon, regular polygon sides=3,inner sep=0.9pt] {} node[pos=0.6, right]{\tria}
  node[pos=.4, draw=\crd, fill= white, circle , inner sep=1.5pt]{} node[pos=.4, left]{\circ}
  coordinate[pos = 0.75] (ml) 
  coordinate[pos = 0.25] (nl);
  \draw[>->] (br) -- (tr) node[pos = 0, below] {$b$} node[pos = 1, above] {$a$}  node[pos=.5,right] {\label}
  node[pos= 0.35 , draw=\crd, fill=white,regular polygon, regular polygon sides=3,inner sep=0.9pt] {} node[pos=0.35, right]{\triang} 
  node[pos=.65, draw=\crd, fill= white, circle , inner sep=1.5pt]{} node[pos=.65, right]{\circle}
  coordinate[pos = 0.85] (mr)  
  coordinate[pos = 0.15] (nr);
  \draw[->-] (ml) -- (mr) node [pos= 0.4, above,yshift=.5mm] {\labe} node[pos=.8, draw=\crd, fill= white, circle , inner sep=1.5pt]{} node[pos=.82, above]{\circl} 
  node[pos= 0.25 , draw=\crd, fill=white,regular polygon, regular polygon sides=3,inner sep=0.9pt] {} 
  node[pos=.5, below,yshift=-.7mm] {\trian};
  \draw[->-] (nr) -- (nl) node [pos= 0.55, below,yshift=-.5mm] {\la} node[pos=.2, draw=\crd, fill= white, circle , inner sep=1.5pt]{} node[pos=.15, below]{\cir} 
  node[pos= 0.75 , draw=\crd, fill=white,regular polygon, regular polygon sides=3,inner sep=0.9pt] {} 
  node[pos=.53, above,yshift=.4mm] {\tri};
\end{scope}

%% file: sch/diff_neg_1.tex
\begin{scope}
  \def\la{$b-i$}
  \def\lab{$a+b-i$}
  \def\labe{$a-i$}
  \def\label{$i$}
  \coordinate (bl) at (-1, -1.5);
  \coordinate (br) at ( 1, -1.5);
  \coordinate (tl) at (-1,  1.5);
  \coordinate (tr) at ( 1,  1.5);
  \draw[>->] (bl) -- (tl) node[pos = 0, below] {$a$} node[pos = 1, above] {$b$}  node[pos=.5, right] {}
  coordinate[pos = 0.75] (ml)
  coordinate[pos = 0.25] (nl);
  \draw[>->] (br) -- (tr) node[pos = 0, below] {$b$} node[pos = 1, above] {$a$}  node[pos=.5,right] {\label}
  coordinate[pos = 0.85] (mr)  
  coordinate[pos = 0.15] (nr);
  \draw[->-] (nr) -- (mr);
  \draw[->-] (nl) -- (ml);
  \draw[->-] (ml) -- (mr) node [pos= 0.5, above,yshift=.5mm] {\labe};
  \draw[->-] (nr) -- (nl) node [pos= 0.5, below,yshift=-.5mm] {\la} ;
\end{scope}

%% file: sch/diff_neg_3.tex
\begin{scope}
  \def\la{$b-i-1$}
  \def\lab{$a+b-i$}
  \def\labe{$a-i-1$}
  \def\label{$i$}
  \coordinate (bl) at (-1, -1.5);
  \coordinate (br) at ( 1, -1.5);
  \coordinate (tl) at (-1,  1.5);
  \coordinate (tr) at ( 1,  1.5);
  \draw[>->] (bl) -- (tl) node[pos = 0, below] {$a$} node[pos = 1, above] {$b$}  node[pos=.5, right] {}
  coordinate[pos = 0.75] (ml)
  coordinate[pos = 0.25] (nl)
  coordinate[pos=0.4] (ld)
  coordinate[pos=0.6] (lu);
  \draw[>->] (br) -- (tr) node[pos = 0, below] {$b$} node[pos = 1, above] {$a$}  node[pos=.5,right] {\label}
  coordinate[pos = 0.85] (mr)  
  coordinate[pos = 0.15] (nr)
  coordinate[pos=0.7] (ru)
  coordinate[pos=0.3] (rd);
  \draw[->-] (nr) -- (mr);
  \draw[->-] (nl) -- (ml);
  \draw[->-] (ml) -- (mr) node [pos= 0.5, above,yshift=.5mm] {\labe};
  \draw[->-] (nr) -- (nl) node [pos= 0.5, below,yshift=-.5mm] {\la} ;
  \draw[->-] (lu) -- (ru) node[pos=.3, below]{$1$};
  \draw[->-] (rd) -- (ld) node[pos=.3, above]{$1$};
\end{scope}

%% file: sch/diff_neg_5.tex
\begin{scope}
  \def\la{$b-i-1$}
  \def\lab{$a+b-i-1$}
  \def\labe{$a-i-1$}
  \def\label{$i+1$}
  \coordinate (bl) at (-1, -1.5);
  \coordinate (br) at ( 1, -1.5);
  \coordinate (tl) at (-1,  1.5);
  \coordinate (tr) at ( 1,  1.5);
  \draw[>->] (bl) -- (tl) node[pos = 0, below] {$a$} node[pos = 1, above] {$b$}  node[pos=.5, right] {}
  coordinate[pos = 0.75] (ml)
  coordinate[pos = 0.25] (nl);
  \draw[>->] (br) -- (tr) node[pos = 0, below] {$b$} node[pos = 1, above] {$a$}  node[pos=.5,left] {\label}
  coordinate[pos = 0.85] (mr)  
  coordinate[pos = 0.15] (nr);
  \draw[->-] (nr) -- (mr);
  \draw[->-] (nl) -- (ml);
  \draw[->-] (ml) -- (mr) node [pos= 0.5, above,yshift=.5mm] {\labe};
  \draw[->-] (nr) -- (nl) node [pos= 0.5, below,yshift=-.5mm] {\la} ;
\end{scope}

%% file: sch/diff_pos_1.tex
\begin{scope}
  \def\la{$b-j$}
  \def\lab{$a+j$}
  \def\labe{$a-j$}
  \def\label{$j$}
  \coordinate (bl) at (-1, -1.5);
  \coordinate (br) at ( 1, -1.5);
  \coordinate (tl) at (-1,  1.5);
  \coordinate (tr) at ( 1,  1.5);
  \draw[>->] (bl) -- (tl) node[pos = 0, below] {$a$} node[pos = 1, above] {$b$}  node[pos=.5, right] {}
  coordinate[pos = 0.75] (ml)
  coordinate[pos = 0.25] (nl);
  \draw[>->] (br) -- (tr) node[pos = 0, below] {$b$} node[pos = 1, above] {$a$}  node[pos=.5,left] {\label}
  coordinate[pos = 0.85] (mr)  
  coordinate[pos = 0.15] (nr);
  \draw[->-] (nr) -- (mr);
  \draw[->-] (nl) -- (ml);
  \draw[->-] (ml) -- (mr) node [pos= 0.5, above,yshift=.5mm] {\labe};
  \draw[->-] (nr) -- (nl) node [pos= 0.5, below,yshift=-.5mm] {\la} ;
\end{scope}

%% file: sch/diff_pos_3.tex
\begin{scope}
  \def\la{$b-j$}
  \def\lab{$a+j$}
  \def\labe{$a-j$}
  \def\label{$j$}
  \coordinate (bl) at (-1, -1.5);
  \coordinate (br) at ( 1, -1.5);
  \coordinate (tl) at (-1,  1.5);
  \coordinate (tr) at ( 1,  1.5);
  \draw[>->] (bl) -- (tl) node[pos = 0, below] {$a$} node[pos = 1, above] {$b$}  node[pos=.5, right] {}
  coordinate[pos = 0.75] (ml)
  coordinate[pos = 0.25] (nl)
  coordinate[pos=0.4] (ld)
  coordinate[pos=0.6] (lu);
  \draw[>->] (br) -- (tr) node[pos = 0, below] {$b$} node[pos = 1, above] {$a$}  node[pos=.5,right] {}
  coordinate[pos = 0.85] (mr)  
  coordinate[pos = 0.15] (nr)
  coordinate[pos=0.7] (ru)
  coordinate[pos=0.3] (rd);
  \draw[->-] (nr) -- (mr);
  \draw[->-] (nl) -- (ml);
  \draw[->-] (ml) -- (mr) node [pos= 0.5, above,yshift=.5mm] {\labe};
  \draw[->-] (nr) -- (nl) node [pos= 0.5, below,yshift=-.5mm] {\la} ;
  \draw[->-] (lu) -- (ru) node[pos=.3, below]{$1$};
  \draw[->-] (rd) -- (ld) node[pos=.3, above]{$1$};
\end{scope}

%% file: sch/diff_pos_5.tex
\begin{scope}
  \def\la{$b-j+1$}
  \def\labe{$a-j+1$}
  \def\label{$j-1$}
  \coordinate (bl) at (-1, -1.5);
  \coordinate (br) at ( 1, -1.5);
  \coordinate (tl) at (-1,  1.5);
  \coordinate (tr) at ( 1,  1.5);
  \draw[>->] (bl) -- (tl) node[pos = 0, below] {$a$} node[pos = 1, above] {$b$}  node[pos=.5, right] {}
  coordinate[pos = 0.75] (ml)
  coordinate[pos = 0.25] (nl);
  \draw[>->] (br) -- (tr) node[pos = 0, below] {$b$} node[pos = 1, above] {$a$}  node[pos=.5,left] {\label}
  coordinate[pos = 0.85] (mr)  
  coordinate[pos = 0.15] (nr);
  \draw[->-] (nr) -- (mr);
  \draw[->-] (nl) -- (ml);
  \draw[->-] (ml) -- (mr) node [pos= 0.5, above] {\labe};
  \draw[->-] (nr) -- (nl) node [pos= 0.5, below] {\la} ;
\end{scope}

%% file: sch/negative_crossing_col.tex
\begin{scope}[font=\tiny]
  \draw[->] (0.5, -0.5) ..controls +(0,0.3) and +(0,-0.3) .. (-0.5, 0.5) node[pos=0,below,yshift=-.55mm] {$a$};
  \fill[white] (0,0) circle (2mm);
  \draw[->] (-0.5, -0.5) ..controls +(0,0.3) and +(0,-0.3) .. (0.5,0.5) node[pos=0,below] {$b$};
\end{scope}

%% file: sch/positive_crossing_col.tex
\begin{scope}[font=\tiny]
  \draw[->] (0.5, -0.5) ..controls +(0,0.3) and +(0,-0.3) .. (-0.5, 0.5) node[pos=0,below] {$b$};
  \fill[white] (0,0) circle (2mm);
  \draw[->] (-0.5, -0.5) ..controls +(0,0.3) and +(0,-0.3) .. (0.5,0.5) node[pos=0,below,yshift=-.55mm] {$a$};
\end{scope}

%% file: Z_simplicity.tex
We recall that in Section \ref{RHC}, the construction of the relative homotopy canopoly induces the following forgetful canopoly morphism:
\begin{equation}
    \forg : \relcatt \rightarrow \mathcal{C}(\catf).
\end{equation}
The next definition is analogous to the notion of $Kh$-simplicity introduced by Bar-Natan \cite[Def.~8.5]{Bar05}. 

\begin{dfn}
    A web-tangle diagram $\wtd$ in $\WT(\bnd)$ is \textit{$\simple$} if the $\scalars$-module of endomorphisms in $\mathcal{C}(\catf)$ of its colored Khovanov--Rozansky $\gl_N$-homology, denoted  $\en(\forg(\colcom(\wtd))$, is free, positively graded and $\en(\forg(\colcom(\wtd))_0$ the endomorphisms of degree 0, has rank $1$.
\end{dfn}

We now develop tools to determine which web-tangle diagrams are $\simple$. To this end, we follow the same strategy as in Section~8.3 of \cite{Bar05}.

\begin{rmk}
    \label{BN}
    Observe that $\forg(\colcom_N(\wtd))$ is the usual colored equivariant Khovanov--Rozansky homology of $\wtd$. In particular, its invariance under tangle isotopies is ensured by \cite{ETW17}.
\end{rmk}

\begin{lem}[{\cite[Lem~8.7]{Bar05}}]
    If two diagrams $\wtd$ and $\wtd'$ of web-tangle are isotopic and $\wtd$ is $\simple$ then $\wtd'$ is also $\simple$.    
\end{lem}

% \begin{proof}
%     By using the Remark~\ref{BN}, the complexes $\forg(\colcom_N(\wtd))$ and $\forg(\colcom_N(\wtd'))$ are homotopy equivalent. Let $H$ be a homotopy equivalence between them and $G$ its inverse.
%     Let $\alpha'$ be an endomorphism of $\forg(\colcom_N(\wtd'))$ that is not zero then, $\alpha:= G \circ \alpha' \circ H \in \forg(\colcom_N(\wtd))$ is an endomorphism of $\forg(\colcom_N(\wtd))$ that is not zero too and has the same degree as $\alpha'$.
%     \begin{equation}
%     \NB{
%     \begin{tikzpicture}
%         \node (a) at (0,0) {$\forg(\colcom_N(\wtd))$};
%         \node (b) at (3.5,0) {$\forg(\colcom_N(\wtd))$};
%         \node (c) at (0,-1.5) {$\forg(\colcom_N(\wtd'))$};
%         \node (d) at (3.5,-1.5) {$\forg(\colcom_N(\wtd'))$};
%         \draw[->] (a) -- (b) node[pos=0.5,above,scale=0.8]{$\alpha$};
%         \draw[->] (-.1,-.3) -- (-.1,-1.2) node[pos=0.5, scale=0.8, left]{$H$}; 
%         \draw[<-] (.1,-.3) -- (.1,-1.2) node[pos=0.5, scale=0.8, right]{$G$};
%         \draw[->] (3.4,-.3) -- (3.4,-1.2) node[pos=0.5, scale=0.8, left]{$H$}; 
%         \draw[<-] (3.6,-.3) -- (3.6,-1.2) node[pos=0.5, scale=0.8, right]{$G$};
%         \draw[->] (c) -- (d) node[pos=.5,scale=.8,above] {$\alpha'$};
%     \end{tikzpicture}} 
%     \end{equation}
%     Therefore, $\en(\forg(\colcom_N(\wtd')))$ is positively graded. If $\alpha'$ is a degree 0 endormorphism, $\alpha$ too, so by simplicity of $\wtd$, $\alpha' \sim \lambda \Id$, with $\lambda \in \scalars\setminus \{0\}$. Thus, $\en(\forg(\colcom_N(\wtd')))$ is free of rank 1. 
%     \end{proof}

\begin{dfn}
    A \textit{pairing} is a web-tangle diagram that has no crossings, trivalent vertices, or closed components.
\end{dfn}

\begin{lem}
    \label{pairing}
    If $\wtd$ is a pairing then it is $\simple$.
\end{lem}

To prove this lemma, we first need to introduce some canopoly operations that will be used in the proof.
For $\bnd = (\bnd_1, \ldots, \bnd_{n})$ a finite sequence of elements of $\Z^*$, we denote by $-\Bar{\bnd} := (-\bnd_{n},\ldots,-\bnd_{1})$ the reversed sequence with each term multiplied by~$-1$. Let $\cpd_\bnd$ be the two-input planar diagram whose two input disks have the boundary $\bnd$ and $-\Bar{\bnd}$ and which connects, by colored arcs with the appropriate colors, the corresponding boundary points.

\begin{figure}[h]
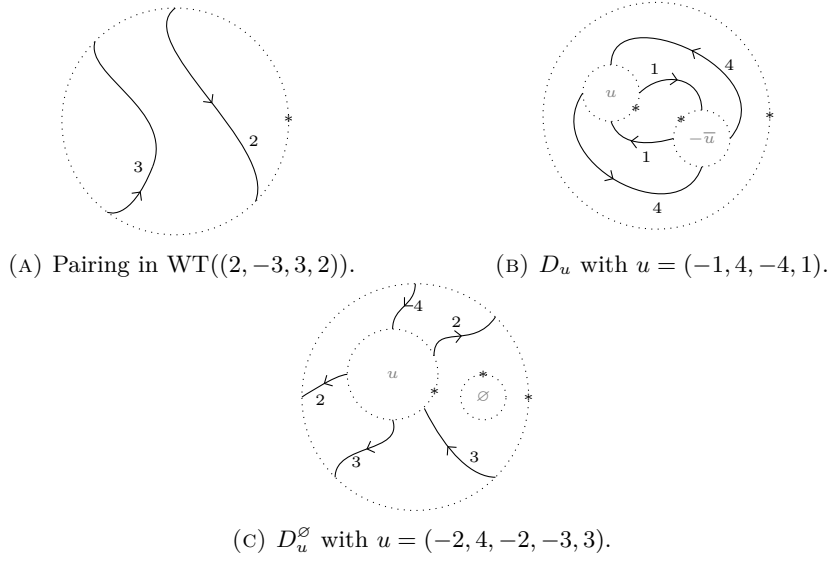

\begin{subfigure}{0.4\textwidth}
    \centering
    \NB{\tikz[scale=1.5]{\input{\imagesfolder/sch/ex_pairing}}}
    \caption{Pairing in $\WT((2,-3,3,2))$.}
\end{subfigure}
\begin{subfigure}{0.4\textwidth}
    \centering
    \NB{\tikz[scale=1.5]{\input{\imagesfolder/sch/ex_cylinder}}}
    \caption{$\cpd_\bnd$ with $\bnd=(-1,4,-4,1)$.}
\end{subfigure}
\begin{subfigure}{0.4\textwidth}
    \centering
    \NB{\tikz[scale=1.5]{\input{\imagesfolder/sch/D_bnd_empty}}}
    \caption{$\cpd_\bnd^\varnothing$ with $\bnd = (-2,4,-2,-3,3)$.}
\end{subfigure}
\caption{Some examples}
\end{figure}

For a given web $\web$ in $\WT(\bnd)$, let $-\Bar{\web}$ in $\WT(-\Bar{\bnd})$ denote its mirror image with reverse orientation. We denote by $\mathrm{cup}(\web)$ the foam in $\morf(\varnothing, \cpd_\bnd(\web, -\Bar{\web})),$ obtained by rotating $\web$ by a half–turn.

For $\web, \web' \in \catf(\bnd)$, define the following canopoly operation:
\begin{equation}
    \label{eq:bend}
    \function{\bend}{\morf(\web,\web')}{\morf(\varnothing,\cpd_\bnd(\web',-\Bar{\web}))}{F}{\cpd_\bnd(F,\Id_{-\Bar{\web}}) \circ \mathrm{cup}(\web)} 
\end{equation}  
Informally, this operation is obtained by bending the entire boundary of the foam to the top of the cylinder.  
The inverse operation is obtained by bending this part of the boundary back down.  
This inverse operation can also be described using canopoly operations:
\begin{equation}
    \label{eq:bend-1}
    \function{\bend^{-1}}{\morf(\varnothing,\cpd_\bnd(\web',-\Bar{\web})) }{\morf(\web,\web')}{F}{\cpd_{\bnd}^{\varnothing}(\Id_{\web'},\mathrm{cup}(-\Bar{\web})) \circ \cpd_{\bnd}^{\varnothing}(F, \Id_\web)}
\end{equation}  
where $\cpd_\bnd^\varnothing$ is a colored planar arc diagram with boundary $\bnd$ and two input disks. One input disk has boundary $\bnd$, and the arcs are the identity between this disk and the boundary of the output disk. The second input disk has empty boundary.

\begin{rmk}
    % Using the previous notation and specifying that $\bnd = (\bnd_1, \ldots, \bnd_{n_\bnd})$, 
    % as a morphism of graded spaces,
    % \[\bend : \morf(\web,\web') \overset{\sim}{\longrightarrow} q^{\gamma}\morf(\varnothing,\cpd_\bnd(\web',-\Bar{\web})),\] 
    The morphism \(\bend\) has quantum degree $\gamma := \sum_i \frac{|s_i|(N-|s_i|)}{2}$. 
\end{rmk}

We can now present the proof of Lemma \ref{pairing}.

\begin{proof}
    If $\wtd$ is a pairing in $\WT(\bnd)$, then $\forg(\colcom_N(\wtd))$ is the one-term complex consisting of $\wtd$ in homological degree $0$.
    Thus, a degree $0$ endomorphism of this complex is a $\scalars$-linear combination of foams of degree $0$.
    Let $\foam$ be one of these foams. Then $\bend(\foam) \in \morf(\varnothing,\cpd_{\bnd}(T,-\Bar{T}))$ and has degree:
    \[-\sum_{i} \frac{|\bnd_i| (N-|\bnd_i|)}{2} = -\sum_{\bnd_i>0} \bnd_i(N-\bnd_i).\] 
    Because $\cpd_\bnd(T,-\Bar{T})$ is a collection of disjoint circles labeled by $\{\bnd_i \in \bnd \mid \bnd_i > 0\}$, as a graded module, the quantum dimension of $\morf(\varnothing,\cpd_\bnd(T,-\Bar{T}))$ is $\prod_{\bnd_i >0}\croch{N}{\bnd_i}.$
    In particular, because 
    \[ q^{k(N-k)}\croch{N}{k} \underset{q \rightarrow 0}{\longrightarrow} 1,\] 
    we know that, up to a scalar in $\scalars\setminus \{0\}$, there is a unique foam in the graded module $\morf(\varnothing,\cpd(T,-\Bar{T}))$ of lowest degree $-\sum_{\bnd_i>0} \bnd_i (N-\bnd_i)$, namely the collection of cups with the correct labeling. By applying $\bend^{-1}$, which is a degree $\sum_{\bnd_i>0} \bnd_i (N-\bnd_i)$ morphism, to this foam and to $\bend(\foam)$, we find that $\foam$ is the identity foam multiplied by a scalar in $\scalars\setminus \{0\}$. So $\wtd$ is $\simple$. 
\end{proof}

We do not provide a proof of Lemma \ref{lem_TX} because, using Remark~\ref{BN}, it is identical to the proofs of Lemmas~8.8 and~8.9 in \cite{Bar05}.

\begin{lem}[{\cite[Lem~8.8,~8.9]{Bar05}}]
    \label{lem_TX}
    Let $\wtd$ be a web-tangle diagram, and denote by $\wtd X$ the web-tangle diagram obtained from $\wtd$ by applying the canopoly operation that crosses two adjacent boundary components of $\wtd$ (cf Figure \ref{TX}). Then $\wtd$ is $\simple$ if and only if $\wtd X$ is $\simple$.
\end{lem}

\begin{lem}
    \label{lem_TY}
    Let $\wtd$ be a web-tangle diagram, and denote by $\wtd Y$ a web-tangle diagram obtained from $\wtd$ after splitting an external string of $\wtd$ (cf Figure \ref{TX}). If $\wtd$ is $\simple$ then $\wtd Y$ is $\simple$ too.
\end{lem}

\begin{proof}
    For every web $\web$ in the complex associated with $\wtd Y$ or $\wtd$, the edge of thickness $a+b$ induces an $\scalars_{a+b}$-module structure on $\morf(\web)$, where $\scalars_{a+b}$ is identified with $Sym_{a+b}$. Lemma~\ref{lemma:blist1} induces the following isomorphism of graded $\scalars_{a+b}$-modules: 
    \begin{equation}
    \morf(\wtd Y) \simeq \morf(\wtd) \otimes_{\scalars_{a+b}} \scalars_{(a,b)},
    \end{equation}
    where $\scalars_{(a,b)} := \scalars_a \otimes_\scalars \scalars_b$. This concludes the proof.
\end{proof}

\begin{figure}[h]
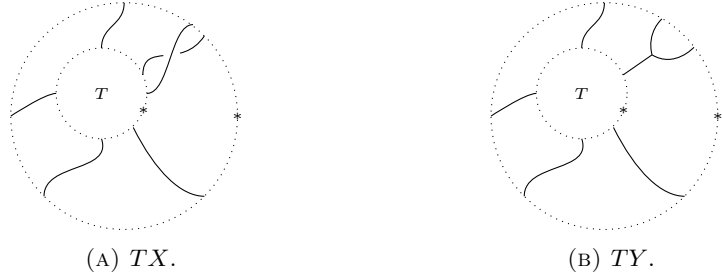

\begin{subfigure}{0.4\textwidth}
    \centering
    \NB{\tikz[scale=1.5]{\input{\imagesfolder/sch/TX}}}
    \caption{$TX$.}
\end{subfigure}
\begin{subfigure}{0.4\textwidth}
    \centering
    \NB{\tikz[scale=1.5]{\input{\imagesfolder/sch/TY}}}
    \caption{$TY$.}
\end{subfigure}
\caption{Some examples of webs that appearing in Lemmas \ref{lem_TX} and \ref{lem_TY}.}
\label{TX}
\end{figure}

\begin{rmk}
    The operation of splitting an external string colored by $a+b$ of $T$ into a string colored by $a$ and another colored by $b$ can be described by canopoly operations. Indeed, this operation consists in gluing the web-tangle 
    \begin{equation}
        \web_{(a,b)} := \NB{\tikz[scale=.75,font=\tiny]{\input{\imagesfolder/sch/splitting_ab}}}
    \end{equation} 
    to the string in question via the appropriate colored planar diagram denoted $\cpd_{(a,b)}$.
\end{rmk}

\begin{lem}
    \label{lemma:blist1}
    Let $\web$ and $\web'$ be webs in $\catf(\bnd)$. Let $\web Y$ and $\web' Y$ be the webs obtained from $\web$ and $\web'$ by splitting the same external edge labeled $a+b$ into an edge labeled $a$ and another labeled $b$. Let $\foam$ be a foam in $\morf(\web Y,\web' Y)$ then, 
    \begin{equation}
        \foam = \sum_{\alpha \in T(a,b)}(-1)^{|\hat{\alpha}|} \cpd_{(a,b)}(\foam_\alpha,m(s_\alpha)),
    \end{equation}
    where $\foam_\alpha$ is a foam in $\morf(\web,\web')$, $s_\alpha$ is the Schur polynomial associated to the partition $\alpha$, $m(s_\alpha)$ is the identity foam of $\web_{(a,b)}$ decorated by $s_\alpha \in Sym_{a}$ on its facet of thickness $a$.
\end{lem}

\begin{rmk}
    Schematically, the Lemma \ref{lemma:blist1} says that we can always decompose a foam $\foam$ in $\morf(\web Y,\web'Y)$ as a sum of foams of the following type: 
    \begin{equation}
        \NB{\tikz[scale=.9]{\input{\imagesfolder/sch/blist_exp}}}
    \end{equation}
    The definition of the $s_\alpha$ polynomials is not relevant to the proof of Lemma~\ref{lem_TY}. The curious reader can find details about these polynomials in \cite{RW20a}.
\end{rmk}

\begin{proof}
    We use the canopoly operation~\eqref{eq:bend} to define a foam $\bend(\foam)$ in $\morf(\varnothing,\cpd_{u'}(\web Y,-\Bar{\web' Y}))$. The web on its boundary, $\cpd_{u'}(\web Y,-\Bar{\web'Y})$, contains the following digon, formed by gluing the two blisted edges:
    \begin{equation}
        d_{(a,b)} := \NB{\tikz[scale=.6,font=\tiny]{\input{\imagesfolder/sch/web_digon_k}}}.
    \end{equation}
    There is a neighborhood of $d_{(a,b)}$ in $\foam$ such that the restriction of $\foam$ to this neighborhood is the identity of $d_{(a,b)}$. Using a local relation %relation (12)
    from the Robert--Wagner evaluation~\cite{RW20a}, which we recall here:
    \begin{equation}
    \label{eq:12}
    \NB{\tikz[scale=.7,font=\tiny]{\input{\imagesfolder/sch/digon_id}}} = \sum_{\alpha \in T(a,b)} (-1)^{|\Hat{\alpha}|} \NB{\tikz[scale=.7,font=\tiny]{\input{\imagesfolder/sch/digon_rel}}},
    \end{equation}
    we obtain that $\foam = \sum (-1)^{|\Hat{\alpha}|} G_\alpha$ where $G_\alpha$ is the same foam as $F$, except in the neighborhood of the digon, where it is modified according to the $\alpha$th term in the sum of relation~\eqref{eq:12}. Therefore, 
    \begin{equation}
        \foam = \bend^{-1} (\bend(\foam)) = \sum_{\alpha \in T(a,b)} (-1)^{|\Hat{\alpha}|} \bend^{-1}(G_\alpha),
    \end{equation}
    and as suggested in the Equation \eqref{eq:12}, we can express $\bend^{-1}(G_\alpha) = \cpd_{(a,b)}(F_\alpha,m(s_\alpha))$ where $F_\alpha$ is a foam in $\morf(\web,\web')$, $m(s_\alpha)$ is the identity foam of $\web_{(a,b)}$ decorated by $s_\alpha \in Sym_{a}$ on its facet of thickness $a$. \qedhere
\end{proof}

\begin{rmk}
In the case where $a=k-1$ and $b=1$, the identity of Lemma \ref{lemma:blist1} becomes:
 \begin{equation}
        \foam = \sum_{i=0}^{k-1}(-1)^{k-1-i} \cpd_{(k-1,1)}(\foam_i,m(e_i)). 
    \end{equation}
\end{rmk}

%% file: sch/splitting_ab.tex
\begin{scope}
    \draw[dashed, very thin](0,0) circle (1);
    \draw[->] (0,0) -- (0,1) node[pos=.25,left]{$a+b$};
    \draw[] (-.35,-.94) -- (0,0) node[pos=.5,left]{$a$};
    \draw[] (.35,-.94) -- (0,0) node[pos=.5,right]{$b$};
\end{scope}

%% file: sch/blist_exp.tex
\begin{scope}[font=\tiny]
\fill[gray!30] (0.2,-.01) rectangle (1.8,3.01);
\draw[dashed,fill=gray!30] 
  (0.2,3) arc[start angle=180,end angle=360,x radius=.8,y radius=0.18];
\draw[dashed,fill=gray!30] 
  (0.2,0) arc[start angle=180,end angle=360,x radius=.8,y radius=0.18];
\draw[dashed,fill=gray!30] 
  (0.2,3) arc[start angle=180,end angle=0,x radius=.8,y radius=0.18];
\draw[dashed,fill=gray!30] 
  (0.2,0) arc[start angle=180,end angle=0,x radius=.8,y radius=0.18];
\draw[dashed] (0.2,0)--(0.2,3);
\draw[dashed] (1.8,0)--(1.8,3);

\draw[dotted] 
  (0,3) arc[start angle=180,end angle=360,x radius=1.6,y radius=0.33];
\draw[dotted] 
  (0,0) arc[start angle=180,end angle=360,x radius=1.6,y radius=0.33];
\draw[dotted] 
  (0,3) arc[start angle=180,end angle=0,x radius=1.6,y radius=0.33];
\draw[dotted] 
  (0,0) arc[start angle=180,end angle=0,x radius=1.6,y radius=0.33];
\draw[dotted] (0,0)--(0,3);
\draw[dotted] (3.2,0)--(3.2,3);

\node at (1,0) {$\web$};
\node at (1,3) {$\web'$};
\node at (2.75,2.9) {$\bullet s_\alpha$};
\node at (3,3.4) {$a$};
\node at (2,3.2) {$a+b$};
\node at (2.75,-.4) {$b$};
\node at (1,1.5) {$\foam_\alpha$};

\draw[] (1.8,0) -- (2.2,0);
\draw[] (2.2,0) -- (3,0.16);
\draw[] (2.2,0) -- (2.7,-0.23);

\draw[] (1.8,3) -- (2.2,3);
\draw[] (2.2,3) -- (3,3.16);
\draw[] (2.2,3) -- (2.7,2.77);

\draw[] (2.7,-0.23) -- (2.7,2.77);
\draw[] (2.2,0) -- (2.2,3);
\draw[] (3,0.16) -- (3,3.16);

\end{scope}

%% file: sch/web_digon_k.tex
\begin{scope}
  \coordinate (bm) at ( 0, -1);
  \coordinate (cm) at ( 0, -0.6);
  \coordinate (sm) at ( 0,  0.6);
  \coordinate (tm) at ( 0, 1);
  \draw[dashed] (0,0) circle (1);
  \draw[->] (sm) -- (tm) node[above, pos =1] {$a+b$};
  \draw[->-] (cm) .. controls + ( 0.4, 0.4) and + ( 0.4, -0.4) .. (sm)
  node [pos = 0.5, right] {$b$}  ;
  \draw[->-] (cm) .. controls + (-0.4, 0.4) and + (-0.4, -0.4) .. (sm)
  node [pos = 0.5,  left] {$a$};
  \draw[] (bm) -- (cm) node[pos=0, below] {$a+b$};
\end{scope}

%% file: sch/digon_id.tex
\tdplotsetmaincoords{60}{140}
\begin{scope}[tdplot_main_coords]
  \coordinate (aT) at (-1, 0, 2);
  \coordinate (bT) at (+1, 0, 2);
  \coordinate (AT) at (-2.5, 0, 2);
  \coordinate (BT) at (+2.5, 0, 2);
  \coordinate (cT) at (0, +1, 2);
  \coordinate (dT) at (0, -1, 2);
  \coordinate (aB) at (-1, 0, -2);
  \coordinate (bB) at (+1, 0, -2);
  \coordinate (AB) at (-2.5, 0, -2);
  \coordinate (BB) at (+2.5, 0, -2);
  \coordinate (cB) at (0, +1, -2);
  \coordinate (dB) at (0, -1, -2);
  \coordinate (aM) at (-1, 0, 0);
  \coordinate (bM) at (+1, 0, 0);
\draw[thick ] (AT) -- (aT);
\draw[thick ] (bT) -- (BT);
\draw[thick ] (aT) .. controls (dT) .. (bT);
\draw[thick ] (aT) .. controls (cT) .. (bT);
\draw[thick ] (AB) -- (aB);
\draw[thick ] (bB) -- (BB);
\draw[thick ] (aB) .. controls (dB) .. (bB);
\draw[thick ] (aB) .. controls (cB) .. (bB);
\draw[thick] (aB) -- (aT);
\draw[thick] (bT) -- (bB);
  \filldraw[draw = black, rounded corners=1pt, thick, fill opacity = 0.3, fill = white]    (AT) -- (aT) node[midway, opacity =1, below, sloped] {$a+b$} -- (aB) -- (AB) -- (AT);
  \filldraw[draw = black, rounded corners=1pt, thick, fill opacity = 0.3, fill = white]    (bT) -- (BT) node[midway, opacity =1, below, sloped] {$a$} -- (BB) -- (bB) -- (bT);
  \filldraw[draw = black, rounded corners=1pt, thick, fill opacity = 0.3, fill = white]   (bT) .. controls (dT) .. (aT) node[near end, below, opacity=1] {$b$} -- (aB) .. controls (dB) .. (bB)-- (bT);
  \filldraw[draw = black, rounded corners=1pt, thick, fill opacity = 0.3, fill = white]   (bT) .. controls (cT) .. (aT) -- (aB)  .. controls (cB) .. (bB) node[near end, above, opacity=1] {$b$} -- (bT);
\end{scope}

%% file: sch/digon_rel.tex
\tdplotsetmaincoords{60}{140}
\begin{scope}[tdplot_main_coords]
  \coordinate (aT) at (-1, 0, 2);
  \coordinate (bT) at (+1, 0, 2);
  \coordinate (AT) at (-2.5, 0, 2);
  \coordinate (BT) at (+2.5, 0, 2);
  \coordinate (cT) at (0, +1, 2);
  \coordinate (dT) at (0, -1, 2);
  \coordinate (aB) at (-1, 0, -2);
  \coordinate (bB) at (+1, 0, -2);
  \coordinate (AB) at (-2.5, 0, -2);
  \coordinate (BB) at (+2.5, 0, -2);
  \coordinate (cB) at (0, +1, -2);
  \coordinate (dB) at (0, -1, -2);
  \coordinate (aM) at (-1, 0, 0);
  \coordinate (bM) at (+1, 0, 0);
\draw[thick] (AT) -- (aT);
\draw[thick] (bT) -- (BT);
\draw[thick] (aT) .. controls (dT) .. (bT);
\draw[thick] (aT) .. controls (cT) .. (bT);
\draw[thick] (AB) -- (aB);
\draw[thick] (bB) -- (BB);
\draw[thick] (aB) .. controls (dB) .. (bB);
\draw[thick] (aB) .. controls (cB) .. (bB);
\draw[thick] (bT) .. controls (bM) and (aM) .. (aT);
\draw[thick] (aB) .. controls (aM) and (bM) .. (bB);
  \filldraw[draw = black, rounded corners=1pt, thick, fill opacity = 0.3, fill = white]    (AT) -- (aT)  .. controls (aM) and (bM) .. (bT) -- (BT) -- (BB) node[midway, opacity =1, right, sloped, rotate= 112 ] {$a+b$} -- (bB) .. controls (bM) and (aM) .. (aB) -- (AB) -- (AT);
  \filldraw[draw = black, rounded corners=1pt, thick, fill opacity = 0.3, fill = white]    (aT) .. controls (aM) and (bM) .. (bT) .. controls (dT) .. (aT) node [opacity=1, near end, below,yshift=-2pt] {$\bullet s_\alpha$} coordinate[pos=0.3] (alpha); 
  \filldraw[draw = black, rounded corners=1pt, thick, fill opacity = 0.3, fill = white]   (aT) .. controls (aM) and (bM) .. (bT) .. controls (cT) .. (aT) node [opacity=1, near start, below] {$b$} ;
  \filldraw[draw = black, rounded corners=1pt, thick, fill opacity = 0.3, fill = white]    (aB) .. controls (aM) and (bM) .. (bB) .. controls (dB) .. (aB) node [opacity=1, near end, above] {$a$};
  \filldraw[draw = black, rounded corners=1pt, thick, fill opacity = 0.3, fill = white]   (aB) .. controls (aM) and (bM) .. (bB) .. controls (cB) .. (aB) node [opacity=1, pos=.08, above]  {$b$} coordinate[pos=0.7] (alphah) node [opacity=1, pos=.3,yshift=2pt, above] {$\bullet s_{\Hat{\alpha}}$} coordinate[pos=0.3] (alpha); ;
\end{scope}

%% file: Z_blist.tex
The purpose of this section is to reduce the proof of the properties of $\colcom_N$ to the uncolored case.  
The idea is to “separate” a colored $k$–strand of a diagram into a colored $(k-1)$–strand and a $1$–colored strand via the blisting procedure:

\begin{figure}[h!]
    \centering
    \begin{tikzpicture}
    \node (a) at (0,0) {$\web_k$ := \NB{\tikz[scale=.8, font=\tiny]{\input{\imagesfolder/sch/blist_1}}} }; 
    \node (b) at (4.5,0) {\NB{\tikz[scale=.8, font=\tiny]{\input{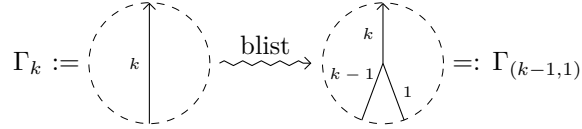}}} =: $\web_{(k-1,1)}$}; 
    \draw[->,line join=round,decorate, decoration={zigzag,segment length=7,amplitude=.9,post=lineto,post length=2pt}] (a) -- (b) node[pos=.5, above]{blist}; 
    \end{tikzpicture}
    \caption{Example of a blisting.}
    \label{ex:blisting}
\end{figure}

This procedure can be described by canopoly operations. Indeed, blistering a colored $k$-strand of a diagram consists in gluing the web-tangle $\web_{(k-1,1)}$ to the strand in question via the appropriate colored planar diagram denoted $\cpd_{(k-1,1)}$.

\begin{prop}[Blist hypothesis]
    \label{blisthyp}
    Let $\wtd, \wtd' \in \WT(\bnd)$. Let $\wtd_b$ and $\wtd_b' \in \WT(\bnd')$ be the diagrams obtained from $\wtd$ and $\wtd'$ by blisting the same boundary component.
    \begin{enumerate}
        \item $\wtd$ is $\simple$
        \item $\colcom_N(\wtd_b) \simeq \colcom_N(\wtd_b')$ in $\relcat{\bnd'}$.
    \end{enumerate}
    Then, 
    \begin{enumerate}
        \item $\wtd_b$ is $\simple$
        \item $\colcom_N(\wtd) \simeq \colcom_N(\wtd')$ in $\relcat{\bnd}$.  
    \end{enumerate}
\end{prop}

\begin{proof}
    The first statement is a special case of Lemma~\ref{lem_TY}. 
    For the second one, let us denote by $\phi$ the isomorphism between $\colcom_N(\wtd_b)$ and $\colcom_N(\wtd_b')$ in $\relcat{\bnd}$. This isomorphism descends to an isomorphism between $\forg(\colcom_N(\wtd_b))$ and $\forg(\colcom_N(\wtd_b'))$ in $\mathcal{C}(\catf(\bnd))$, which we will also denote by $\phi$. Because of the construction of $\colcom_N$, each foams involved in the definition of $\phi$, denoted $\foam_i$, is a foam between webs that have the $\web_{(k-1,1)}$ part. Each of these foams can be decomposed according Lemma~\ref{lemma:blist1}:
    \begin{equation}
        \foam_{i} = \sum_{j=0}^{k-1} (-1)^{k-1-j}\cpd_{(k-1,1)}(\foam_{i,j},m(e_j)).
    \end{equation}
    Note that the degree of $\foam_{i,j}$ is $-2j$. Thank to this decomposition, we can define:
    \begin{equation}
        \psi: \forg(\colcom_N(\wtd)) \xrightarrow{\sim} \forg(\colcom_N(\wtd')) \quad \text{by} \quad \psi := \sum_i (-1)^{k-1}\foam_{i,0}.
    \end{equation} 
    We denote by $G_\ell$ the foams defining the inverse of $\phi$. Then, using the decomposition from Lemma~\ref{lemma:blist1}, we have
    \begin{equation}
        G_{\ell} = \sum_{r=0}^{k-1} (-1)^{k-1-r}\cpd_{(k-1,1)}(G_{\ell,r},m(e_r)).
    \end{equation} 
    The composition of $G_{\ell,r}$ with $F_{i,j}$, whenever it is defined, gives an endomorphism of $\colcom_N(T)$ of degree $-2(j+r)$. By the simplicity of $T$, this composition is nonzero if and only if $j=r=0$. Therefore, the map $\psi$ is an isomorphism. We need to check that this map is $\Witt$-equivariant.\newline
    Thanks to the previous considerations, $\cpd_{(k-1,1)}(\psi,m(e_0))$ and $\phi$ are both isomorphisms between $\colcom_N(\wtd_b)$ and
    $\colcom_N(\wtd_b')$. So $\phi^{-1} \circ (\cpd_{(k-1,1)}(\psi,m(e_0)))$ is an endomorphism of $\colcom_N(\wtd_b)$ of degree 0, therefore thanks to the first statement, $\cpd_{(k-1,1)}(\psi,m(e_0)) = \alpha \phi$ with $\alpha \in \scalars \setminus \{0\}$. Hence, 
    \[\cpd_{(k-1,1)} ( (\wgen_n \cdot \psi) , m(e_0) )= \wgen_n \cdot (\cpd_{(k-1,1)}(\psi,m(e_0)))= \wgen_n \cdot (\alpha \phi)=0,\]
    and because the canopoly operation $\cpd_{(k-1,1)} ( ~\cdot~ , m(e_0) )$ has a right inverse, we find that $\psi$ is indeed $\Witt$-equivariant.
\end{proof}

\begin{rmk}
    The blist procedure can be applied inductively to split a $k$–colored strand into $k$ $1$–colored strands.  
    We refer to this procedure as the \textit{full blist}.

\[
\begin{tikzpicture}
    \node (a) at (0,0) {\NB{\tikz[scale=1, font=\tiny]{\input{\imagesfolder/sch/blist_1}}} }; 
    \node (b) at (4.5,0) {\NB{\tikz[scale=1, font=\tiny]{\input{\imagesfolder/sch/blist_3}}}}; 
    \draw[->,line join=round,decorate, decoration={zigzag,segment length=7,amplitude=.9,post=lineto,post length=2pt}] (a) -- (b) node[pos=.5, above]{full blist}; 
\end{tikzpicture}
\]
\end{rmk}

\begin{rmk}
    Using the notation and the hypothesis of \ref{blisthyp}, we prove that there is the following commutative square:
    \begin{equation*}
    \NB{
    \begin{tikzpicture}
        \node (a) at (0,0) {$\colcom_N(\wtd)$};
        \node (b) at (4,0) {$\colcom_N(\wtd')$};
        \node (c) at (0,-1.5) {$\colcom_N(\wtd_b)$};
        \node (d) at (4,-1.5) {$\colcom_N(\wtd_b')$};
        \draw[->,dashed] (a) -- (b) node[pos=0.5,above]{$\sim$};
        \draw[->] (b) -- (d) node[pos=0.5,right]{blisting};
        \draw[->] (a) -- (c) node[pos=.5,left]{blisting};
        \draw[->] (c) -- (d) node[pos=.5,above]{$\sim$};
    \end{tikzpicture}}
    \end{equation*}
    where the dotted arrow is the isomorphism given by the Proposition \ref{blisthyp}.
\end{rmk}

%% file: sch/blist_1.tex
\begin{scope}
    \draw[dashed, very thin](0,0) circle (1);
    \draw[->] (0,-1) -- (0,1) node[pos=.5,left]{$k$};
\end{scope}

%% file: Z_forkslide.tex
In \cite{QRSW21}, the authors define the forkslide moves. 
The goal of the following two sections is to prove that $\colcom_N$ is invariant under these forkslide moves:

\begin{prop}
\label{forkslide}    
Let $a,b,c \in \N^*$. We have isomorphisms, in the relative homotopy canopoly, between the colored Khovanov--Rozansky homology of the following colored web-tangles: 
\begin{align}
     \NB{\tikz[font=\tiny,xscale=1,yscale=.7]{\input{\imagesfolder/sch/forkslide_1}}}  \simeq \NB{\tikz[font=\tiny,scale=1,yscale=.7]{\input{\imagesfolder/sch/forkslide_2}}} & \hspace{2cm}\NB{\tikz[font=\tiny,xscale=1,yscale=.7]{\input{\imagesfolder/sch/forkslide_3}}}  \simeq \NB{\tikz[font=\tiny,scale=1,yscale=.7]{\input{\imagesfolder/sch/forkslide_4}}} \\ \nonumber
     ~ \\ 
     \NB{\tikz[font=\tiny,xscale=1,yscale=.7]{\input{\imagesfolder/sch/forkslide_9}}}  \simeq \NB{\tikz[font=\tiny,scale=1,yscale=.7]{\input{\imagesfolder/sch/forkslide_10}}} & \hspace{2cm} \NB{\tikz[font=\tiny,xscale=1,yscale=.7]{\input{\imagesfolder/sch/forkslide_11}}} \simeq\NB{\tikz[font=\tiny,scale=1,yscale=.7]{\input{\imagesfolder/sch/forkslide_12}}} \\ \nonumber
     ~ \\
     \NB{\tikz[font=\tiny,xscale=1,yscale=.7]{\input{\imagesfolder/sch/forkslide_5}}}  \simeq \NB{\tikz[font=\tiny,scale=1,yscale=.7]{\input{\imagesfolder/sch/forkslide_6}}} & \hspace{2cm}\NB{\tikz[font=\tiny,xscale=1,yscale=.7]{\input{\imagesfolder/sch/forkslide_7}}} \simeq\NB{\tikz[font=\tiny,scale=1,yscale=.7]{\input{\imagesfolder/sch/forkslide_8}}} \\ \nonumber
     ~ \\
     \NB{\tikz[font=\tiny,xscale=1,yscale=.7]{\input{\imagesfolder/sch/forkslide_13}}} \simeq\NB{\tikz[font=\tiny,scale=1,yscale=.7]{\input{\imagesfolder/sch/forkslide_14}}} & \hspace{2cm}\NB{\tikz[font=\tiny,xscale=1,yscale=.7]{\input{\imagesfolder/sch/forkslide_15}}} \simeq\NB{\tikz[font=\tiny,scale=1,yscale=.7]{\input{\imagesfolder/sch/forkslide_16}}}.
\end{align}
\end{prop}

\begin{lem}
    \label{redforkslide}
   Let $b, c \in \N^*$. To check the invariance of $\colcom_N$ under the forkslide moves, it suffices to prove that we have isomorphisms, in the relative homotopy canopoly, between the colored Khovanov--Rozansky homology of the following colored web-tangles:
\begin{align}
    \label{line1a}
     \NB{\tikz[font=\tiny,xscale=1,yscale=.7]{\input{\imagesfolder/sch/forkslide_1a}}}  \simeq \NB{\tikz[font=\tiny,scale=1,yscale=.7]{\input{\imagesfolder/sch/forkslide_2a}}} & \hspace{2cm}\NB{\tikz[font=\tiny,xscale=1,yscale=.7]{\input{\imagesfolder/sch/forkslide_3a}}}  \simeq \NB{\tikz[font=\tiny,scale=1,yscale=.7]{\input{\imagesfolder/sch/forkslide_4a}}} \\ \nonumber
     ~ \\ 
     \label{line2a}
     \NB{\tikz[font=\tiny,xscale=1,yscale=.7]{\input{\imagesfolder/sch/forkslide_9a}}}  \simeq \NB{\tikz[font=\tiny,scale=1,yscale=.7]{\input{\imagesfolder/sch/forkslide_10a}}} & \hspace{2cm} \NB{\tikz[font=\tiny,xscale=1,yscale=.7]{\input{\imagesfolder/sch/forkslide_11a}}} \simeq\NB{\tikz[font=\tiny,scale=1,yscale=.7]{\input{\imagesfolder/sch/forkslide_12a}}} \\ \nonumber
     ~ \\
     \label{line3a}
     \NB{\tikz[font=\tiny,xscale=1,yscale=.7]{\input{\imagesfolder/sch/forkslide_5a}}}  \simeq\NB{\tikz[font=\tiny,scale=1,yscale=.7]{\input{\imagesfolder/sch/forkslide_6a}}} & \hspace{2cm}\NB{\tikz[font=\tiny,xscale=1,yscale=.7]{\input{\imagesfolder/sch/forkslide_7a}}} \simeq\NB{\tikz[font=\tiny,scale=1,yscale=.7]{\input{\imagesfolder/sch/forkslide_8a}}} \\ \nonumber
     ~ \\
     \label{line4a}
     \NB{\tikz[font=\tiny,xscale=1,yscale=.7]{\input{\imagesfolder/sch/forkslide_13a}}} \simeq\NB{\tikz[font=\tiny,scale=1,yscale=.7]{\input{\imagesfolder/sch/forkslide_14a}}} & \hspace{2cm}\NB{\tikz[font=\tiny,xscale=1,yscale=.7]{\input{\imagesfolder/sch/forkslide_15a}}} \simeq\NB{\tikz[font=\tiny,scale=1,yscale=.7]{\input{\imagesfolder/sch/forkslide_16a}}}.
\end{align}

\end{lem}  

\begin{proof}
    The proof of this lemma does not rely on the explicit definition of $\colcom_N$, but only on the fact that it satisfies the blist hypothesis of Proposition~\ref{blisthyp}.
    
    To prove that it suffices to consider the case $a = 1$ for all moves we verify the inductive step in the induction on $a$. We only detail the proof for the first move of (\ref{line1a}). We start by blisting the $a$-edge in the two sides of the move:
    \[ \NB{\tikz[font=\tiny,xscale=1,yscale=.7]{\input{\imagesfolder/sch/forkslide_1}}} \rightsquigarrow \NB{\tikz[font=\tiny,xscale=1,yscale=.7]{\input{\imagesfolder/sch/forkslide_red_1}}} \hspace{2mm} \text{and} \hspace{2mm} \NB{\tikz[font=\tiny,xscale=1,yscale=.7]{\input{\imagesfolder/sch/forkslide_2}}} \rightsquigarrow \NB{\tikz[font=\tiny,xscale=1,yscale=.7]{\input{\imagesfolder/sch/forkslide_red_2}}} .\]
    According to the blist hypothesis, it suffices to prove that $\colcom_N$ is invariant under the "blisted move". Thus, using the fact that associative foams induce isomorphisms and by the induction hypothesis, we obtain:
    \[\NB{\tikz[font=\tiny,xscale=.9,yscale=.6]{\input{\imagesfolder/sch/forkslide_red_1}}} \simeq\NB{\tikz[font=\tiny,xscale=.9,yscale=.7]{\input{\imagesfolder/sch/forkslide_red_3}}} \simeq\NB{\tikz[font=\tiny,xscale=.9,yscale=.7]{\input{\imagesfolder/sch/forkslide_red_4}}} \simeq\NB{\tikz[font=\tiny,xscale=.9,yscale=.7]{\input{\imagesfolder/sch/forkslide_red_5}}} \simeq\NB{\tikz[font=\tiny,xscale=.9,yscale=.7]{\input{\imagesfolder/sch/forkslide_red_6}}} \simeq\NB{\tikz[font=\tiny,xscale=.9,yscale=.7]{\input{\imagesfolder/sch/forkslide_red_2}}}.\]
    We can now stick to the case where $a=1$. 
    
    To prove that we can stick to the case where $c=1$ for the moves in (\ref{line2a}), (\ref{line3a}), and (\ref{line4a}), we use the same strategy and proceed by induction. Let us prove the reduction for the first move of (\ref{line2a}). To prove it, we need to use that $\colcom_N$ is invariant under the moves (\ref{line1a}): 
    \[\NB{\tikz[font=\tiny,xscale=.9,yscale=.7]{\input{\imagesfolder/sch/forkslide_red_7}}} \simeq\NB{\tikz[font=\tiny,xscale=.9,yscale=.7]{\input{\imagesfolder/sch/forkslide_red_8}}} \simeq\NB{\tikz[font=\tiny,xscale=.9,yscale=.7]{\input{\imagesfolder/sch/forkslide_red_9}}} \simeq\NB{\tikz[font=\tiny,xscale=.9,yscale=.7]{\input{\imagesfolder/sch/forkslide_red_10}}} \simeq\NB{\tikz[font=\tiny,xscale=.9,yscale=.7]{\input{\imagesfolder/sch/forkslide_red_11}}}. \]
    The idea used to prove the reduction of the remaining moves is the same. The only subtlety arises in the proof of (\ref{line3a}), where moves (\ref{line2a}) are required.
\end{proof}

%% file: sch/forkslide_2.tex
\begin{scope}[xscale=1, font=\tiny]
    \coordinate (a) at (0,0);
    \coordinate (b) at (1,0);
    \coordinate (m) at (.5,1);
    \coordinate (ab) at (.5,2);
    \coordinate (cb) at (1.5,0);
    \coordinate (ct) at (0,2);
    \draw [->] (cb) .. controls +(0,1.5) and + (0,-1) .. (ct) node[pos = 0, below] {$c$} node[pos = 1, above] {};
    \fill[white] (0.5, 1.3) circle (.7mm);
    \draw [] (a) .. controls +(0,0.5) and + (0, 0) .. (m) node[pos = 0, below] {$a$} ;
    \draw [] (b) .. controls +(0,0.5) and + (0, 0) .. (m) node[pos = 0, below] {$b$} ;
    \draw [->] (m) -- (ab) node[pos=1, above] {$a+b$};
\end{scope}

%% file: sch/forkslide_4.tex
\begin{scope}[xscale=1, font=\tiny]
    \coordinate (a) at (0,0);
    \coordinate (b) at (1,0);
    \coordinate (m) at (.5,1);
    \coordinate (ab) at (.5,2);
    \coordinate (cb) at (1.5,0);
    \coordinate (ct) at (0,2);
    \draw [] (a) .. controls +(0,0.5) and + (0, 0) .. (m) node[pos = 0, below] {$a$} ;
    \draw [] (b) .. controls +(0,0.5) and + (0, 0) .. (m) node[pos = 0, below] {$b$} ;
    \draw [->] (m) -- (ab) node[pos=1, above] {$a+b$};
    \fill[white] (0.5, 1.3) circle (.7mm);
    \draw [->] (cb) .. controls +(0,1.5) and + (0,-1) .. (ct) node[pos = 0, below] {$c$} node[pos = 1, above] {};
\end{scope}

%% file: sch/forkslide_9.tex
\begin{scope}[xscale=-1,font=\tiny]
    \coordinate (a) at (0,0);
    \coordinate (b) at (1,0);
    \coordinate (m) at (.5,1.3);
    \coordinate (ab) at (.5,2);
    \coordinate (cb) at (1.5,0);
    \coordinate (ct) at (0,2);
    \draw [->] (cb) .. controls +(0,1) and + (0,-2) .. (ct) node[pos = 0, below] {$c$} node[pos = 1, above] {};
    \fill[white] (0.85, 0.6) circle (.7mm);
    \fill[white] (0.28, 0.9) circle (.7mm);
    \draw [] (a) .. controls +(0,0.5) and + (0, 0) .. (m) node[pos = 0, below] {$a$} ;
    \draw [] (b) .. controls +(0,0.5) and + (0, 0) .. (m) node[pos = 0, below] {$b$} ;
    \draw [->] (m) -- (ab) node[pos=1,above] {$a+b$};
\end{scope}

%% file: sch/forkslide_10.tex
\begin{scope}[xscale=-1, font=\tiny]
    \coordinate (a) at (0,0);
    \coordinate (b) at (1,0);
    \coordinate (m) at (.5,1);
    \coordinate (ab) at (.5,2);
    \coordinate (cb) at (1.5,0);
    \coordinate (ct) at (0,2);
    \draw [->] (cb) .. controls +(0,1.5) and + (0,-1) .. (ct) node[pos = 0, below] {$c$} node[pos = 1, above] {};
    \fill[white] (0.5, 1.3) circle (.7mm);
    \draw [] (a) .. controls +(0,0.5) and + (0, 0) .. (m) node[pos = 0, below] {$a$} ;
    \draw [] (b) .. controls +(0,0.5) and + (0, 0) .. (m) node[pos = 0, below] {$b$} ;
    \draw [->] (m) -- (ab) node[pos=1, above] {$a+b$};
\end{scope}

%% file: sch/forkslide_11.tex
\begin{scope}[xscale=-1, font=\tiny]
    \coordinate (a) at (0,0);
    \coordinate (b) at (1,0);
    \coordinate (m) at (.5,1.3);
    \coordinate (ab) at (.5,2);
    \coordinate (cb) at (1.5,0);
    \coordinate (ct) at (0,2);
    \draw [] (a) .. controls +(0,0.5) and + (0, 0) .. (m) node[pos = 0, below] {$a$} ;
    \draw [] (b) .. controls +(0,0.5) and + (0, 0) .. (m) node[pos = 0, below] {$b$} ;
    \draw [->] (m) -- (ab) node[pos=1,above] {$a+b$};
    \fill[white] (0.85, 0.6) circle (.7mm);
    \fill[white] (0.28, 0.9) circle (.7mm);
    \draw [->] (cb) .. controls +(0,1) and + (0,-2) .. (ct) node[pos = 0, below] {$c$} node[pos = 1, above] {};
\end{scope}

%% file: sch/forkslide_12.tex
\begin{scope}[xscale=-1, font=\tiny]
    \coordinate (a) at (0,0);
    \coordinate (b) at (1,0);
    \coordinate (m) at (.5,1);
    \coordinate (ab) at (.5,2);
    \coordinate (cb) at (1.5,0);
    \coordinate (ct) at (0,2);
    \draw [] (a) .. controls +(0,0.5) and + (0, 0) .. (m) node[pos = 0, below] {$a$} ;
    \draw [] (b) .. controls +(0,0.5) and + (0, 0) .. (m) node[pos = 0, below] {$b$} ;
    \draw [->] (m) -- (ab) node[pos=1, above] {$a+b$};
    \fill[white] (0.5, 1.3) circle (.7mm);
    \draw [->] (cb) .. controls +(0,1.5) and + (0,-1) .. (ct) node[pos = 0, below] {$c$} node[pos = 1, above] {};
\end{scope}

%% file: sch/forkslide_5.tex
\begin{scope}[xscale=-1, yscale=-1, font=\tiny]
    \coordinate (a) at (0,0);
    \coordinate (b) at (1,0);
    \coordinate (m) at (.5,1.3);
    \coordinate (ab) at (.5,2);
    \coordinate (cb) at (1.5,0);
    \coordinate (ct) at (0,2);
    \draw [<-] (cb) .. controls +(0,1) and + (0,-2) .. (ct) node[pos = 0, above] {$c$} node[pos = 1, below] {};
    \fill[white] (0.85, 0.6) circle (.7mm);
    \fill[white] (0.28, 0.9) circle (.7mm);
    \draw [<-] (a) .. controls +(0,0.5) and + (0, 0) .. (m) node[pos = 0, above] {$a$} ;
    \draw [<-] (b) .. controls +(0,0.5) and + (0, 0) .. (m) node[pos = 0, above] {$b$} ;
    \draw [] (m) -- (ab) node[pos=1,below] {$a+b$};
\end{scope}

%% file: sch/forkslide_6.tex
\begin{scope}[xscale=-1, yscale=-1, font=\tiny]
    \coordinate (a) at (0,0);
    \coordinate (b) at (1,0);
    \coordinate (m) at (.5,1);
    \coordinate (ab) at (.5,2);
    \coordinate (cb) at (1.5,0);
    \coordinate (ct) at (0,2);
    \draw [<-] (cb) .. controls +(0,1.5) and + (0,-1) .. (ct) node[pos = 0, above] {$c$} node[pos = 1, below] {};
    \fill[white] (0.5, 1.3) circle (.7mm);
    \draw [<-] (a) .. controls +(0,0.5) and + (0, 0) .. (m) node[pos = 0, above] {$a$} ;
    \draw [<-] (b) .. controls +(0,0.5) and + (0, 0) .. (m) node[pos = 0, above] {$b$} ;
    \draw [] (m) -- (ab) node[pos=1,below] {$a+b$};
\end{scope}

%% file: sch/forkslide_7.tex
\begin{scope}[xscale=-1, yscale=-1, font=\tiny]
    \coordinate (a) at (0,0);
    \coordinate (b) at (1,0);
    \coordinate (m) at (.5,1.3);
    \coordinate (ab) at (.5,2);
    \coordinate (cb) at (1.5,0);
    \coordinate (ct) at (0,2);
    \draw [<-] (a) .. controls +(0,0.5) and + (0, 0) .. (m) node[pos = 0, above] {$a$} ;
    \draw [<-] (b) .. controls +(0,0.5) and + (0, 0) .. (m) node[pos = 0, above] {$b$} ;
    \draw [] (m) -- (ab) node[pos=1,below] {$a+b$};
    \fill[white] (0.85, 0.6) circle (.7mm);
    \fill[white] (0.28, 0.9) circle (.7mm);
    \draw [<-] (cb) .. controls +(0,1) and + (0,-2) .. (ct) node[pos = 0, above] {$c$} node[pos = 1, below] {};
\end{scope}

%% file: sch/forkslide_8.tex
\begin{scope}[xscale=-1, yscale=-1, font=\tiny]
    \coordinate (a) at (0,0);
    \coordinate (b) at (1,0);
    \coordinate (m) at (.5,1);
    \coordinate (ab) at (.5,2);
    \coordinate (cb) at (1.5,0);
    \coordinate (ct) at (0,2);
    \draw [<-] (a) .. controls +(0,0.5) and + (0, 0) .. (m) node[pos = 0, above] {$a$} ;
    \draw [<-] (b) .. controls +(0,0.5) and + (0, 0) .. (m) node[pos = 0, above] {$b$} ;
    \draw [] (m) -- (ab) node[pos=1,below] {$a+b$};
    \fill[white] (0.5, 1.3) circle (.7mm);
    \draw [<-] (cb) .. controls +(0,1.5) and + (0,-1) .. (ct) node[pos = 0, above] {$c$} node[pos = 1, below] {};
\end{scope}

%% file: sch/forkslide_13.tex
\begin{scope}[xscale=1, yscale=-1, font=\tiny]
    \coordinate (a) at (0,0);
    \coordinate (b) at (1,0);
    \coordinate (m) at (.5,1.3);
    \coordinate (ab) at (.5,2);
    \coordinate (cb) at (1.5,0);
    \coordinate (ct) at (0,2);
    \draw [<-] (cb) .. controls +(0,1) and + (0,-2) .. (ct) node[pos = 0, above] {$c$} node[pos = 1, below] {};
    \fill[white] (0.85, 0.6) circle (.7mm);
    \fill[white] (0.28, 0.9) circle (.7mm);
    \draw [<-] (a) .. controls +(0,0.5) and + (0, 0) .. (m) node[pos = 0, above] {$a$} ;
    \draw [<-] (b) .. controls +(0,0.5) and + (0, 0) .. (m) node[pos = 0, above] {$b$} ;
    \draw [] (m) -- (ab) node[pos=1,below] {$a+b$};
\end{scope}

%% file: sch/forkslide_14.tex
\begin{scope}[xscale=1, yscale=-1, font=\tiny]
    \coordinate (a) at (0,0);
    \coordinate (b) at (1,0);
    \coordinate (m) at (.5,1);
    \coordinate (ab) at (.5,2);
    \coordinate (cb) at (1.5,0);
    \coordinate (ct) at (0,2);
    \draw [<-] (cb) .. controls +(0,1.5) and + (0,-1) .. (ct) node[pos = 0, above] {$c$} node[pos = 1, below] {};
    \fill[white] (0.5, 1.3) circle (.7mm);
    \draw [<-] (a) .. controls +(0,0.5) and + (0, 0) .. (m) node[pos = 0, above] {$a$} ;
    \draw [<-] (b) .. controls +(0,0.5) and + (0, 0) .. (m) node[pos = 0, above] {$b$} ;
    \draw [] (m) -- (ab) node[pos=1,below] {$a+b$};
\end{scope}

%% file: sch/forkslide_15.tex
\begin{scope}[yscale=-1, font=\tiny]
    \coordinate (a) at (0,0);
    \coordinate (b) at (1,0);
    \coordinate (m) at (.5,1.3);
    \coordinate (ab) at (.5,2);
    \coordinate (cb) at (1.5,0);
    \coordinate (ct) at (0,2);
    \draw [<-] (a) .. controls +(0,0.5) and + (0, 0) .. (m) node[pos = 0, above] {$a$} ;
    \draw [<-] (b) .. controls +(0,0.5) and + (0, 0) .. (m) node[pos = 0, above] {$b$} ;
    \draw [] (m) -- (ab) node[pos=1,below] {$a+b$};
    \fill[white] (0.85, 0.6) circle (.7mm);
    \fill[white] (0.28, 0.9) circle (.7mm);
    \draw [<-] (cb) .. controls +(0,1) and + (0,-2) .. (ct) node[pos = 0, above] {$c$} node[pos = 1, below] {};
\end{scope}

%% file: sch/forkslide_16.tex
\begin{scope}[xscale=1, yscale=-1, font=\tiny]
    \coordinate (a) at (0,0);
    \coordinate (b) at (1,0);
    \coordinate (m) at (.5,1);
    \coordinate (ab) at (.5,2);
    \coordinate (cb) at (1.5,0);
    \coordinate (ct) at (0,2);
    \draw [<-] (a) .. controls +(0,0.5) and + (0, 0) .. (m) node[pos = 0, above] {$a$} ;
    \draw [<-] (b) .. controls +(0,0.5) and + (0, 0) .. (m) node[pos = 0, above] {$b$} ;
    \draw [] (m) -- (ab) node[pos=1,below] {$a+b$};
    \fill[white] (0.5, 1.3) circle (.7mm);
    \draw [<-] (cb) .. controls +(0,1.5) and + (0,-1) .. (ct) node[pos = 0, above] {$c$} node[pos = 1, below] {};
\end{scope}

%% file: sch/forkslide_2a.tex
\begin{scope}[xscale=1, font=\tiny]
    \coordinate (a) at (0,0);
    \coordinate (b) at (1,0);
    \coordinate (m) at (.5,1);
    \coordinate (ab) at (.5,2);
    \coordinate (cb) at (1.5,0);
    \coordinate (ct) at (0,2);
    \draw [->] (cb) .. controls +(0,1.5) and + (0,-1) .. (ct) node[pos = 0, below] {$c$} node[pos = 1, above] {};
    \fill[white] (0.5, 1.3) circle (.7mm);
    \draw [] (a) .. controls +(0,0.5) and + (0, 0) .. (m) node[pos = 0, below] {$1$} ;
    \draw [] (b) .. controls +(0,0.5) and + (0, 0) .. (m) node[pos = 0, below] {$b$} ;
    \draw [->] (m) -- (ab) node[pos=1, above] {$b+1$};
\end{scope}

%% file: sch/forkslide_4a.tex
\begin{scope}[xscale=1, font=\tiny]
    \coordinate (a) at (0,0);
    \coordinate (b) at (1,0);
    \coordinate (m) at (.5,1);
    \coordinate (ab) at (.5,2);
    \coordinate (cb) at (1.5,0);
    \coordinate (ct) at (0,2);
    \draw [] (a) .. controls +(0,0.5) and + (0, 0) .. (m) node[pos = 0, below] {$1$} ;
    \draw [] (b) .. controls +(0,0.5) and + (0, 0) .. (m) node[pos = 0, below] {$b$} ;
    \draw [->] (m) -- (ab) node[pos=1, above] {$b+1$};
    \fill[white] (0.5, 1.3) circle (.7mm);
    \draw [->] (cb) .. controls +(0,1.5) and + (0,-1) .. (ct) node[pos = 0, below] {$c$} node[pos = 1, above] {};
\end{scope}

%% file: sch/forkslide_9a.tex
\begin{scope}[xscale=-1,font=\tiny]
    \coordinate (a) at (0,0);
    \coordinate (b) at (1,0);
    \coordinate (m) at (.5,1.3);
    \coordinate (ab) at (.5,2);
    \coordinate (cb) at (1.5,0);
    \coordinate (ct) at (0,2);
    \draw [->] (cb) .. controls +(0,1) and + (0,-2) .. (ct) node[pos = 0, below] {$1$} node[pos = 1, above] {};
    \fill[white] (0.85, 0.6) circle (.7mm);
    \fill[white] (0.28, 0.9) circle (.7mm);
    \draw [] (a) .. controls +(0,0.5) and + (0, 0) .. (m) node[pos = 0, below] {$1$} ;
    \draw [] (b) .. controls +(0,0.5) and + (0, 0) .. (m) node[pos = 0, below] {$b$} ;
    \draw [->] (m) -- (ab) node[pos=1,above] {$b+1$};
\end{scope}

%% file: sch/forkslide_10a.tex
\begin{scope}[xscale=-1, font=\tiny]
    \coordinate (a) at (0,0);
    \coordinate (b) at (1,0);
    \coordinate (m) at (.5,1);
    \coordinate (ab) at (.5,2);
    \coordinate (cb) at (1.5,0);
    \coordinate (ct) at (0,2);
    \draw [->] (cb) .. controls +(0,1.5) and + (0,-1) .. (ct) node[pos = 0, below] {$1$} node[pos = 1, above] {};
    \fill[white] (0.5, 1.3) circle (.7mm);
    \draw [] (a) .. controls +(0,0.5) and + (0, 0) .. (m) node[pos = 0, below] {$1$} ;
    \draw [] (b) .. controls +(0,0.5) and + (0, 0) .. (m) node[pos = 0, below] {$b$} ;
    \draw [->] (m) -- (ab) node[pos=1, above] {$b+1$};
\end{scope}

%% file: sch/forkslide_11a.tex
\begin{scope}[xscale=-1, font=\tiny]
    \coordinate (a) at (0,0);
    \coordinate (b) at (1,0);
    \coordinate (m) at (.5,1.3);
    \coordinate (ab) at (.5,2);
    \coordinate (cb) at (1.5,0);
    \coordinate (ct) at (0,2);
    \draw [] (a) .. controls +(0,0.5) and + (0, 0) .. (m) node[pos = 0, below] {$1$} ;
    \draw [] (b) .. controls +(0,0.5) and + (0, 0) .. (m) node[pos = 0, below] {$b$} ;
    \draw [->] (m) -- (ab) node[pos=1,above] {$b+1$};
    \fill[white] (0.85, 0.6) circle (.7mm);
    \fill[white] (0.28, 0.9) circle (.7mm);
    \draw [->] (cb) .. controls +(0,1) and + (0,-2) .. (ct) node[pos = 0, below] {$1$} node[pos = 1, above] {};
\end{scope}

%% file: sch/forkslide_12a.tex
\begin{scope}[xscale=-1, font=\tiny]
    \coordinate (a) at (0,0);
    \coordinate (b) at (1,0);
    \coordinate (m) at (.5,1);
    \coordinate (ab) at (.5,2);
    \coordinate (cb) at (1.5,0);
    \coordinate (ct) at (0,2);
    \draw [] (a) .. controls +(0,0.5) and + (0, 0) .. (m) node[pos = 0, below] {$1$} ;
    \draw [] (b) .. controls +(0,0.5) and + (0, 0) .. (m) node[pos = 0, below] {$b$} ;
    \draw [->] (m) -- (ab) node[pos=1, above] {$b+1$};
    \fill[white] (0.5, 1.3) circle (.7mm);
    \draw [->] (cb) .. controls +(0,1.5) and + (0,-1) .. (ct) node[pos = 0, below] {$1$} node[pos = 1, above] {};
\end{scope}

%% file: sch/forkslide_5a.tex
\begin{scope}[xscale=-1, yscale=-1, font=\tiny]
    \coordinate (a) at (0,0);
    \coordinate (b) at (1,0);
    \coordinate (m) at (.5,1.3);
    \coordinate (ab) at (.5,2);
    \coordinate (cb) at (1.5,0);
    \coordinate (ct) at (0,2);
    \draw [<-] (cb) .. controls +(0,1) and + (0,-2) .. (ct) node[pos = 0, above] {$1$} node[pos = 1, below] {};
    \fill[white] (0.85, 0.6) circle (.7mm);
    \fill[white] (0.28, 0.9) circle (.7mm);
    \draw [<-] (a) .. controls +(0,0.5) and + (0, 0) .. (m) node[pos = 0, above] {$1$} ;
    \draw [<-] (b) .. controls +(0,0.5) and + (0, 0) .. (m) node[pos = 0, above] {$b$} ;
    \draw [] (m) -- (ab) node[pos=1,below] {$b+1$};
\end{scope}

%% file: sch/forkslide_6a.tex
\begin{scope}[xscale=-1, yscale=-1, font=\tiny]
    \coordinate (a) at (0,0);
    \coordinate (b) at (1,0);
    \coordinate (m) at (.5,1);
    \coordinate (ab) at (.5,2);
    \coordinate (cb) at (1.5,0);
    \coordinate (ct) at (0,2);
    \draw [<-] (cb) .. controls +(0,1.5) and + (0,-1) .. (ct) node[pos = 0, above] {$1$} node[pos = 1, below] {};
    \fill[white] (0.5, 1.3) circle (.7mm);
    \draw [<-] (a) .. controls +(0,0.5) and + (0, 0) .. (m) node[pos = 0, above] {$1$} ;
    \draw [<-] (b) .. controls +(0,0.5) and + (0, 0) .. (m) node[pos = 0, above] {$b$} ;
    \draw [] (m) -- (ab) node[pos=1,below] {$b+1$};
\end{scope}

%% file: sch/forkslide_7a.tex
\begin{scope}[xscale=-1, yscale=-1, font=\tiny]
    \coordinate (a) at (0,0);
    \coordinate (b) at (1,0);
    \coordinate (m) at (.5,1.3);
    \coordinate (ab) at (.5,2);
    \coordinate (cb) at (1.5,0);
    \coordinate (ct) at (0,2);
    \draw [<-] (a) .. controls +(0,0.5) and + (0, 0) .. (m) node[pos = 0, above] {$1$} ;
    \draw [<-] (b) .. controls +(0,0.5) and + (0, 0) .. (m) node[pos = 0, above] {$b$} ;
    \draw [] (m) -- (ab) node[pos=1,below] {$b+1$};
    \fill[white] (0.85, 0.6) circle (.7mm);
    \fill[white] (0.28, 0.9) circle (.7mm);
    \draw [<-] (cb) .. controls +(0,1) and + (0,-2) .. (ct) node[pos = 0, above] {$1$} node[pos = 1, below] {};
\end{scope}

%% file: sch/forkslide_8a.tex
\begin{scope}[xscale=-1, yscale=-1, font=\tiny]
    \coordinate (a) at (0,0);
    \coordinate (b) at (1,0);
    \coordinate (m) at (.5,1);
    \coordinate (ab) at (.5,2);
    \coordinate (cb) at (1.5,0);
    \coordinate (ct) at (0,2);
    \draw [<-] (a) .. controls +(0,0.5) and + (0, 0) .. (m) node[pos = 0, above] {$1$} ;
    \draw [<-] (b) .. controls +(0,0.5) and + (0, 0) .. (m) node[pos = 0, above] {$b$} ;
    \draw [] (m) -- (ab) node[pos=1,below] {$b+1$};
    \fill[white] (0.5, 1.3) circle (.7mm);
    \draw [<-] (cb) .. controls +(0,1.5) and + (0,-1) .. (ct) node[pos = 0, above] {$1$} node[pos = 1, below] {};
\end{scope}

%% file: sch/forkslide_13a.tex
\begin{scope}[xscale=1, yscale=-1, font=\tiny]
    \coordinate (a) at (0,0);
    \coordinate (b) at (1,0);
    \coordinate (m) at (.5,1.3);
    \coordinate (ab) at (.5,2);
    \coordinate (cb) at (1.5,0);
    \coordinate (ct) at (0,2);
    \draw [<-] (cb) .. controls +(0,1) and + (0,-2) .. (ct) node[pos = 0, above] {$1$} node[pos = 1, below] {};
    \fill[white] (0.85, 0.6) circle (.7mm);
    \fill[white] (0.28, 0.9) circle (.7mm);
    \draw [<-] (a) .. controls +(0,0.5) and + (0, 0) .. (m) node[pos = 0, above] {$1$} ;
    \draw [<-] (b) .. controls +(0,0.5) and + (0, 0) .. (m) node[pos = 0, above] {$b$} ;
    \draw [] (m) -- (ab) node[pos=1,below] {$b+1$};
\end{scope}

%% file: sch/forkslide_14a.tex
\begin{scope}[xscale=1, yscale=-1, font=\tiny]
    \coordinate (a) at (0,0);
    \coordinate (b) at (1,0);
    \coordinate (m) at (.5,1);
    \coordinate (ab) at (.5,2);
    \coordinate (cb) at (1.5,0);
    \coordinate (ct) at (0,2);
    \draw [<-] (cb) .. controls +(0,1.5) and + (0,-1) .. (ct) node[pos = 0, above] {$1$} node[pos = 1, below] {};
    \fill[white] (0.5, 1.3) circle (.7mm);
    \draw [<-] (a) .. controls +(0,0.5) and + (0, 0) .. (m) node[pos = 0, above] {$1$} ;
    \draw [<-] (b) .. controls +(0,0.5) and + (0, 0) .. (m) node[pos = 0, above] {$b$} ;
    \draw [] (m) -- (ab) node[pos=1,below] {$b+1$};
\end{scope}

%% file: sch/forkslide_15a.tex
\begin{scope}[yscale=-1, font=\tiny]
    \coordinate (a) at (0,0);
    \coordinate (b) at (1,0);
    \coordinate (m) at (.5,1.3);
    \coordinate (ab) at (.5,2);
    \coordinate (cb) at (1.5,0);
    \coordinate (ct) at (0,2);
    \draw [<-] (a) .. controls +(0,0.5) and + (0, 0) .. (m) node[pos = 0, above] {$1$} ;
    \draw [<-] (b) .. controls +(0,0.5) and + (0, 0) .. (m) node[pos = 0, above] {$b$} ;
    \draw [] (m) -- (ab) node[pos=1,below] {$b+1$};
    \fill[white] (0.85, 0.6) circle (.7mm);
    \fill[white] (0.28, 0.9) circle (.7mm);
    \draw [<-] (cb) .. controls +(0,1) and + (0,-2) .. (ct) node[pos = 0, above] {$1$} node[pos = 1, below] {};
\end{scope}

%% file: sch/forkslide_16a.tex
\begin{scope}[xscale=1, yscale=-1, font=\tiny]
    \coordinate (a) at (0,0);
    \coordinate (b) at (1,0);
    \coordinate (m) at (.5,1);
    \coordinate (ab) at (.5,2);
    \coordinate (cb) at (1.5,0);
    \coordinate (ct) at (0,2);
    \draw [<-] (a) .. controls +(0,0.5) and + (0, 0) .. (m) node[pos = 0, above] {$1$} ;
    \draw [<-] (b) .. controls +(0,0.5) and + (0, 0) .. (m) node[pos = 0, above] {$b$} ;
    \draw [] (m) -- (ab) node[pos=1,below] {$b+1$};
    \fill[white] (0.5, 1.3) circle (.7mm);
    \draw [<-] (cb) .. controls +(0,1.5) and + (0,-1) .. (ct) node[pos = 0, above] {$1$} node[pos = 1, below] {};
\end{scope}

%% file: Z_forkiso.tex
Let us now prove that $\colcom_N$ is invariant under the reduced forkslide moves from Lemma \ref{redforkslide}.

\begin{lem}[{\cite[Lem~3.26]{QRSW21}}] 
    \label{nulhom}
    Let $a, b \in \mathbb{N}^*$, and set $m = \min(a,b)$. The two following complexes are exact in the relative homotopy canopoly:
        \[\cdots \rightarrow  H_{i-1}^- \rightarrow H_{i}^- \xrightarrow{d_i^-} H_{i+1}^- \rightarrow \cdots\]
        \[\cdots \rightarrow  H_{j-1}^+ \rightarrow H_{j}^+ \xrightarrow{d_j^+} H_{j+1}^+ \rightarrow \cdots\]
    where 
        \[H_i^- := \NB{\tikz[font=\fontsize{1.6mm}{0mm},scale=1.25]{\input{\imagesfolder/sch/nul_hom_1}}} \quad \text{and} \quad H_j^+ := \NB{\tikz[font=\fontsize{1.6mm}{0mm},scale=1.25]{\input{\imagesfolder/sch/nul_hom_2}}}.\]
    and the differentials are given by: 
\[
    d_i^-:=\mymoviefive[scale=.6]{\NB{\tikz[font=\fontsize{1.6mm}{0mm},scale=.55]{\input{\imagesfolder/sch/diff_neg_1_p}}}}{\NB{\tikz[font=\fontsize{1.6mm}{0mm},scale=.55]{\input{\imagesfolder/sch/diff_neg_2_p}}}}{\NB{\tikz[font=\fontsize{1.6mm}{0mm},scale=.55]{\input{\imagesfolder/sch/diff_neg_3_p}}}}{\NB{\tikz[font=\fontsize{1.6mm}{0mm},scale=.55]{\input{\imagesfolder/sch/diff_neg_4_p}}}}{\NB{\tikz[font=\fontsize{1.6mm}{0mm},scale=.55]{\input{\imagesfolder/sch/diff_neg_5_p}}}}
\]
\[
    d_j^+:=\mymoviefive[scale=.6]{\NB{\tikz[font=\fontsize{1.25mm}{0mm},scale=.55]{\input{\imagesfolder/sch/diff_pos_1_p}}}}{\NB{\tikz[font=\fontsize{1.25mm}{0mm},scale=.55]{\input{\imagesfolder/sch/diff_pos_2_p}}}}{\NB{\tikz[font=\fontsize{1.25mm}{0mm},scale=.55]{\input{\imagesfolder/sch/diff_pos_3_p}}}}{\NB{\tikz[font=\fontsize{1.25mm}{0mm},scale=.55]{\input{\imagesfolder/sch/diff_pos_4_p}}}}{\NB{\tikz[font=\fontsize{1.25mm}{0mm},scale=.55]{\input{\imagesfolder/sch/diff_pos_5_p}}}}
\]
\end{lem}

\begin{prop}
    There is an isormorphism in the relative homotopy canopoly,
    \[ \colcom_N\left ( \NB{\tikz[font=\tiny,xscale=1,yscale=.7]{\input{\imagesfolder/sch/forkslide_3a}}}  \right ) \cong \colcom_N \left ( \NB{\tikz[font=\tiny,scale=1,yscale=.7]{\input{\imagesfolder/sch/forkslide_4a}}} \right ). \]
\end{prop}

\begin{proof}
    We omit the $q$-shifts for the first part of the proof and we use the following notations:
    \[C:= \colcom_N\left ( \NB{\tikz[font=\tiny,xscale=1,yscale=.7]{\input{\imagesfolder/sch/forkslide_3a}}}  \right ) \hspace{2mm} \text{and} \hspace{2mm} D:= \colcom_N \left ( \NB{\tikz[font=\tiny,scale=1,yscale=.7]{\input{\imagesfolder/sch/forkslide_4a}}} \right ). \]
    The complex $C$ is the flattening of the bicomplex $C' = (C'_{i,j},d^-,d_V)$ with $0 \leq i \leq \min(b,c)$ and $0 \leq j \leq 1$: 
   \[
\NB{\tikz[xscale = 3, yscale = 2]{
	\node (i0) at (-.5, 0) {$C':=$}; 
    \node (i1) at (1,2) {\NB{\tikz[scale=1.1]{\input{\imagesfolder/sch/forkslide_iso_1}}}};
    \node (i2) at (3.5,2) {\NB{\tikz[scale=1.1]{\input{\imagesfolder/sch/forkslide_iso_2}}}};
    \node (i3) at (1,-2) {\NB{\tikz[scale=1.1]{\input{\imagesfolder/sch/forkslide_iso_3}}}};
    \node (i4) at (3.5,-2) {\NB{\tikz[scale=1.1]{\input{\imagesfolder/sch/forkslide_iso_4}}}};
    \draw[->] (i1) -- (i2) node[pos=.5,above] {$d^{-}$};
    \draw[->] (i3) -- (i4) node[pos=.5,above] {$d^{-}$};
    \draw[->] (i1) -- (i3) node[pos=.5,left] {$d_V$};
    \draw[->] (i2) -- (i4) node[pos=.5,left] {$d_V$};
    \node at (4.4,2) {$\rightarrow \cdots$};
    \node at (4.4,-2) {$\rightarrow \cdots $};
    \node at (0,2) {$\cdots \rightarrow $};
    \node at (0,-2) {$\cdots  \rightarrow $};
}}.
\]
The complex $D=(D_i,d^-,0)$, which can be considered as a bicomplex with trivial vertical differentials, is the following:

\[
\NB{\tikz[xscale = 3, yscale = 2]{
	\node (i0) at (-.5, 0) {$D=$}; 
    \node (i1) at (1,0) {\NB{\tikz[scale=1.1]{\input{\imagesfolder/sch/forkslide_iso_5}}}};
    \node (i2) at (3,0) {\NB{\tikz[scale=1.1]{\input{\imagesfolder/sch/forkslide_iso_6}}}};
    \draw[->] (i1) -- (i2) node[pos=.5,above] {$d^{-}$};
    \node at (0.1,0) {$\cdots \rightarrow$};
    \node at (3.7,0) {$\rightarrow \cdots$};
}}.
\]

Consider the bicomplex $H=(H_i,d_H,0)$:
\[
\NB{\tikz[xscale = 3, yscale = 2]{
	\node (i0) at (-.5, 0) {$H=$}; 
    \node (i1) at (1,0) {\NB{\tikz[scale=1.1]{\input{\imagesfolder/sch/forkslide_iso_7}}}};
    \node (i2) at (3,0) {\NB{\tikz[scale=1.1]{\input{\imagesfolder/sch/forkslide_iso_8}}}};
    \draw[->] (i1) -- (i2) node[pos=.5,above] {$d_H$};
    \node at (0,0) {$\cdots \rightarrow$};
    \node at (4,0) {$\rightarrow \cdots$};
}}.
\]
The horizontal differentials are those from Lemma~\ref{nulhom}, which ensures that the bicomplex $H$ is null-homotopic (forgetting the $\Witt$ structure).

Let us now define three $\Witt$-equivariant maps of bicomplexes, $\iota = H \rightarrow C'$, $\phi : D \rightarrow H$ and $\psi : D \rightarrow C'$ define by the movies: 

\begin{equation}
    \iota := \mymoviethree[yscale=1.45]{\NB{\tikz[scale=.6,font=\tiny]{\input{\imagesfolder/sch/forkslide_movie_2}}}}{\NB{\tikz[scale=.6,font=\tiny]{\input{\imagesfolder/sch/forkslide_movie_3}}}}{\NB{\tikz[scale=.6,font=\tiny]{\input{\imagesfolder/sch/forkslide_movie_1}}}}{\mapasso \circ \mapdcup}{\mapasso} 
\end{equation}
\begin{equation}
    \phi := \mymoviethree[yscale=1.45]{\NB{\tikz[scale=.75,font=\tiny]{\input{\imagesfolder/sch/forkslide_movie_4}}}}{\NB{\tikz[scale=.75,font=\tiny]{\input{\imagesfolder/sch/forkslide_movie_5}}}}{\NB{\tikz[scale=.7,font=\tiny]{\input{\imagesfolder/sch/forkslide_movie_2}}}}{\mapasso}{\mapsingsdle} 
\end{equation}
\begin{equation}
    \psi := \mymoviefour[,xscale=.85,yscale=1.45]{\NB{\tikz[scale=.6,font=\tiny]{\input{\imagesfolder/sch/forkslide_movie_4}}}}{\NB{\tikz[scale=.6,font=\tiny]{\input{\imagesfolder/sch/forkslide_movie_5}}}}{\NB{\tikz[scale=.6,font=\tiny]{\input{\imagesfolder/sch/forkslide_movie_6}}}}{\NB{\tikz[scale=.55,font=\tiny]{\input{\imagesfolder/sch/forkslide_movie_7}}}}{\mapasso}{\mapasso^2 \circ \mapdcup}{\mapsingsdle}
\end{equation}
where we omit the red dots for the sake of readability. These three morphisms are $\Witt$-equivariant, which implies in particular that $\text{cone}(\iota) \simeq C' \oplus tH$ is isomorphic to $C'$ in the relative homotopy canopoly. 

We consider the following complex: 

\[
\begin{tikzpicture}[xscale=1.5]
    \begin{scope}{}
        \node (a) at (0,0) {0};
        \node (b) at (.8,0) {$qD_i$};
        \node (c) at (2,1) {$qC'_{i,0}$};
        \node (d) at (2,-1) {$q^{i-b}H_i$};
        \node (e) at (2,0) {$\oplus$};
        \node (f) at (3.2,0) {$C'_{i,1}$};
        \node (g) at (4,0) {0.};
        \draw[->] (a) -- (b);
        \draw[->] (b) -- (c) node[pos=.7,left,yshift=3pt] {$-\psi$};
        \draw[->] (b) -- (d) node[pos=.7,left,yshift=-3pt] {$\phi$};
        \draw[->] (c) -- (f) node[pos=.3,right,yshift=3pt] {$d_V$};
        \draw[->] (d) -- (f) node[pos=.3,right,yshift=-3pt] {$\iota$};
        \draw[->] (f) -- (g);
    \end{scope}
\end{tikzpicture}
\]

Forgetting the $\Witt$-structure, we have that $C'_{i,1} \simeq [b+1-i]H_i$ and $C'_{i,0} \simeq D_i \oplus [b-i] H_i$. Using these isomorphisms we obtain: 

\[
\begin{tikzpicture}[xscale=2]
    \begin{scope}{}
        \node (a) at (.2,0) {0};
        \node (b) at (.8,0) {$qD_i$};
        \node (c1) at (2,1.5) {$q[b-i]H_i$};
        \node (c2) at (2,.5) {$qD_i$};
        \node (c0) at (2,1) {$\oplus$};
        \node (d) at (2,-1) {$q^{i-b}H_i$};
        \node (e) at (2,-.25) {$\oplus$};
        \node (f1) at (3.2,0.5) {$q[b-i]H_i$};
        \node (f0) at (3.2,0) {$\oplus$};
        \node (f2) at (3.2,-0.5) {$q^{i-b}H_i$};
        \node (g) at (4.2,0) {0,};
        \draw[->] (a) -- (b);
        \draw[->] (b) -- (c1) node[pos=.7,left,yshift=3pt] {$-\widetilde{\psi}$};
        \draw[->] (b) -- (c2) node[pos=.5,below] {$-\Id$};
        \draw[->] (b) -- (d) node[pos=.7,left,yshift=-3pt] {$\phi$};
        \draw[->] (c1) -- (f1) node[pos=.5,above] {$\Id$};
        \draw[->] (d) -- (f2) node[pos=.5,below] {$\Id$};
        \draw[->,blue] (c1) -- (f2) ;
        \draw[->,blue] (c2) -- (f2) node[pos=.5,below] {$\widetilde{d_V}$};
        \draw[->,blue] (c2) -- (f1) ;
        \draw[->] (f1) -- (g);
        \draw[->] (f2) -- (g);
    \end{scope}
\end{tikzpicture}
\]
which is null-homotopic. This chain complex constitutes the homogeneous terms of the mapping cone of 
\[D \xrightarrow{(-\psi,\phi)} \text{cone}(\iota).\]
This implies that $C$ and $D$ are isomorphic in the relative homotopy canopoly.  
\end{proof}

%% file: sch/nul_hom_1.tex
\begin{scope}
  \def\la{$b-i$}
  \def\cir{$\color{\crd}-s/2$}
  \def\tri{$\color{\crd}-\lambda(b-i-1)+s/2$}
  \def\lab{$a+b-i$}
  \def\circ{$\color{\crd}s/2$}
  \def\tria{$\color{\crd}\lambda(a+b-i-1)-s/2$}
  \def\labe{$a-i+1$}
  \def\circl{$\color{\crd}-s/2$}
  \def\trian{$\color{\crd}-\lambda(a-i)+s/2$}
  \def\label{$i$}
  \def\circle{$\color{\crd}\bar{s}/2$}
  \def\triang{$\color{\crd}-\lambda(i-1)-\bar{s}/2$}
  \coordinate (bl) at (-1, -1.5);
  \coordinate (br) at ( 1, -1.5);
  \coordinate (tl) at (-1,  1.5);
  \coordinate (tr) at ( 1,  1.5);
  \draw[>->] (bl) -- (tl) node[pos = 0, below] {$a$} node[pos = 1, above] {$b-1$}  node[pos=.5, left] {\lab}
  node[pos= 0.6 , draw=\crd, fill=white,regular polygon, regular polygon sides=3,inner sep=0.9pt] {} node[pos=0.6, right]{\tria} 
  node[pos=.4, draw=\crd, fill= white, circle , inner sep=1.5pt]{} node[pos=.4, left]{\circ}
  coordinate[pos = 0.75] (ml) 
  coordinate[pos = 0.25] (nl);
  \draw[>->] (br) -- (tr) node[pos = 0, below] {$b$} node[pos = 1, above] {$a+1$}  node[pos=.5,right] {\label}
  node[pos= 0.4 , draw=\crd, fill=white,regular polygon, regular polygon sides=3,inner sep=0.9pt] {} node[pos=0.4, left]{\triang} 
  node[pos=.65, draw=\crd, fill= white, circle , inner sep=1.5pt]{} node[pos=.65, right]{\circle}
  coordinate[pos = 0.85] (mr)  
  coordinate[pos = 0.15] (nr);
  \draw[->-] (ml) -- (mr) node [pos= 0.4, above,yshift=.5mm] {\labe} node[pos=.8, draw=\crd, fill= white, circle , inner sep=1.5pt]{} node[pos=.82, above]{\circl} 
  node[pos= 0.25 , draw=\crd, fill=white,regular polygon, regular polygon sides=3,inner sep=0.9pt] {} 
  node[pos=.45, below,yshift=-.5mm] {\trian};
  \draw[->-] (nr) -- (nl) node [pos= 0.55, below,yshift=-.5mm] {\la} node[pos=.2, draw=\crd, fill= white, circle , inner sep=1.5pt]{} node[pos=.15, below]{\cir} 
  node[pos= 0.75 , draw=\crd, fill=white,regular polygon, regular polygon sides=3,inner sep=0.9pt] {} 
  node[pos=.55, above,yshift=.2mm] {\tri};
\end{scope}

%% file: sch/nul_hom_2.tex
\begin{scope}
  \def\la{$b-m+j$}
  \def\cir{$\color{\crd}\bar{s}/2$}
  \def\tri{$\color{\crd}-\lambda(b-m+j-1)+\bar{s}/2$}
  \def\lab{$a+b-m+j$}
  \def\circ{$\color{\crd}-\bar{s}/2$}
  \def\tria{$\color{\crd}\lambda(a+b-m+j)+\bar{s}/2$}
  \def\labe{$a-m+j+1$}
  \def\circl{$\color{\crd}\bar{s}/2$}
  \def\trian{$\color{\crd}-\lambda(a-m+j)-\bar{s}/2$}
  \def\label{$m-j$}
  \def\circle{$\color{\crd}-s/2$}
  \def\triang{$\color{\crd}-\lambda(m-j-1)+s/2$}
  \coordinate (bl) at (-1, -1.5);
  \coordinate (br) at ( 1, -1.5);
  \coordinate (tl) at (-1,  1.5);
  \coordinate (tr) at ( 1,  1.5);
  \draw[>->] (bl) -- (tl) node[pos = 0, below] {$a$} node[pos = 1, above] {$b-1$}  node[pos=.5, left] {\lab}
  node[pos= 0.6 , draw=\crd, fill=white,regular polygon, regular polygon sides=3,inner sep=0.9pt] {} node[pos=0.6, right]{\tria} 
  node[pos=.4, draw=\crd, fill= white, circle , inner sep=1.5pt]{} node[pos=.4, left]{\circ}
  coordinate[pos = 0.75] (ml) 
  coordinate[pos = 0.25] (nl);
  \draw[>->] (br) -- (tr) node[pos = 0, below] {$b$} node[pos = 1, above] {$a+1$}  node[pos=.5,right] {\label}
  node[pos= 0.4 , draw=\crd, fill=white,regular polygon, regular polygon sides=3,inner sep=0.9pt] {} node[pos=0.4, left]{\triang} 
  node[pos=.65, draw=\crd, fill= white, circle , inner sep=1.5pt]{} node[pos=.65, right]{\circle}
  coordinate[pos = 0.85] (mr)  
  coordinate[pos = 0.15] (nr);
  \draw[->-] (ml) -- (mr) node [pos= 0.4, above,yshift=.5mm] {\labe} node[pos=.8, draw=\crd, fill= white, circle , inner sep=1.5pt]{} node[pos=.82, above]{\circl} 
  node[pos= 0.25 , draw=\crd, fill=white,regular polygon, regular polygon sides=3,inner sep=0.9pt] {} 
  node[pos=.5, below,yshift=-.7mm] {\trian};
  \draw[->-] (nr) -- (nl) node [pos= 0.55, below,yshift=-.5mm] {\la} node[pos=.2, draw=\crd, fill= white, circle , inner sep=1.5pt]{} node[pos=.15, below]{\cir} 
  node[pos= 0.75 , draw=\crd, fill=white,regular polygon, regular polygon sides=3,inner sep=0.9pt] {} 
  node[pos=.53, above,yshift=.4mm] {\tri};
\end{scope}

%% file: sch/diff_neg_1_p.tex
\begin{scope}
  \def\la{$b-i$}
  \def\lab{$a+b-i$}
  \def\labe{$a-i+1$}
  \def\label{$i$}
  \coordinate (bl) at (-1, -1.5);
  \coordinate (br) at ( 1, -1.5);
  \coordinate (tl) at (-1,  1.5);
  \coordinate (tr) at ( 1,  1.5);
  \draw[>->] (bl) -- (tl) node[pos = 0, below] {$a$} node[pos = 1, above] {$b-1$}  node[pos=.5, right] {}
  coordinate[pos = 0.75] (ml)
  coordinate[pos = 0.25] (nl);
  \draw[>->] (br) -- (tr) node[pos = 0, below] {$b$} node[pos = 1, above] {$a+1$}  node[pos=.5,right] {\label}
  coordinate[pos = 0.85] (mr)  
  coordinate[pos = 0.15] (nr);
  \draw[->-] (nr) -- (mr);
  \draw[->-] (nl) -- (ml);
  \draw[->-] (ml) -- (mr) node [pos= 0.5, above,yshift=.5mm] {\labe};
  \draw[->-] (nr) -- (nl) node [pos= 0.5, below,yshift=-.5mm] {\la} ;
\end{scope}

%% file: sch/diff_neg_3_p.tex
\begin{scope}
  \def\la{$b-i-1$}
  \def\lab{$a+b-i$}
  \def\labe{$a-i$}
  \def\label{$i$}
  \coordinate (bl) at (-1, -1.5);
  \coordinate (br) at ( 1, -1.5);
  \coordinate (tl) at (-1,  1.5);
  \coordinate (tr) at ( 1,  1.5);
  \draw[>->] (bl) -- (tl) node[pos = 0, below] {$a$} node[pos = 1, above] {$b-1$}  node[pos=.5, right] {}
  coordinate[pos = 0.75] (ml)
  coordinate[pos = 0.25] (nl)
  coordinate[pos=0.4] (ld)
  coordinate[pos=0.6] (lu);
  \draw[>->] (br) -- (tr) node[pos = 0, below] {$b$} node[pos = 1, above] {$a+1$}  node[pos=.5,right] {\label}
  coordinate[pos = 0.85] (mr)  
  coordinate[pos = 0.15] (nr)
  coordinate[pos=0.7] (ru)
  coordinate[pos=0.3] (rd);
  \draw[->-] (nr) -- (mr);
  \draw[->-] (nl) -- (ml);
  \draw[->-] (ml) -- (mr) node [pos= 0.5, above,yshift=.5mm] {\labe};
  \draw[->-] (nr) -- (nl) node [pos= 0.5, below,yshift=-.5mm] {\la} ;
  \draw[->-] (lu) -- (ru) node[pos=.3, below]{$1$};
  \draw[->-] (rd) -- (ld) node[pos=.3, above]{$1$};
\end{scope}

%% file: sch/diff_neg_5_p.tex
\begin{scope}
  \def\la{$b-i-1$}
  \def\lab{$a+b-i-1$}
  \def\labe{$a-i-1$}
  \def\label{$i+1$}
  \coordinate (bl) at (-1, -1.5);
  \coordinate (br) at ( 1, -1.5);
  \coordinate (tl) at (-1,  1.5);
  \coordinate (tr) at ( 1,  1.5);
  \draw[>->] (bl) -- (tl) node[pos = 0, below] {$a$} node[pos = 1, above] {$b-1$}  node[pos=.5, right] {}
  coordinate[pos = 0.75] (ml)
  coordinate[pos = 0.25] (nl);
  \draw[>->] (br) -- (tr) node[pos = 0, below] {$b$} node[pos = 1, above] {$a+1$}  node[pos=.5,left] {\label}
  coordinate[pos = 0.85] (mr)  
  coordinate[pos = 0.15] (nr);
  \draw[->-] (nr) -- (mr);
  \draw[->-] (nl) -- (ml);
  \draw[->-] (ml) -- (mr) node [pos= 0.5, above,yshift=.5mm] {\labe};
  \draw[->-] (nr) -- (nl) node [pos= 0.5, below,yshift=-.5mm] {\la} ;
\end{scope}

%% file: sch/diff_pos_1_p.tex
\begin{scope}
  \def\la{$b-m+j$}
  \def\lab{$a+b-m+j+1$}
  \def\labe{$a-m+j$}
  \def\label{$m-j$}
  \coordinate (bl) at (-1, -1.5);
  \coordinate (br) at ( 1, -1.5);
  \coordinate (tl) at (-1,  1.5);
  \coordinate (tr) at ( 1,  1.5);
  \draw[>->] (bl) -- (tl) node[pos = 0, below] {$a$} node[pos = 1, above] {$b-1$}  node[pos=.5, right] {}
  coordinate[pos = 0.75] (ml)
  coordinate[pos = 0.25] (nl);
  \draw[>->] (br) -- (tr) node[pos = 0, below] {$b$} node[pos = 1, above] {$a+1$}  node[pos=.5,left] {\label}
  coordinate[pos = 0.85] (mr)  
  coordinate[pos = 0.15] (nr);
  \draw[->-] (nr) -- (mr);
  \draw[->-] (nl) -- (ml);
  \draw[->-] (ml) -- (mr) node [pos= 0.5, above,yshift=.5mm] {\labe};
  \draw[->-] (nr) -- (nl) node [pos= 0.5, below,yshift=-.5mm] {\la} ;
\end{scope}

%% file: sch/diff_pos_2_p.tex
\begin{scope}
  \def\la{$b-m+j$}
  \def\lab{$a+b-m+j$}
  \def\labe{$a-m+j+1$}
  \def\label{$m-j-1$}
  \coordinate (bl) at (-1, -1.5);
  \coordinate (br) at ( 1, -1.5);
  \coordinate (tl) at (-1,  1.5);
  \coordinate (tr) at ( 1,  1.5);
  \draw[>->] (bl) -- (tl) node[pos = 0, below] {$a$} node[pos = 1, above] {$b-1$}  node[pos=.5, right] {}
  coordinate[pos = 0.75] (ml)
  coordinate[pos = 0.25] (nl);
  \draw[>->] (br) -- (tr) node[pos = 0, below] {$b$} node[pos = 1, above] {$a+1$}  node[pos=.5,right] {}
  coordinate[pos = 0.85] (mr)  
  coordinate[pos = 0.15] (nr);
  \draw[->-] (nr) -- (mr) coordinate[pos=0.2] (lu) coordinate[pos=0.8] (ru);
  \draw[->-] (nl) -- (ml)  coordinate[pos=0.8] (ld) coordinate[pos=0.2] (rd);
  \draw[->-] (ml) -- (mr) node [pos= 0.5, above,yshift=.5mm] {\labe};
  \draw[->-] (nr) -- (nl) node [pos= 0.5, below,yshift=-.5mm] {\la} ;
  \draw[->-] (lu) .. controls + (-0.6, 0.6) and + (-0.6, -0.6) .. (ru) node[pos=.5, left]{$1$};
\end{scope}

%% file: sch/diff_pos_3_p.tex
\begin{scope}
  \def\la{$b-m+j$}
  \def\lab{$a+b-m+j+1$}
  \def\labe{$a-m+j+1$}
  \def\label{$m-j-1$}
  \coordinate (bl) at (-1, -1.5);
  \coordinate (br) at ( 1, -1.5);
  \coordinate (tl) at (-1,  1.5);
  \coordinate (tr) at ( 1,  1.5);
  \draw[>->] (bl) -- (tl) node[pos = 0, below] {$a$} node[pos = 1, above] {$b-1$}  node[pos=.5, right] {}
  coordinate[pos = 0.75] (ml)
  coordinate[pos = 0.25] (nl)
  coordinate[pos=0.4] (ld)
  coordinate[pos=0.6] (lu);
  \draw[>->] (br) -- (tr) node[pos = 0, below] {$b$} node[pos = 1, above] {$a+1$}  node[pos=.5,right] {}
  coordinate[pos = 0.85] (mr)  
  coordinate[pos = 0.15] (nr)
  coordinate[pos=0.7] (ru)
  coordinate[pos=0.3] (rd);
  \draw[->-] (nr) -- (mr);
  \draw[->-] (nl) -- (ml);
  \draw[->-] (ml) -- (mr) node [pos= 0.5, above,yshift=.5mm] {\labe};
  \draw[->-] (nr) -- (nl) node [pos= 0.5, below,yshift=-.5mm] {\la} ;
  \draw[->-] (lu) -- (ru) node[pos=.3, below]{$1$};
  \draw[->-] (rd) -- (ld) node[pos=.3, above]{$1$};
\end{scope}

%% file: sch/diff_pos_5_p.tex
\begin{scope}
  \def\la{$b-m+j+1$}
  \def\lab{$a+b-m+j+1$}
  \def\labe{$a-m+j+2$}
  \def\label{$m-j-1$}
  \coordinate (bl) at (-1, -1.5);
  \coordinate (br) at ( 1, -1.5);
  \coordinate (tl) at (-1,  1.5);
  \coordinate (tr) at ( 1,  1.5);
  \draw[>->] (bl) -- (tl) node[pos = 0, below] {$a$} node[pos = 1, above] {$b-1$}  node[pos=.5, right] {}
  coordinate[pos = 0.75] (ml)
  coordinate[pos = 0.25] (nl);
  \draw[>->] (br) -- (tr) node[pos = 0, below] {$b$} node[pos = 1, above] {$a+1$}  node[pos=.5,left] {\label}
  coordinate[pos = 0.85] (mr)  
  coordinate[pos = 0.15] (nr);
  \draw[->-] (nr) -- (mr);
  \draw[->-] (nl) -- (ml);
  \draw[->-] (ml) -- (mr) node [pos= 0.5, above] {\labe};
  \draw[->-] (nr) -- (nl) node [pos= 0.5, below] {\la} ;
\end{scope}

%% file: sch/forkslide_iso_1.tex
\begin{scope}[font=\tiny]
  \def\la{$c-i$}
  \def\cir{\color{\labrd}$-\frac{s}{2}$}
  \def\tri{\color{\labrd}$-\lambda(c-i-1)+\frac{s}{2}$}
  \def\lab{}
  \def\circ{\color{\labrd}$\frac{s}{2}$}
  \def\tria{\color{\labrd}$\lambda(b+c-i-1)-\frac{s}{2}$}
  \def\labe{$b-i$}
  \def\circl{\color{\labrd}$-\frac{s}{2}$}
  \def\trian{\color{\labrd}$-\lambda(b-i-1)+\frac{s}{2}$}
  \def\label{$i$}
  \def\circle{\color{\labrd}$\frac{\bar{s}}{2}$}
  \def\triang{\color{\labrd}$-\lambda(i-1)-\frac{\bar{s}}{2}$}
  \coordinate (bl) at (-1, -1.5);
  \coordinate (br) at ( 1, -1.5);
  \coordinate (ml) at (-1, .8);
  \coordinate (nl) at (-1,-.8);
  \coordinate (mr) at (1,1);
  \coordinate (nr) at (1,-1);
  \coordinate (tl) at (-2,  1.5);
  \coordinate (tr) at ( -1,  3.5);
  \coordinate (ttl) at (-2,2.5);
  \coordinate (ttr) at (-1,4);
  \coordinate (Tl) at  (-3,3.5);
  \coordinate (Bl) at (-3,0);  
  \draw[>-] (bl) -- (ml) node[pos = 0, below] {$b$}   node[pos=.5, left] {\lab}
  node[pos= 0.5 , draw=\crd, fill=white,regular polygon, regular polygon sides=3,inner sep=0.9pt] {} node[pos=0.5, left]{\tria} 
  node[pos=.75, draw=\crd, fill= white, circle , inner sep=1.5pt]{} node[pos=.75, left]{\circ};
  \draw[>-] (br) -- (mr) node[pos = 0, below] {$c$} node[pos=.55,right] {\label}
  node[pos= 0.5 , draw=\crd, fill=white,regular polygon, regular polygon sides=3,inner sep=0.9pt] {} node[pos=0.5, left]{\triang} 
  node[pos=.7, draw=\crd, fill= white, circle , inner sep=1.5pt]{} node[pos=.7, right]{\circle};
  \draw[->-] (ml) -- (mr) node [pos= 0.4, above,yshift=.5mm] {\labe} node[pos=.8, draw=\crd, fill= white, circle , inner sep=1.5pt]{} node[pos=.82, above]{\circl} 
  node[pos= 0.25 , draw=\crd, fill=white,regular polygon, regular polygon sides=3,inner sep=0.9pt] {} 
  node[pos=.45, below,yshift=-.5mm] {\trian};
  \draw[->-] (nr) -- (nl) node [pos= 0.55, above,yshift=-.5mm] {\la} node[pos=.85, draw=\crd, fill= white, circle , inner sep=1.5pt]{} node[pos=.8, above]{\cir} 
  node[pos= 0.2 , draw=\crd, fill=white,regular polygon, regular polygon sides=3,inner sep=0.9pt] {} 
  node[pos=.5, below] {\tri};
  \draw[->-] (mr) .. controls +(0,2) and + (0, 0) .. (tr) node[pos = .5, right] {$b$};
  \draw[->-] (ml) -- (tl) node[pos=.7,below] {$c$} node[pos=.3, draw=\crd, fill= white, circle , inner sep=1.5pt]{} node[pos=.3, below,xshift=-3pt]{\color{\labrd}$-\frac{s}{2}$} node[pos= 0.8 , draw=\crd, fill=white,regular polygon, regular polygon sides=3,inner sep=0.9pt] {} node[pos=0.8, right,yshift=3pt]{\color{\labrd}$-\lambda(c-1) +\frac{s}{2}$} ;
  \draw[->-] (tl) -- (ttl) node[pos=.5,left] {$c+1$} node[pos=.3, draw=\crd, fill= white, circle , inner sep=1.5pt]{} node[pos=.3, right]{\color{\labrd}$\frac{s}{2}$} node[pos= 0.8 , draw=\crd, fill=white,regular polygon, regular polygon sides=3,inner sep=0.9pt] {} node[pos=0.8, right]{\color{\labrd}$\lambda c -\frac{s}{2}$};
  \draw[->-] (ttl) -- (tr) node[pos=.5,above] {$1$};
  \draw[->] (ttl) .. controls +(-.8,.4) and +(0,0) .. (Tl) node[pos=1,above] {$c$};
  \draw[->] (tr) -- (ttr) node[pos=1,above] {$b+1$};
  \draw[->-] (Bl) .. controls +(0,1) and +(0,0) .. (tl) node[pos=0.5,left] {$1$};
\end{scope}

%% file: sch/forkslide_iso_2.tex
\begin{scope}[font=\tiny]
  \def\la{$c-i-1$}
  \def\cir{\color{\labrd}$-\frac{s}{2}$}
  \def\tri{\color{\labrd}$-\lambda(c-i-2)+\frac{s}{2}$}
  \def\lab{}
  \def\circ{\color{\labrd}$\frac{s}{2}$}
  \def\tria{\color{\labrd}$\lambda(b+c-i-2)-\frac{s}{2}$}
  \def\labe{$b-i-1$}
  \def\circl{\color{\labrd}$-\frac{s}{2}$}
  \def\trian{\color{\labrd}$-\lambda(b-i-2)+\frac{s}{2}$}
  \def\label{$i+1$}
  \def\circle{\color{\labrd}$\frac{\bar{s}}{2}$}
  \def\triang{\color{\labrd}$-\lambda i-\frac{\bar{s}}{2}$}
  \coordinate (bl) at (-1, -1.5);
  \coordinate (br) at ( 1, -1.5);
  \coordinate (ml) at (-1, .8);
  \coordinate (nl) at (-1,-.8);
  \coordinate (mr) at (1,1);
  \coordinate (nr) at (1,-1);
  \coordinate (tl) at (-2,  1.5);
  \coordinate (tr) at ( -1,  3.5);
  \coordinate (ttl) at (-2,2.5);
  \coordinate (ttr) at (-1,4);
  \coordinate (Tl) at  (-3,3.5);
  \coordinate (Bl) at (-3,0);  
  \draw[>-] (bl) -- (ml) node[pos = 0, below] {$b$}   node[pos=.5, left] {\lab}
  node[pos= 0.5 , draw=\crd, fill=white,regular polygon, regular polygon sides=3,inner sep=0.9pt] {} node[pos=0.5, left]{\tria} 
  node[pos=.75, draw=\crd, fill= white, circle , inner sep=1.5pt]{} node[pos=.75, left]{\circ};
  \draw[>-] (br) -- (mr) node[pos = 0, below] {$c$} node[pos=.55,right] {\label}
  node[pos= 0.5 , draw=\crd, fill=white,regular polygon, regular polygon sides=3,inner sep=0.9pt] {} node[pos=0.5, left]{\triang} 
  node[pos=.7, draw=\crd, fill= white, circle , inner sep=1.5pt]{} node[pos=.7, right]{\circle};
  \draw[->-] (ml) -- (mr) node [pos= 0.4, above,yshift=.5mm] {\labe} node[pos=.8, draw=\crd, fill= white, circle , inner sep=1.5pt]{} node[pos=.82, above]{\circl} 
  node[pos= 0.25 , draw=\crd, fill=white,regular polygon, regular polygon sides=3,inner sep=0.9pt] {} 
  node[pos=.45, below,yshift=-.5mm] {\trian};
  \draw[->-] (nr) -- (nl) node [pos= 0.4, above] {\la} node[pos=.85, draw=\crd, fill= white, circle , inner sep=1.5pt]{} node[pos=.8, above]{\cir} 
  node[pos= 0.2 , draw=\crd, fill=white,regular polygon, regular polygon sides=3,inner sep=0.9pt] {} 
  node[pos=.5, below] {\tri};
  \draw[->-] (mr) .. controls +(0,2) and + (0, 0) .. (tr) node[pos = .5, right] {$b$};
  \draw[->-] (ml) -- (tl) node[pos=.7,below] {$c$} node[pos=.3, draw=\crd, fill= white, circle , inner sep=1.5pt]{} node[pos=.3, below,xshift=-3pt]{\color{\labrd}$-\frac{s}{2}$} node[pos= 0.8 , draw=\crd, fill=white,regular polygon, regular polygon sides=3,inner sep=0.9pt] {} node[pos=0.8, right,yshift=3pt]{\color{\labrd}$-\lambda(c-1) +\frac{s}{2}$} ;
  \draw[->-] (tl) -- (ttl) node[pos=.5,left] {$c+1$} node[pos=.3, draw=\crd, fill= white, circle , inner sep=1.5pt]{} node[pos=.3, right]{\color{\labrd}$\frac{s}{2}$} node[pos= 0.8 , draw=\crd, fill=white,regular polygon, regular polygon sides=3,inner sep=0.9pt] {} node[pos=0.8, right]{\color{\labrd}$\lambda c -\frac{s}{2}$};
  \draw[->-] (ttl) -- (tr) node[pos=.5,above] {$1$};
  \draw[->] (ttl) .. controls +(-.8,.4) and +(0,0) .. (Tl) node[pos=1,above] {$c$};
  \draw[->] (tr) -- (ttr) node[pos=1,above] {$b+1$};
  \draw[->-] (Bl) .. controls +(0,1) and +(0,0) .. (tl) node[pos=0.5,left] {$1$};
\end{scope}

%% file: sch/forkslide_iso_3.tex
\begin{scope}[font=\tiny]
  \def\la{$c-i$}
  \def\cir{\color{\labrd}$-\frac{s}{2}$}
  \def\tri{\color{\labrd}$-\lambda(c-i-1)+\frac{s}{2}$}
  \def\lab{}
  \def\circ{\color{\labrd}$\frac{s}{2}$}
  \def\tria{\color{\labrd}$\lambda(b+c-i-1)-\frac{s}{2}$}
  \def\labe{$b-i$}
  \def\circl{\color{\labrd}$-\frac{s}{2}$}
  \def\trian{\color{\labrd}$-\lambda(b-i-1)+\frac{s}{2}$}
  \def\label{$i$}
  \def\circle{\color{\labrd}$\frac{\bar{s}}{2}$}
  \def\triang{\color{\labrd}$-\lambda(i-1)-\frac{\bar{s}}{2}$}
  \coordinate (bl) at (-1, -1.5);
  \coordinate (br) at ( 1, -1.5);
  \coordinate (ml) at (-1, .8);
  \coordinate (nl) at (-1,-.8);
  \coordinate (mr) at (1,1);
  \coordinate (nr) at (1,-1);
  \coordinate (tl) at (-1.5,  1.75);
  \coordinate (tr) at ( -1,  3);
  \coordinate (ttl) at (-2.8,2);
  \coordinate (ttr) at (-1,3.5);
  \coordinate (Tl) at  (-2.8,3.5);
  \coordinate (Bl) at (-3,0);  
  \draw[>-] (bl) -- (ml) node[pos = 0, below] {$b$}   node[pos=.5, left] {\lab}
  node[pos= 0.5 , draw=\crd, fill=white,regular polygon, regular polygon sides=3,inner sep=0.9pt] {} node[pos=0.5, left]{\tria} 
  node[pos=.75, draw=\crd, fill= white, circle , inner sep=1.5pt]{} node[pos=.75, left]{\circ};
  \draw[>-] (br) -- (mr) node[pos = 0, below] {$c$} node[pos=.55,right] {\label}
  node[pos= 0.5 , draw=\crd, fill=white,regular polygon, regular polygon sides=3,inner sep=0.9pt] {} node[pos=0.5, left]{\triang} 
  node[pos=.7, draw=\crd, fill= white, circle , inner sep=1.5pt]{} node[pos=.7, right]{\circle};
  \draw[->-] (ml) -- (mr) node [pos= 0.4, above,yshift=.5mm] {\labe} node[pos=.8, draw=\crd, fill= white, circle , inner sep=1.5pt]{} node[pos=.82, above]{\circl} 
  node[pos= 0.25 , draw=\crd, fill=white,regular polygon, regular polygon sides=3,inner sep=0.9pt] {} 
  node[pos=.45, below,yshift=-.5mm] {\trian};
  \draw[->-] (nr) -- (nl) node [pos= 0.55, above,yshift=-.5mm] {\la} node[pos=.85, draw=\crd, fill= white, circle , inner sep=1.5pt]{} node[pos=.8, above]{\cir} 
  node[pos= 0.2 , draw=\crd, fill=white,regular polygon, regular polygon sides=3,inner sep=0.9pt] {} 
  node[pos=.5, below] {\tri};
  \draw[->-] (mr) .. controls +(0,1) and + (0, 0) .. (tr) node[pos = .5, right] {$b$};
  \draw[->-] (ml) -- (tl) node[pos=.5,left] {$c$};
  \draw[->-] (tl) -- (ttl) node[pos=.7,below,yshift=-2pt] {$c-1$} node[pos=.3, draw=\crd, fill= white, circle , inner sep=1.5pt]{} node[pos=.3, below]{\color{\labrd}$-\frac{s}{2}$} node[pos= 0.8 , draw=\crd, fill=white,regular polygon, regular polygon sides=3,inner sep=0.9pt] {} node[pos=0.5, above]{\color{\labrd}$-\lambda (c-2) +\frac{s}{2}$};
  \draw[->-] (tl) -- (tr) node[pos=.5,left] {$1$};
  \draw[->] (ttl) .. controls +(-.1,.1) and +(0,0) .. (Tl) node[pos=1,above] {$c$} node[pos=.4, draw=\crd, fill= white, circle , inner sep=1.5pt]{} node[pos=.4, right]{\color{\labrd}$\frac{s}{2}$} node[pos= 0.6 , draw=\crd, fill=white,regular polygon, regular polygon sides=3,inner sep=0.9pt] {} node[pos=0.6, right]{\color{\labrd}$\lambda (c-1) -\frac{s}{2}$};
  \draw[->] (tr) -- (ttr) node[pos=1,above] {$b+1$};
  \draw[->-] (Bl) .. controls +(0,1) and +(0,0) .. (ttl) node[pos=0.5,left] {$1$};
\end{scope}

%% file: sch/forkslide_iso_4.tex
\begin{scope}[font=\tiny]
  \def\la{$c-i-1$}
  \def\cir{\color{\labrd}$-\frac{s}{2}$}
  \def\tri{\color{\labrd}$-\lambda(c-i-2)+\frac{s}{2}$}
  \def\lab{}
  \def\circ{\color{\labrd}$\frac{s}{2}$}
  \def\tria{\color{\labrd}$\lambda(b+c-i-2)-\frac{s}{2}$}
  \def\labe{$b-i-1$}
  \def\circl{\color{\labrd}$-\frac{s}{2}$}
  \def\trian{\color{\labrd}$-\lambda(b-i-2)+\frac{s}{2}$}
  \def\label{$i+1$}
  \def\circle{\color{\labrd}$\frac{\bar{s}}{2}$}
  \def\triang{\color{\labrd}$-\lambda i-\frac{\bar{s}}{2}$}
  \coordinate (bl) at (-1, -1.5);
  \coordinate (br) at ( 1, -1.5);
  \coordinate (ml) at (-1, .8);
  \coordinate (nl) at (-1,-.8);
  \coordinate (mr) at (1,1);
  \coordinate (nr) at (1,-1);
  \coordinate (tl) at (-1.5,  1.75);
  \coordinate (tr) at ( -1,  3);
  \coordinate (ttl) at (-2.8,2);
  \coordinate (ttr) at (-1,3.5);
  \coordinate (Tl) at  (-2.8,3.5);
  \coordinate (Bl) at (-3,0);  
  \draw[>-] (bl) -- (ml) node[pos = 0, below] {$b$}   node[pos=.5, left] {\lab}
  node[pos= 0.5 , draw=\crd, fill=white,regular polygon, regular polygon sides=3,inner sep=0.9pt] {} node[pos=0.5, left]{\tria} 
  node[pos=.75, draw=\crd, fill= white, circle , inner sep=1.5pt]{} node[pos=.75, left]{\circ};
  \draw[>-] (br) -- (mr) node[pos = 0, below] {$c$} node[pos=.55,right] {\label}
  node[pos= 0.5 , draw=\crd, fill=white,regular polygon, regular polygon sides=3,inner sep=0.9pt] {} node[pos=0.5, left]{\triang} 
  node[pos=.7, draw=\crd, fill= white, circle , inner sep=1.5pt]{} node[pos=.7, right]{\circle};
  \draw[->-] (ml) -- (mr) node [pos= 0.4, above,yshift=.5mm] {\labe} node[pos=.8, draw=\crd, fill= white, circle , inner sep=1.5pt]{} node[pos=.82, above]{\circl} 
  node[pos= 0.25 , draw=\crd, fill=white,regular polygon, regular polygon sides=3,inner sep=0.9pt] {} 
  node[pos=.45, below,yshift=-.5mm] {\trian};
  \draw[->-] (nr) -- (nl) node [pos= 0.55, above,yshift=-.5mm] {\la} node[pos=.85, draw=\crd, fill= white, circle , inner sep=1.5pt]{} node[pos=.8, above]{\cir} 
  node[pos= 0.2 , draw=\crd, fill=white,regular polygon, regular polygon sides=3,inner sep=0.9pt] {} 
  node[pos=.5, below] {\tri};
  \draw[->-] (mr) .. controls +(0,1) and + (0, 0) .. (tr) node[pos = .5, right] {$b$};
  \draw[->-] (ml) -- (tl) node[pos=.5,left] {$c$};
  \draw[->-] (tl) -- (ttl) node[pos=.7,below,yshift=-2pt] {$c-1$} node[pos=.3, draw=\crd, fill= white, circle , inner sep=1.5pt]{} node[pos=.3, below]{\color{\labrd}$-\frac{s}{2}$} node[pos= 0.8 , draw=\crd, fill=white,regular polygon, regular polygon sides=3,inner sep=0.9pt] {} node[pos=0.5, above]{\color{\labrd}$-\lambda (c-2) +\frac{s}{2}$};
  \draw[->-] (tl) -- (tr) node[pos=.5,left] {$1$};
  \draw[->] (ttl) .. controls +(-.1,.1) and +(0,0) .. (Tl) node[pos=1,above] {$c$} node[pos=.4, draw=\crd, fill= white, circle , inner sep=1.5pt]{} node[pos=.4, right]{\color{\labrd}$\frac{s}{2}$} node[pos= 0.6 , draw=\crd, fill=white,regular polygon, regular polygon sides=3,inner sep=0.9pt] {} node[pos=0.6, right]{\color{\labrd}$\lambda (c-1) -\frac{s}{2}$};
  \draw[->] (tr) -- (ttr) node[pos=1,above] {$b+1$};
  \draw[->-] (Bl) .. controls +(0,1) and +(0,0) .. (ttl) node[pos=0.5,left] {$1$};
\end{scope}

%% file: sch/forkslide_iso_5.tex
\begin{scope}[font=\tiny]
  \def\la{$c-i$}
  \def\cir{\color{\labrd}$-\frac{s}{2}$}
  \def\tri{\color{\labrd}$-\lambda(c-i-1)+\frac{s}{2}$}
  \def\lab{$b+c+1-i$}
  \def\circ{\color{\labrd}$\frac{s}{2}$}
  \def\tria{\color{\labrd}$\lambda(b+c-i)-\frac{s}{2}$}
  \def\labe{$b+1-i$}
  \def\circl{\color{\labrd}$-\frac{s}{2}$}
  \def\trian{\color{\labrd}$-\lambda(b-i)+\frac{s}{2}$}
  \def\label{$i$}
  \def\circle{\color{\labrd}$\frac{\bar{s}}{2}$}
  \def\triang{\color{\labrd}$-\lambda(i-1)-\frac{\bar{s}}{2}$}
  \coordinate (bl) at (-1, -1.5);
  \coordinate (br) at ( 1.3, -1.5);
  \coordinate (tl) at (-1,  1.5);
  \coordinate (tr) at ( 1.3,  1.5);
  \coordinate (Bll) at (-1.75,-3);
  \coordinate (Bl) at (-.25,-3);
  \coordinate (Br) at (1.3,-3);
  \draw[->] (bl) -- (tl) node[pos = 0, below] {} node[pos = 1, above] {$c$}  node[pos=.5, right] {\lab}
  node[pos= 0.6 , draw=\crd, fill=white,regular polygon, regular polygon sides=3,inner sep=0.9pt] {} node[pos=0.6, right]{\tria} 
  node[pos=.4, draw=\crd, fill= white, circle , inner sep=1.5pt]{} node[pos=.4, left]{\circ}
  coordinate[pos = 0.75] (ml) 
  coordinate[pos = 0.25] (nl);
  \draw[->] (br) -- (tr) node[pos = 0, below] {} node[pos = 1, above] {$b+1$}  node[pos=.5,right] {\label}
  node[pos= 0.4 , draw=\crd, fill=white,regular polygon, regular polygon sides=3,inner sep=0.9pt] {} node[pos=0.4, left]{\triang} 
  node[pos=.65, draw=\crd, fill= white, circle , inner sep=1.5pt]{} node[pos=.65, right]{\circle}
  coordinate[pos = 0.85] (mr)  
  coordinate[pos = 0.15] (nr);
  \draw[->-] (ml) -- (mr) node [pos= 0.4, above,yshift=.5mm] {\labe} node[pos=.8, draw=\crd, fill= white, circle , inner sep=1.5pt]{} node[pos=.82, above]{\circl} 
  node[pos= 0.25 , draw=\crd, fill=white,regular polygon, regular polygon sides=3,inner sep=0.9pt] {} 
  node[pos=.45, below,yshift=-.5mm] {\trian};
  \draw[->-] (nr) -- (nl) node [pos= 0.4, above] {\la} node[pos=.8, draw=\crd, fill= white, circle , inner sep=1.5pt]{} node[pos=.75, above]{\cir} 
  node[pos= 0.25 , draw=\crd, fill=white,regular polygon, regular polygon sides=3,inner sep=0.9pt] {} 
  node[pos=.55, below,yshift=-2pt] {\tri};
  \draw[>-] (Br) -- (br) node[below,pos=0] {$c$}; 
  \draw[>-] (Bll) .. controls +(0,1) and + (0, 0) .. (bl) node[below,pos=0] {$1$};
  \draw[>-] (Bl) .. controls +(0,1) and + (0, 0) .. (bl)
  node[below,pos=0] {$b$};
\end{scope}

%% file: sch/forkslide_iso_6.tex
\begin{scope}[font=\tiny]
  \def\la{$c-i-1$}
  \def\cir{\color{\labrd}$-\frac{s}{2}$}
  \def\tri{\color{\labrd}$-\lambda(c-i-2)+\frac{s}{2}$}
  \def\lab{$b+c-i$}
  \def\circ{\color{\labrd}$\frac{s}{2}$}
  \def\tria{\color{\labrd}$\lambda(b+c-i-1)-\frac{s}{2}$}
  \def\labe{$b-i$}
  \def\circl{\color{\labrd}$-\frac{s}{2}$}
  \def\trian{\color{\labrd}$-\lambda(b-i-1)+\frac{s}{2}$}
  \def\label{$i+1$}
  \def\circle{\color{\labrd}$\frac{\bar{s}}{2}$}
  \def\triang{\color{\labrd}$-\lambda i-\frac{\bar{s}}{2}$}
  \coordinate (bl) at (-1, -1.5);
  \coordinate (br) at ( 1.3, -1.5);
  \coordinate (tl) at (-1,  1.5);
  \coordinate (tr) at ( 1.3,  1.5);
  \coordinate (Bll) at (-1.75,-3);
  \coordinate (Bl) at (-.25,-3);
  \coordinate (Br) at (1.3,-3);
  \draw[->] (bl) -- (tl) node[pos = 0, below] {} node[pos = 1, above] {$c$}  node[pos=.5, right] {\lab}
  node[pos= 0.57 , draw=\crd, fill=white,regular polygon, regular polygon sides=3,inner sep=0.9pt] {} node[pos=0.57, right]{\tria} 
  node[pos=.4, draw=\crd, fill= white, circle , inner sep=1.5pt]{} node[pos=.4, left]{\circ}
  coordinate[pos = 0.75] (ml) 
  coordinate[pos = 0.25] (nl);
  \draw[->] (br) -- (tr) node[pos = 0, below] {} node[pos = 1, above] {$b+1$}  node[pos=.5,right] {\label}
  node[pos= 0.4 , draw=\crd, fill=white,regular polygon, regular polygon sides=3,inner sep=0.9pt] {} node[pos=0.4, left]{\triang} 
  node[pos=.7, draw=\crd, fill= white, circle , inner sep=1.5pt]{} node[pos=.7, right]{\circle}
  coordinate[pos = 0.85] (mr)  
  coordinate[pos = 0.15] (nr);
  \draw[->-] (ml) -- (mr) node [pos= 0.4, above,yshift=.5mm] {\labe} node[pos=.8, draw=\crd, fill= white, circle , inner sep=1.5pt]{} node[pos=.82, above]{\circl} 
  node[pos= 0.25 , draw=\crd, fill=white,regular polygon, regular polygon sides=3,inner sep=0.9pt] {} 
  node[pos=.45, below,yshift=-.5mm] {\trian};
  \draw[->-] (nr) -- (nl) node [pos= 0.4, above] {\la} node[pos=.8, draw=\crd, fill= white, circle , inner sep=1.5pt]{} node[pos=.75, above]{\cir} 
  node[pos= 0.25 , draw=\crd, fill=white,regular polygon, regular polygon sides=3,inner sep=0.9pt] {} 
  node[pos=.55, below,yshift=-2pt] {\tri};
  \draw[>-] (Br) -- (br) node[below,pos=0] {$c$}; 
  \draw[>-] (Bll) .. controls +(0,1) and + (0, 0) .. (bl) node[below,pos=0] {$1$};
  \draw[>-] (Bl) .. controls +(0,1) and + (0, 0) .. (bl)
  node[below,pos=0] {$b$};
\end{scope}

%% file: sch/forkslide_iso_7.tex
\begin{scope}[font=\tiny]
  \def\la{$c-i$}
  \def\cir{\color{\labrd}$-\frac{s}{2}$}
  \def\tri{\color{\labrd}$-\lambda(c-i-1)+\frac{s}{2}$}
  \def\lab{}
  \def\circ{\color{\labrd}$\frac{s}{2}$}
  \def\tria{\color{\labrd}$\lambda(b+c-i)-\frac{s}{2}$}
  \def\labe{$b-i$}
  \def\circl{\color{\labrd}$-\frac{s}{2}$}
  \def\trian{\color{\labrd}$-\lambda(b-i-1)+\frac{s}{2}$}
  \def\label{$i$}
  \def\circle{\color{\labrd}$\frac{\bar{s}}{2}$}
  \def\triang{\color{\labrd}$-\lambda(i-1)-\frac{\bar{s}}{2}$}
  \coordinate (bl) at (-1, -1.5);
  \coordinate (br) at ( 1, -1.5);
  \coordinate (ml) at (-1, .8);
  \coordinate (nl) at (-1,-.8);
  \coordinate (mr) at (1,1);
  \coordinate (nr) at (1,-1);
  \coordinate (tl) at (-1,  1.75);
  \coordinate (tr) at (0,  3);
  \coordinate (ttl) at (-2.8,2);
  \coordinate (ttr) at (0,3.5);
  \coordinate (Tl) at  (-2.8,3.5);
  \coordinate (Bl) at (-2.8,0);  
  \draw[>-] (bl) -- (ml) node[pos = 0, below] {$b$}   node[pos=.5, left] {\lab}
  node[pos= 0.45 , draw=\crd, fill=white,regular polygon, regular polygon sides=3,inner sep=0.9pt] {} node[pos=0.45, left]{\tria} 
  node[pos=.9, draw=\crd, fill= white, circle , inner sep=1.5pt]{} node[pos=.9, left]{\circ};
  \draw[->-] (mr) .. controls +(0,1) and + (0, 0) .. (tr) node[pos = .5, right] {\label};
  \draw[>-] (br) -- (mr) node[pos = 0, below] {$c$} node[pos=.55,right] {}
  node[pos= 0.7 , draw=\crd, fill=white,regular polygon, regular polygon sides=3,inner sep=0.9pt] {} node[pos=0.7, left]{\triang} 
  node[pos=1, draw=\crd, fill= white, circle , inner sep=1.5pt]{} node[pos=1, left]{\circle};
  \draw[->-] (nr) -- (nl) node [pos= 0.45, above] {\la} node[pos=.8, draw=\crd, fill= white, circle , inner sep=1.5pt]{} node[pos=.75, above]{\cir} 
  node[pos= 0.2 , draw=\crd, fill=white,regular polygon, regular polygon sides=3,inner sep=0.9pt] {} 
  node[pos=.5, below] {\tri};
  \draw[->-] (ml) -- (tl) node[pos=0,right] {$b+c-i$};
  \draw[->-] (tl) -- (ttl) node[pos=.7,below,yshift=-2pt] {$c-1$} node[pos=.3, draw=\crd, fill= white, circle , inner sep=1.5pt]{} node[pos=.3, below]{\color{\labrd}$-\frac{s}{2}$};
  \draw[->-] (tl) .. controls +(0,.5) and + (0, 0) .. (tr) node[pos=.5,left] {$b+1-i$} 
  node[pos=.25, draw=\crd, fill= white, circle , inner sep=1.5pt]{} node[pos=.25, left]{\color{\labrd}$-\frac{s}{2}$}
  node[pos= 0.7 , draw=\crd, fill=white,regular polygon, regular polygon sides=3,inner sep=0.9pt] {} node[pos=0.7, left]{\color{\labrd}$-\lambda(b-i-1)+\frac{s}{2}$};
  \draw[->] (ttl) -- (Tl) node[pos=1,above] {$c$} node[pos=.4, draw=\crd, fill= white, circle , inner sep=1.5pt]{} node[pos=.4, right]{\color{\labrd}$\frac{s}{2}$};
  \draw[->] (tr) -- (ttr) node[pos=1,above] {$b+1$};
  \draw[->-] (Bl) .. controls +(0,1) and +(0,0) .. (ttl) node[pos=0,below] {$1$} node[pos= 0.4 , draw=\crd, fill=white,regular polygon, regular polygon sides=3,inner sep=0.9pt] {} node[pos=0.4, right]{\color{\labrd}$\lambda (c-1) -\frac{s}{2}$};
\end{scope}

%% file: sch/forkslide_iso_8.tex
\begin{scope}[font=\tiny]
  \def\la{$c-i-1$}
  \def\cir{\color{\labrd}$-\frac{s}{2}$}
  \def\tri{\color{\labrd}$-\lambda(c-i-2)+\frac{s}{2}$}
  \def\lab{}
  \def\circ{\color{\labrd}$\frac{s}{2}$}
  \def\tria{\color{\labrd}$\lambda(b+c-i-1)-\frac{s}{2}$}
  \def\labe{$b-i-1$}
  \def\circl{\color{\labrd}$-\frac{s}{2}$}
  \def\trian{\color{\labrd}$-\lambda(b-i-2)+\frac{s}{2}$}
  \def\label{$i+1$}
  \def\circle{\color{\labrd}$\frac{\bar{s}}{2}$}
  \def\triang{\color{\labrd}$-\lambda i-\frac{\bar{s}}{2}$}
  \coordinate (bl) at (-1, -1.5);
  \coordinate (br) at ( 1, -1.5);
  \coordinate (ml) at (-1, .8);
  \coordinate (nl) at (-1,-.8);
  \coordinate (mr) at (1,1);
  \coordinate (nr) at (1,-1);
  \coordinate (tl) at (-1,  1.75);
  \coordinate (tr) at (0,  3);
  \coordinate (ttl) at (-2.8,2);
  \coordinate (ttr) at (0,3.5);
  \coordinate (Tl) at  (-2.8,3.5);
  \coordinate (Bl) at (-2.8,0);  
  \draw[>-] (bl) -- (ml) node[pos = 0, below] {$b$}   node[pos=.5, left] {\lab}
  node[pos= 0.45 , draw=\crd, fill=white,regular polygon, regular polygon sides=3,inner sep=0.9pt] {} node[pos=0.45, left]{\tria} 
  node[pos=.9, draw=\crd, fill= white, circle , inner sep=1.5pt]{} node[pos=.9, left]{\circ};
  \draw[->-] (mr) .. controls +(0,1) and + (0, 0) .. (tr) node[pos = .5, right] {\label};
  \draw[>-] (br) -- (mr) node[pos = 0, below] {$c$} node[pos=.55,right] {}
  node[pos= 0.7 , draw=\crd, fill=white,regular polygon, regular polygon sides=3,inner sep=0.9pt] {} node[pos=0.7, left]{\triang} 
  node[pos=1, draw=\crd, fill= white, circle , inner sep=1.5pt]{} node[pos=1, left]{\circle};
  \draw[->-] (nr) -- (nl) node [pos= 0.4, above] {\la} node[pos=.8, draw=\crd, fill= white, circle , inner sep=1.5pt]{} node[pos=.75, above]{\cir} 
  node[pos= 0.2 , draw=\crd, fill=white,regular polygon, regular polygon sides=3,inner sep=0.9pt] {} 
  node[pos=.5, below] {\tri};
  \draw[->-] (ml) -- (tl) node[pos=0,right] {$b+c-i-1$};
  \draw[->-] (tl) -- (ttl) node[pos=.7,below,yshift=-2pt] {$c-1$} node[pos=.3, draw=\crd, fill= white, circle , inner sep=1.5pt]{} node[pos=.3, below]{\color{\labrd}$-\frac{s}{2}$};
  \draw[->-] (tl) .. controls +(0,.5) and + (0, 0) .. (tr) node[pos=.5,left] {$b+1-i-1$} 
  node[pos=.25, draw=\crd, fill= white, circle , inner sep=1.5pt]{} node[pos=.25, left]{\color{\labrd}$-\frac{s}{2}$}
  node[pos= 0.7 , draw=\crd, fill=white,regular polygon, regular polygon sides=3,inner sep=0.9pt] {} node[pos=0.7, left]{\color{\labrd}$-\lambda(b-i-2)+\frac{s}{2}$};
  \draw[->] (ttl) -- (Tl) node[pos=1,above] {$c$} node[pos=.4, draw=\crd, fill= white, circle , inner sep=1.5pt]{} node[pos=.4, right]{\color{\labrd}$\frac{s}{2}$};
  \draw[->] (tr) -- (ttr) node[pos=1,above] {$b+1$};
  \draw[->-] (Bl) .. controls +(0,1) and +(0,0) .. (ttl) node[pos=0,below] {$1$} node[pos= 0.4 , draw=\crd, fill=white,regular polygon, regular polygon sides=3,inner sep=0.9pt] {} node[pos=0.4, right]{\color{\labrd}$\lambda (c-1) -\frac{s}{2}$};
\end{scope}

%% file: sch/forkslide_movie_2.tex
\begin{scope}[font=\tiny]
  \def\la{$c-i$}
  \def\lab{}
  \def\labe{$b-i$}
  \def\label{$i$}
  \coordinate (bl) at (-1, -1.5);
  \coordinate (br) at ( 1, -1.5);
  \coordinate (ml) at (-1, .8);
  \coordinate (nl) at (-1,-.8);
  \coordinate (mr) at (1,1);
  \coordinate (nr) at (1,-1);
  \coordinate (tl) at (-1,  1.75);
  \coordinate (tr) at (0,  3);
  \coordinate (ttl) at (-2.8,2);
  \coordinate (ttr) at (0,3.5);
  \coordinate (Tl) at  (-2.8,3.5);
  \coordinate (Bl) at (-2.8,0);  
  \draw[>-] (bl) -- (ml) node[pos = 0, below] {$b$}   node[pos=.5, left] {\lab};
  \draw[->-] (mr) .. controls +(0,1) and + (0, 0) .. (tr) node[pos = .5, right] {\label};
  \draw[>-] (br) -- (mr) node[pos = 0, below] {$c$} node[pos=.55,right] {};
  \draw[->-] (nr) -- (nl) node [pos= 0.45, above] {\la};
  \draw[->-] (ml) -- (tl) node[pos=0,right] {$b+c-i$};
  \draw[->-] (tl) -- (ttl) node[pos=.7,below,yshift=-2pt] {$c-1$};
  \draw[->-] (tl) .. controls +(0,.5) and + (0, 0) .. (tr) node[pos=.6,left] {$b+1-i$};
  \draw[->] (ttl) -- (Tl) node[pos=1,above] {$c$};
  \draw[->] (tr) -- (ttr) node[pos=1,above] {$b+1$};
  \draw[->-] (Bl) .. controls +(0,1) and +(0,0) .. (ttl) node[pos=0,below]{$1$};
\end{scope}

%% file: sch/forkslide_movie_3.tex
\begin{scope}[font=\tiny]
  \def\la{$c-i$}
  \def\lab{}
  \def\labe{$b-i$}
  \def\label{$i$}
  \coordinate (bl) at (-1, -1.5);
  \coordinate (br) at ( 1, -1.5);
  \coordinate (ml) at (-1, .8);
  \coordinate (nl) at (-1,-.8);
  \coordinate (mr) at (1,1);
  \coordinate (nr) at (1,-1);
  \coordinate (tl) at (-1,  1.75);
  \coordinate (tr) at (0,  3);
  \coordinate (ttl) at (-2.8,2);
  \coordinate (ttr) at (0,3.5);
  \coordinate (Tl) at  (-2.8,3.5);
  \coordinate (Bl) at (-2.8,0);  
  \draw[>-] (bl) -- (ml) node[pos = 0, below] {$b$}   node[pos=.5, left] {\lab};
  \draw[->-] (mr) .. controls +(0,1) and + (0, 0) .. (tr) node[pos = .5, right] {\label};
  \draw[>-] (br) -- (mr) node[pos = 0, below] {$c$} node[pos=.55,right] {};
  \draw[->-] (nr) -- (nl) node [pos= 0.45, above] {\la};
  \draw[->-] (ml) -- (tl) node[pos=0,right] {$b+c-i$};
  \draw[->-] (tl) -- (ttl) node[pos=.7,below,yshift=-2pt] {$c-1$};
  \draw[->-] (tl) .. controls +(0,.5) and + (0, 0) .. (tr) node[pos=.6,left] {$1$};
  \draw[->-] (-.93,2) .. controls +(.5,-.3) and + (0, 0) .. (.34,2.6) node[pos=.4,below] {$b-i$};
  \draw[->] (ttl) -- (Tl) node[pos=1,above] {$c$};
  \draw[->] (tr) -- (ttr) node[pos=1,above] {$b+1$};
  \draw[->-] (Bl) .. controls +(0,1) and +(0,0) .. (ttl) node[pos=0,below]{$1$};
\end{scope}

%% file: sch/forkslide_movie_1.tex
\begin{scope}[font=\tiny]
  \def\la{$c-i$}
  \def\lab{}
  \def\labe{$b-i$}
  \def\label{$i$}
  \coordinate (bl) at (-1, -1.5);
  \coordinate (br) at ( 1, -1.5);
  \coordinate (ml) at (-1, .8);
  \coordinate (nl) at (-1,-.8);
  \coordinate (mr) at (1,1);
  \coordinate (nr) at (1,-1);
  \coordinate (tl) at (-1.5,  1.75);
  \coordinate (tr) at ( -1,  3);
  \coordinate (ttl) at (-2.8,2);
  \coordinate (ttr) at (-1,3.5);
  \coordinate (Tl) at  (-2.8,3.5);
  \coordinate (Bl) at (-2.8,0);  
  \draw[>-] (bl) -- (ml) node[pos = 0, below] {$b$}   node[pos=.5, left] {\lab};
  \draw[>-] (br) -- (mr) node[pos = 0, below] {$c$} node[pos=.55,right] {\label};
  \draw[->-] (ml) -- (mr) node [pos= 0.5, below] {\labe};
  \draw[->-] (nr) -- (nl) node [pos= 0.55, above,yshift=-.5mm] {\la};
  \draw[->-] (mr) .. controls +(0,1) and + (0, 0) .. (tr) node[pos = .6, right] {$b$};
  \draw[->-] (ml) -- (tl) node[pos=.5,right] {$c$};
  \draw[->-] (tl) -- (ttl) node[pos=.5,below] {$c-1$};
  \draw[->-] (tl) -- (tr) node[pos=.5,left] {$1$};
  \draw[->] (ttl) -- (Tl);
  \draw[->] (tr) -- (ttr) node[pos=1,above] {$b+1$};
  \draw[->-] (Bl) -- (ttl) node[pos=0,below] {$1$};
\end{scope}

%% file: sch/forkslide_movie_4.tex
\begin{scope}[font=\tiny]
  \def\la{$c-i$}
  \def\cir{}
  \def\tri{}
  \def\lab{}
  \def\circ{}
  \def\tria{}
  \def\labe{$b+1-i$}
  \def\circl{}
  \def\trian{}
  \def\label{$i$}
  \def\circle{}
  \def\triang{}
  \coordinate (bl) at (-1, -1.5);
  \coordinate (br) at ( 1.3, -1.5);
  \coordinate (tl) at (-1,  1.5);
  \coordinate (tr) at ( 1.3,  1.5);
  \coordinate (Bll) at (-1.75,-3);
  \coordinate (Bl) at (-.25,-3);
  \coordinate (Br) at (1.3,-3);
  \draw[->] (bl) -- (tl) node[pos = 0, below] {} node[pos = 1, above] {$c$}  node[pos=.6, right] {$b+c$} node[pos=.45,right] {$+1-i$}
  coordinate[pos = 0.75] (ml) 
  coordinate[pos = 0.25] (nl);
  \draw[->] (br) -- (tr) node[pos = 0, below] {} node[pos = 1, above] {$b+1$}  node[pos=.5,right] {\label}
  coordinate[pos = 0.85] (mr)  
  coordinate[pos = 0.15] (nr);
  \draw[->-] (ml) -- (mr) node [pos= 0.4, above,yshift=.5mm] {\labe};
  \draw[->-] (nr) -- (nl) node [pos= 0.4, above] {\la};
  \draw[>-] (Br) -- (br) node[below,pos=0] {$c$}; 
  \draw[>-] (Bll) .. controls +(0,1) and + (0, 0) .. (bl) node[below,pos=0] {$1$};
  \draw[>-] (Bl) .. controls +(0,1) and + (0, 0) .. (bl)
  node[below,pos=0] {$b$};
\end{scope}

%% file: sch/forkslide_movie_5.tex
\begin{scope}[font=\tiny]
  \def\la{$c-i$}
  \def\cir{}
  \def\tri{}
  \def\lab{}
  \def\circ{}
  \def\tria{}
  \def\labe{$b+1-i$}
  \def\circl{}
  \def\trian{}
  \def\label{$i$}
  \def\circle{}
  \def\triang{}
  \coordinate (bl) at (-1, -1.5);
  \coordinate (br) at ( 1.3, -1.5);
  \coordinate (tl) at (-1,  1.5);
  \coordinate (tr) at ( 1.3,  1.5);
  \coordinate (Bll) at (-1.75,-3);
  \coordinate (Bl) at (-.25,-3);
  \coordinate (Br) at (1.3,-3);
  \draw[->] (bl) -- (tl) node[pos = 0, below] {} node[pos = 1, above] {$c$}  node[pos=.55, right] {$b+c$} node[pos=.4,right] {$+1-i$}
  coordinate[pos = 0.75] (ml);
  \draw[->] (br) -- (tr) node[pos = 0, below] {} node[pos = 1, above] {$b+1$}  node[pos=.5,right] {\label}
  coordinate[pos = 0.85] (mr)  
  coordinate[pos = 0.15] (nr);
  \draw[->-] (ml) -- (mr) node [pos= 0.4, above,yshift=.5mm] {\labe};
  \draw[>-] (Br) -- (br) node[below,pos=0] {$c$}; 
  \draw[>-] (Bll) .. controls +(0,1) and + (0, 0) .. (bl) node[below,pos=0] {$1$};
  \draw[>-] (Bl) .. controls +(0,1) and + (0, 0) .. (bl) coordinate[pos = 0.5] (nl)
  node[below,pos=0] {$b$};
  \draw[->-] (nr) -- (nl) node [pos= 0.5, below] {\la};
\end{scope}

%% file: sch/forkslide_movie_6.tex
\begin{scope}[font=\tiny]
  \def\la{$c-i$}
  \def\cir{}
  \def\tri{}
  \def\lab{}
  \def\circ{}
  \def\tria{}
  \def\labe{$1$}
  \def\circl{}
  \def\trian{}
  \def\label{$i$}
  \def\circle{}
  \def\triang{}
  \coordinate (bl) at (-1, -1.5);
  \coordinate (br) at ( 1.3, -1.5);
  \coordinate (tl) at (-1,  1.5);
  \coordinate (tr) at ( 1.3,  1.5);
  \coordinate (Bll) at (-1.75,-3);
  \coordinate (Bl) at (-.25,-3);
  \coordinate (Br) at (1.3,-3);
  \draw[->] (bl) -- (tl) node[pos = 0, below] {} node[pos = 1, above] {$c$}  node[pos=.35, right] {$b+c$} node[pos=.2,right] {$+1-i$}
  coordinate[pos = 0.75] (ml) coordinate[pos = 0.5] (mll);
  \draw[->] (br) -- (tr) node[pos = 0, below] {} node[pos = 1, above] {$b+1$}  node[pos=.5,right] {\label}
  coordinate[pos = 0.85] (mr) coordinate[pos = 0.6] (mrr) coordinate[pos = 0.15] (nr);
  \draw[->-] (ml) -- (mr) node [pos= 0.4, above] {\labe};
  \draw[->-] (mll) -- (mrr) node [pos= 0.4, above] {$b-i$};
  \draw[>-] (Br) -- (br) node[below,pos=0] {$c$}; 
  \draw[>-] (Bll) .. controls +(0,1) and + (0, 0) .. (bl) node[below,pos=0] {$1$};
  \draw[>-] (Bl) .. controls +(0,1) and + (0, 0) .. (bl) coordinate[pos = 0.5] (nl)
  node[below,pos=0] {$b$};
  \draw[->-] (nr) -- (nl) node [pos= 0.5, below] {\la};
\end{scope}

%% file: sch/forkslide_movie_7.tex
\begin{scope}[font=\tiny]
  \def\la{$c-i$}
  \def\cir{\color{\labrd}$-\frac{s}{2}$}
  \def\tri{\color{\labrd}$-\lambda(c-i-1)+\frac{s}{2}$}
  \def\lab{}
  \def\circ{\color{\labrd}$\frac{s}{2}$}
  \def\tria{\color{\labrd}$\lambda(b+c-i-1)-\frac{s}{2}$}
  \def\labe{$b-i$}
  \def\circl{\color{\labrd}$-\frac{s}{2}$}
  \def\trian{\color{\labrd}$-\lambda(b-i-1)+\frac{s}{2}$}
  \def\label{$i$}
  \def\circle{\color{\labrd}$\frac{\bar{s}}{2}$}
  \def\triang{\color{\labrd}$-\lambda(i-1)-\frac{\bar{s}}{2}$}
  \coordinate (bl) at (-1, -1.5);
  \coordinate (br) at ( 1, -1.5);
  \coordinate (ml) at (-1, .8);
  \coordinate (nl) at (-1,-.8);
  \coordinate (mr) at (1,1);
  \coordinate (nr) at (1,-1);
  \coordinate (tl) at (-2,  1.5);
  \coordinate (tr) at ( -1,  3.5);
  \coordinate (ttl) at (-2,2.5);
  \coordinate (ttr) at (-1,4);
  \coordinate (Tl) at  (-3,3.5);
  \coordinate (Bl) at (-3,0);  
  \draw[>-] (bl) -- (ml) node[pos = 0, below] {$b$}   node[pos=.5, left] {\lab};
  \draw[>-] (br) -- (mr) node[pos = 0, below] {$c$} node[pos=.55,right] {\label};
  \draw[->-] (ml) -- (mr) node [pos= 0.4, above,yshift=.5mm] {\labe};
  \draw[->-] (nr) -- (nl) node [pos= 0.55, above,yshift=-.5mm] {\la};
  \draw[->-] (mr) .. controls +(0,2) and + (0, 0) .. (tr) node[pos = .5, right] {$b$};
  \draw[->-] (ml) -- (tl) node[pos=.7,below] {$c$};
  \draw[->-] (tl) -- (ttl) node[pos=.5,left] {$c+1$} ;
  \draw[->-] (ttl) -- (tr) node[pos=.5,above] {$1$};
  \draw[->] (ttl) .. controls +(-.8,.4) and +(0,0) .. (Tl) node[pos=1,above] {$c$};
  \draw[->] (tr) -- (ttr) node[pos=1,above] {$b+1$};
  \draw[->-] (Bl) .. controls +(0,1) and +(0,0) .. (tl) node[pos=0.5,left] {$1$};
\end{scope}

%% file: Z_uncolored.tex
In \cite{GR25}, the $\Witt$-equivariant (uncolored) $\gl_N$-Khovanov--Rozansky homology is defined. We briefly recall in this subsection the definition and useful properties of it. 

We construct a hypercube complex for any link diagrams by specifying locally a length two complex at each crossing according to its sign:

\begin{equation}
  \NB{\tikz[xscale = -0.6]{\input{\imagesfolder/sch/positive_crossing}}} :=
\NB{\tikz[xscale = 3.5, yscale = 3]
{\node (i0) at (0.2, 0){$\NB{\tikz[font= \tiny,scale=0.6]{\input{\imagesfolder/sch/reidem_lemm_14}}}$};
\node (i1) at (-1, 0) {$q\NB{\tikz[font= \tiny,scale=0.6]{\input{\imagesfolder/sch/def_homo_cross_pos}}} $};
\draw[->] (i1) -- (i0) coordinate[pos=0.45] (a);
\node[above] at (a) {\NB{\tikz[font=\tiny, scale=.5]{\input{\imagesfolder/sch/foam_unzip_empty}}}};
  }}
\end{equation}

\begin{equation}
   \NB{\tikz[xscale =0.6]{\input{\imagesfolder/sch/positive_crossing}}}:=
    \NB{ \tikz[xscale = 3.5, yscale = 3]{
    \node (i0) at (-1, 0) {$ \, \NB{\tikz[font= \tiny, scale=0.6]{\input{\imagesfolder/sch/reidem_lemm_14}}}$};
    \node (i1) at (0.2, 0) {$q^{-1}\NB{\tikz[font= \tiny, scale=0.6]{\input{\imagesfolder/sch/def_homo_cross_neg}}}$ };
    \draw[->] (i0) -- (i1) coordinate[pos=0.6] (b);
    \node[above] at (b) {\NB{\tikz[font=\tiny, scale=.5]{\input{\imagesfolder/sch/foam_zip_empty}}}}; }} \,
\end{equation}.
In both complexes, the term marked by a zigzag lies in cohomological degree $0$. For clarity, $\ssp_N(\cdot)$ has been omitted in these diagrams. 

For a link $L$, we define
$
\KRW_N(L,\scalars) := \KRW_{N,\lambda,s}(L,\scalars),
$
the (uncolored) Khovanov--Rozansky $\mathfrak{gl}_N$-homology of $L$ with coefficients in a ring $\scalars$, equipped with an action of $\Witt$.
To avoid cluttered diagrams, we will omit $\KRW_N(\cdot)$.

\begin{rmk}
    The attentive reader will notice that the conventions have been reversed between negative and positive crossings compared to \cite{GR25}. The purpose is to align the conventions of the colored case with those of \cite{ETW17}. 
\end{rmk}

\begin{prop}
    \label{uncol_reidem}
    There are isomorphisms in the relative homotopy canopoly:
    \begin{align} 
    \label{uncol_reidem_1}
    \NB{\tikz[font=\tiny, scale=0.6]{\input{\imagesfolder/sch/reidem1_1}}} \cong t^{-1}q^{N-1}\NB{\tikz[font=\tiny, scale=0.6]{\input{\imagesfolder/sch/reidem1_4}}} \qquad &\text{and} \qquad   \NB{\tikz[font=\tiny, scale=0.6, yscale=-1]{\input{\imagesfolder/sch/reidem1_3}}} \cong t q^{1-N} \NB{\tikz[font=\tiny, scale=0.6]{\input{\imagesfolder/sch/reidem1_2}}}. \\ 
    \label{uncol_reidem_2}
    \NB{\tikz[font=\tiny,scale=.7]{\input{\imagesfolder/sch/reidem2_1}}} \cong \NB{\tikz[font=\tiny,scale=.7]{\input{\imagesfolder/sch/reidem2_2}}} \cong \NB{\tikz[font=\tiny, xscale=-1,scale=.7]{\input{\imagesfolder/sch/reidem2_1}}} \qquad &\text{and} \qquad \NB{\tikz[font=\tiny,xscale=.8]{\input{\imagesfolder/sch/reidem2b_1}}} \cong \NB{\tikz[font=\tiny,xscale=.8]{\input{\imagesfolder/sch/reidem2b_2}}} \\
    \label{uncol_reidem_3}
    \NB{\tikz[scale=0.5]{\input{\imagesfolder/sch/reidem3_1}}} &\cong \NB{\tikz[ xscale=-.5,yscale=.5]{\input{\imagesfolder/sch/reidem3_1_1}}}.
    \end{align}
\end{prop}

\begin{lem}[{Forktwist lemma}]
\label{forktwist_u}
There are the following isomorphisms in the relative homotopy canopoly:
\begin{equation} 
\NB{\tikz[font=\tiny, scale=0.6, yscale=1]{\input{\imagesfolder/sch/forktwist_1_u}}} \cong q^2\NB{\tikz[font=\tiny, scale=0.6]{\input{\imagesfolder/sch/forktwist_2c_u}}} \hspace{3mm}\text{and}\hspace{3mm} \NB{\tikz[font=\tiny, scale=0.6, xscale=-1]{\input{\imagesfolder/sch/forktwist_1_u}}} \cong q^{-2}\NB{\tikz[font=\tiny, scale=0.6]{\input{\imagesfolder/sch/forktwist_2d_u}}}.
\end{equation}
\end{lem}

%% file: sch/positive_crossing.tex
\begin{scope}[font=\tiny]
  \draw[->] (0.5, -0.5) ..controls +(0,0.3) and +(0,-0.3) .. (-0.5, 0.5) node[pos=0,below] {$1$};
  \fill[white] (0,0) circle (2mm);
  \draw[->] (-0.5, -0.5) ..controls +(0,0.3) and +(0,-0.3) .. (0.5,0.5) node[pos=0,below] {$1$};
\end{scope}

%% file: sch/reidem_lemm_14.tex
\begin{scope}
    \coordinate (bl) at (-0.45, -1);
    \coordinate (br) at ( 0.4, -1);
    \coordinate (tl) at (-0.45,  1);
    \coordinate (tr) at ( 0.4,  1);
    \draw[->] (bl) -- (tl) node[pos=0,below]{$1$};
    \draw[->] (br) -- (tr) node[pos=0,below]{$1$};
\end{scope}

%% file: sch/foam_unzip_empty.tex
\begin{scope}
  \begin{scope}
    \coordinate (L1) at (0.2,0.4);
    \coordinate (L2) at (0,0);
    \coordinate (R1) at (2.2,0.4);
    \coordinate (R2) at (2,0);
    \coordinate (ML) at (0.6, 0.2);
    \coordinate (MR) at (1.6, 0.2);
    \draw[double] (ML) -- (MR) node[below, midway] {}; %[->-]
    \draw (MR) .. controls +(0, 0) and +(-0.3,0) .. (R1) ;
    \draw (MR) .. controls +(0, 0) and +(-0.3,0) .. (R2);
    \draw (L1) .. controls +( 0.3, 0) and +(0,0) .. (ML);
    \draw (L2) .. controls +( 0.3, 0) and +(0,0) .. (ML);
  \end{scope}  
 \begin{scope}[yshift = 1cm]
    \coordinate (L1B) at (0.2,0.4);
    \coordinate (L2B) at (0,0);
    \coordinate (R1B) at (2.2,0.4);
    \coordinate (R2B) at (2,0);
    \draw[->-] (L1B) .. controls +( 0, 0) and +(0,0) .. (R1B) node [left, pos
    = 0] {};
    \draw[->-] (L2B) .. controls +( 0, 0) and +(0,0) .. (R2B) node [left, pos
    = 0] {};
 \end{scope}  
  \draw (R1) -- (R1B);
  \draw (R2) -- (R2B);
  \draw (L1) -- (L1B);
  \draw (L2) -- (L2B);
  \draw (ML) .. controls +(0, 0.6) and +(0, 0.6) .. (MR);
\end{scope}

%% file: sch/foam_zip_empty.tex
\begin{scope}
  \begin{scope}
    \coordinate (L1) at (0.2,0.4);
    \coordinate (L2) at (0,0);
    \coordinate (R1) at (2.2,0.4);
    \coordinate (R2) at (2,0);
    \coordinate (ML) at (0.6, 0.2);
    \coordinate (MR) at (1.6, 0.2);
    \draw[double] (ML) -- (MR) node[above, midway] {};
    \draw (MR) .. controls +(0, 0) and +(-0.3,0) .. (R1) ;
    \draw (MR) .. controls +(0, 0) and +(-0.3,0) .. (R2);
    \draw (L1) .. controls +( 0.3, 0) and +(0,0) .. (ML);
    \draw (L2) .. controls +( 0.3, 0) and +(0,0) .. (ML);
  \end{scope}  
 \begin{scope}[yshift = -1cm]
    \coordinate (L1B) at (0.2,0.4);
    \coordinate (L2B) at (0,0);
    \coordinate (R1B) at (2.2,0.4);
    \coordinate (R2B) at (2,0);
    \draw[->-] (L1B) .. controls +( 0, 0) and +(0,0) .. (R1B) node [right, pos
    = 1] {};
    \draw[->-] (L2B) .. controls +( 0, 0) and +(0,0) .. (R2B) node [right, pos
    = 1] {};
 \end{scope}  
  \draw (R1) -- (R1B);
  \draw (R2) -- (R2B);
  \draw (L1) -- (L1B);
  \draw (L2) -- (L2B);
  \draw (ML) .. controls +(0, -0.6) and +(0, -0.6) .. (MR);
\end{scope}

%% file: sch/reidem1_1.tex
\begin{scope}
  \draw [<-] (0, 1) -- (0, 0.5) .. controls +(0,-0.5) and +(0, 0.5) .. (1, -0.5) arc (180:360:0.5) -- (2,0) ;  
    \fill[white] (0.5, 0) circle (1mm);
  \draw (0, -1) -- (0, -0.5) .. controls +(0,0.5) and +(0, -0.5) .. (1, 0.5) arc (180:0:0.5) -- (2,0);
  \node at (0,-1.3) {$1$};
\end{scope}

%% file: sch/reidem1_4.tex
\begin{scope}
  \draw [->] (0, -1) -- +(0,2) node[pos=0,below] {$1$};
       \coordinate (B) at (0, 0);
        \filldraw[draw= \crd, fill = \csrd] (B) circle (1mm)
  node[right]{\color{\crd}$\frac{1}{2}$};
\end{scope}

%% file: sch/reidem1_3.tex
\begin{scope}[xshift = 2.5cm]
  \draw  (0, 1) -- (0, 0.5) .. controls +(0,-0.5) and +(0, 0.5)
  .. (1, -0.5) arc (180:360:0.5) -- (2,0);  
    \fill[white] (0.5, 0) circle (1mm);
  \draw[<-] (0, -1) -- (0, -0.5) .. controls +(0,0.5) and +(0, -0.5)
  .. (1, 0.5) arc (180:0:0.5) -- (2,0);
  \node at (0,1.3) {$1$};
\end{scope}

%% file: sch/reidem1_2.tex
\begin{scope}
  \draw [->] (0, -1) -- +(0,2) node[pos=0,below]{$1$};
       \coordinate (B) at (0, 0);
        \filldraw[draw= \crd, fill = \csrd] (B) circle (1mm)
  node[right]{\color{\crd}$-\frac{1}{2}$};
\end{scope}

%% file: sch/reidem2_2.tex
\begin{scope}
    \coordinate (bl) at (-0.5, -1);
    \coordinate (br) at ( 0.5, -1);
    \coordinate (tl) at (-0.5,  1);
    \coordinate (tr) at ( 0.5,  1);
    \draw[->] (bl) -- (tl) node[pos=0, below] {$1$};
    \draw[->] (br) -- (tr) node[pos=0, below] {$1$};
\end{scope}

%% file: sch/reidem2b_1.tex
\begin{scope}[xshift = 2.5cm]
  \draw [<-](0, 0) .. controls +(1,.5) 
  .. (2, 0) node[pos=1, right] {$1$};
    \draw [->](0, 1) .. controls +(1,-.5) 
  .. (2, 1) node[pos=0, left] {$1$};
\end{scope}

%% file: sch/forktwist_1_u.tex
\begin{scope}
	\draw[<-] (0.5,2) -- (0.5,1) node[pos=0,above] {$2$};
	\draw[-<] (1, 0.5) arc (0:180:0.5) -- (0,0.5) .. controls +(0,-0.5) and +(0, 0.5)
		.. (1, -0.5) -- (1,-1) node[pos=1,below]{$1$};
		\fill[white] (0.5, 0) circle (1mm);
	\draw[>-] (0, -1) -- (0, -0.5) node[pos=0,below]{$1$} .. controls +(0,0.5) and +(0, -0.5)
		.. (1, 0.5) ;
\end{scope}

%% file: sch/forktwist_2c_u.tex
\begin{scope}[]
	\draw[>-<] (1,-1) -- (1, 0.5) node[pos=0,below] {$1$} arc (0:180:0.5) -- (0,0.5) -- ( 0,-1) node[pos=1,below] {$1$};
	\draw[<-] (0.5,2) -- (0.5,1) node[pos=1, draw=\crd, fill= white, circle , inner sep=1.5pt]{} node[pos=1, yshift=5pt, right]{\textcolor{\crd}{$1$}} node [pos=0,above] {$2$};
\end{scope}

%% file: sch/forktwist_2d_u.tex
\begin{scope}[]
	\draw[>-<] (1,-1) -- (1, 0.5) node[pos=0,below] {$1$} arc (0:180:0.5) -- (0,0.5) -- ( 0,-1) node[pos=1,below] {$1$};
	\draw[<-] (0.5,2) -- (0.5,1) node[pos=1, draw=\crd, fill= white, circle , inner sep=1.5pt]{} node[pos=1, yshift=5pt, right]{\textcolor{\crd}{$-1$}} node [pos=0,above] {$2$};
\end{scope}

%% file: Z_R23.tex
In this section we prove the invariance of $\colcom_N$ under colored Reidemeister moves.

\begin{prop}
    \label{R2_col}
    Let $a, b \in \N^*$, there are isomorphisms in the relative homotopy canopoly:
    \begin{equation}
        \NB{\tikz[font=\tiny,scale=.7]{\input{\imagesfolder/sch/reidem2_1_col}}} \cong \NB{\tikz[font=\tiny,scale=.7]{\input{\imagesfolder/sch/reidem2_2_col}}} \cong \NB{\tikz[font=\tiny, xscale=-1,scale=.7]{\input{\imagesfolder/sch/reidem2_1_col}}}.
    \end{equation}
\end{prop}

\begin{proof}
    We carry out the proof for the first isomorphism:
    \[ \NB{\tikz[font=\tiny,scale=.7]{\input{\imagesfolder/sch/reidem2_1_col}}} \cong \NB{\tikz[font=\tiny,scale=.7]{\input{\imagesfolder/sch/reidem2_2_col}}}.\]
    For this, we proceed by induction on $n = a + b$. The base case is $a = b = 1$, which is the uncolored case (Proposition \ref{uncol_reidem}). We now assume that $a + b = n + 1 > 2$, and we blist the leg with label greater than 2. In what follows, we assume that $a \geq 2$. Using the forkslide isomorphism together with the induction hypothesis, we obtain:
    \[\NB{\tikz[font=\tiny,scale=.7]{\input{\imagesfolder/sch/reidem2_pf_1}}} \cong \NB{\tikz[font=\tiny,scale=.7]{\input{\imagesfolder/sch/reidem2_pf_2}}} \cong \NB{\tikz[font=\tiny,scale=.7]{\input{\imagesfolder/sch/reidem2_pf_3}}} \cong \NB{\tikz[font=\tiny,scale=.7]{\input{\imagesfolder/sch/reidem2_pf_4}}}  \cong \NB{\tikz[font=\tiny,scale=.7]{\input{\imagesfolder/sch/reidem2_pf_5}}}. \]
    Finally, by using the blist hypothesis \ref{blisthyp}, we obtain the desired isomorphism. The construction of the second isomorphism is proved in the same way.
\end{proof}

The proof of the invariance of $\colcom$ under the second Reidemeister 2 move of Proposition \ref{R2_col_2}, proceeds in the same way as for the Proposition \ref{R2_col}.
\begin{prop}
    \label{R2_col_2}
    Let $a, b \in \N^*$, there are isomorphisms in the relative homotopy canopoly:
    \begin{equation}
         \NB{\tikz[font=\tiny,xscale=.8]{\input{\imagesfolder/sch/reidem2b_2_col}}} \cong \NB{\tikz[font=\tiny,xscale=.8]{\input{\imagesfolder/sch/reidem2b_1_col}}}.
    \end{equation}
\end{prop}

\begin{prop}
    \label{reidem_3}
    Let $a, b, c \in \N^*$, there are isomorphisms in the relative homotopy canopoly:
    \begin{equation}
        \NB{\tikz[font=\tiny,scale=.6]{\input{\imagesfolder/sch/reidem3_col_1}}} \cong \NB{\tikz[font=\tiny,scale=.6]{\input{\imagesfolder/sch/reidem3_col_2}}}.
    \end{equation}
\end{prop}

\begin{proof}
    As in the case of Reidemeister 2, we proceed by induction on $a+b+c = n$, using the blist hypothesis (Proposition~\ref{blisthyp}) and the following sequence of isomorphisms, in the case where $a\geq2$ for the inductive step:
    \[ \NB{\tikz[font=\tiny,scale=.6]{\input{\imagesfolder/sch/reidem3_pf_1}}} \cong  \NB{\tikz[font=\tiny,scale=.6]{\input{\imagesfolder/sch/reidem3_pf_2}}} \cong \NB{\tikz[font=\tiny,scale=.6]{\input{\imagesfolder/sch/reidem3_pf_3}}}  \cong \NB{\tikz[font=\tiny,scale=.6]{\input{\imagesfolder/sch/reidem3_pf_4}}} \cong \NB{\tikz[font=\tiny,scale=.6]{\input{\imagesfolder/sch/reidem3_pf_5}}} \cong \NB{\tikz[font=\tiny,scale=.6]{\input{\imagesfolder/sch/reidem3_pf_6}}}.\]
\end{proof}

%% file: sch/reidem2_2_col.tex
\begin{scope}
    \coordinate (bl) at (-0.5, -1);
    \coordinate (br) at ( 0.5, -1);
    \coordinate (tl) at (-0.5,  1);
    \coordinate (tr) at ( 0.5,  1);
    \draw[->] (bl) -- (tl) node[pos=0, below] {$a$};
    \draw[->] (br) -- (tr) node[pos=0, below] {$b$};
\end{scope}

%% file: sch/reidem2_pf_5.tex
\begin{scope}[font=\tiny]
  \draw[->] (0.5, -0.5) ..controls +(-.3,1) and +(-.3,-1) ..   (0.5, 1.55) node[pos=0,below] {$b$};
  \draw[->] (-0.5, -0.5)  ..controls +(.3,1) and +(.3,-1) .. (-0.5,1.55) node[pos=0,below,xshift=5pt] {$a-1$};
  \draw (-.9,-.5) ..controls+(0,.5) and +(-.3,-.3).. (-.43,1.31) node[pos=0,below] {$1$};
\end{scope}

%% file: sch/reidem2b_1_col.tex
\begin{scope}[xshift = 2.5cm]
  \draw [<-](0, 0) .. controls +(1,.5) 
  .. (2, 0) node[pos=1, right] {$b$};
    \draw [->](0, 1) .. controls +(1,-.5) 
  .. (2, 1) node[pos=0, left] {$a$};
\end{scope}

%% file: Z_R1.tex
To prove the invariance under Reidemeister 1, the following forktwist lemma is needed.

\begin{lem}
\label{forktwist}
For any $a,b \in \N^*$, there are the following isomorphisms in the relative homotopy canopoly:
\begin{equation} 
\NB{\tikz[font=\tiny, scale=0.6, xscale=1]{\input{\imagesfolder/sch/forktwist_1}}} \cong q^{-a-1}\NB{\tikz[font=\tiny, scale=0.6]{\input{\imagesfolder/sch/forktwist_2d}}} \hspace{3mm}\text{and}\hspace{3mm} \NB{\tikz[font=\tiny, scale=0.6,xscale=-1]{\input{\imagesfolder/sch/forktwist_1a}}} \cong q^{a+1}\NB{\tikz[font=\tiny, scale=0.6]{\input{\imagesfolder/sch/forktwist_2c}}},
\end{equation}

\end{lem}

\begin{proof}
    We construct the first isomorphism:
\begin{equation} \label{forktwist1}
\NB{\tikz[font=\tiny, scale=0.6, yscale=1]{\input{\imagesfolder/sch/forktwist_1b}}} \cong q^{-a-1}\NB{\tikz[font=\tiny, scale=0.6]{\input{\imagesfolder/sch/forktwist_2b}}}.
\end{equation}

The first complex is defined to be:
\[
\NB{\tikz[font=\tiny, scale=0.6]{\input{\imagesfolder/sch/forktwist_1b}}} =
\NB{
	\tikz[xscale = 3, yscale = 3]{
		\node (AT) at (0,0) {$t^{-1}\NB{\tikz[font=\tiny, scale=0.7]{\input{\imagesfolder/sch/forktwist_2}}}$};
		\node (CT) at (1.5,0) {$q^{-1}\NB{\tikz[font=\tiny, scale=0.7]{\input{\imagesfolder/sch/forktwist_3}}}$};
		\draw[->] (AT) -- (CT) node[pos=0.5, above]{$d^+_{-1}$};
	}
}.
\]
We do not specify the labels of the red dots on the leftmost element of the complex, since they are easy to recover from the definition of the complex and are not important for our purposes here.

Using red dots isomorphisms we have: 
\[\NB{\tikz[font=\tiny, scale=0.7]{\input{\imagesfolder/sch/forktwist_3}}} \cong \NB{\tikz[font=\tiny, scale=0.7]{\input{\imagesfolder/sch/forktwist_3b}}}\]

We embed these two complexes in a short exact sequence of $\Witt$-equivariant complexes that splits when one forgets the $\Witt$ action:

\[
\NB{\tikz[xscale = 3, yscale = 4]{
	\node (i0) at (0, 0) {$t^{-1}\NB{\tikz[font= \tiny, scale=.7]{\input{\imagesfolder/sch/forktwist_2}}}$};
	\node (i1) at (1.5, 0) { $q^{-1}\NB{\tikz[font= \tiny,scale=0.7]{\input{\imagesfolder/sch/forktwist_3b}}}$ };
	\node (i2) at (1.5, 1.1) {$q^{-a-1}\NB{\tikz[font= \tiny, scale=0.7]{\input{\imagesfolder/sch/forktwist_2b}}}$};
	\node (i3) at (1.5, -1.2) {$[a]\NB{\tikz[font= \tiny, scale=0.7]{\input{\imagesfolder/sch/forktwist_2b}}}$};
	\node (i4) at (0, -1.2) {$[a]\NB{\tikz[font= \tiny, scale=0.7]{\input{\imagesfolder/sch/forktwist_2b}}}$};
	\draw[->] (i0) -- (i1)  node[pos=0.5, above] {$d^+_{-1}$};
	\draw[<-] (i1) -- (i2) node[pos=0.5, left] {$\mapdcup$} ;
	\draw[<-] (i4) -- (i0)  node[pos=0.5, left] {$\begin{pmatrix}
	    \mapdcap \circ d_{-1}^+ \\ \mapdcap_{\deco_1} \circ d_{-1}^+\\ \vdots \\ \mapdcap_{\deco_{a-1}} \circ d_{-1}^+
	\end{pmatrix}$} ;
	\draw[<-] (i3) -- (i1)  node[pos=0.5, right] {$\begin{pmatrix}
	    \mapdcap \\ \mapdcap_{\deco_1} \vdots \\ \mapdcap_{\deco_{a-1}}
	\end{pmatrix}$} ;
	\draw[->] (i4) -- (i3)  node[pos=0.5, above] {$\Id$};
}},
\]
where $\mapdcap_{\deco_i}$ denotes a digon cap decorated by $p_i$ on the facet of thickness $a+1$.  
There is a $\Witt$-structure on the bottom complex, which can be recovered by viewing all the complexes in the category $\mathcal{C}^{\Witt}(\scalars_N\text{-mod})$ via the functor $\ssp_N$ and interpreting the bottom complex as the cokernel of the $\Witt$-equivariant map between the other two complexes. 
Since the bottom complex is contractible, and thanks to Corollary~\ref{lemme_RHC}, we obtain the desired isomorphism.
\end{proof}

\begin{prop}
    Let $a,b \in \N^*$, there are the following isomorphisms in the relative homotopy canopoly:
    \begin{equation} 
    \NB{\tikz[font=\tiny, scale=0.6, xscale=1]{\input{\imagesfolder/sch/forktwist_1_p}}} \cong q^{-ab-1}\NB{\tikz[font=\tiny, scale=0.6]{\input{\imagesfolder/sch/forktwist_2d_p}}} \hspace{3mm}\text{and}\hspace{3mm} \NB{\tikz[font=\tiny, scale=0.6,xscale=-1]{\input{\imagesfolder/sch/forktwist_1a_p}}} \cong q^{ab+1}\NB{\tikz[font=\tiny, scale=0.6]{\input{\imagesfolder/sch/forktwist_2c_p}}},
    \end{equation}
\end{prop}

\begin{proof}
    We focus on the proof of the first isomorphism, the proof of the second one is left as an exercise.  
    The argument proceeds by induction on $b$.  
   We present the inductive step, the base case can be obtained by setting $b=1$. We first use the forkslide isomorphisms of Proposition \ref{forkslide}, the forktwist lemma for a positive crossing of Proposition \ref{forktwist} and the inductive hypothesis:
    \[\NB{\tikz[font=\tiny, scale=0.6, xscale=1]{\input{\imagesfolder/sch/forktwist_ab_1}}} \cong
    \NB{\tikz[font=\tiny, scale=0.6, xscale=1]{\input{\imagesfolder/sch/forktwist_ab_2}}} \cong
    \NB{\tikz[font=\tiny, scale=0.6, xscale=1]{\input{\imagesfolder/sch/forktwist_ab_3}}} \cong
    q^{-a-1-(a+1)b-1}\NB{\tikz[font=\tiny, scale=0.6, xscale=1]{\input{\imagesfolder/sch/forktwist_ab_4}}}.\]
    We conclude by using the forktwist lemma for a negative crossing of Proposition \ref{forktwist}:
    \[\NB{\tikz[font=\tiny, scale=0.6, xscale=1]{\input{\imagesfolder/sch/forktwist_ab_4}}} \cong 
    \NB{\tikz[font=\tiny, scale=0.6, xscale=1]{\input{\imagesfolder/sch/forktwist_ab_5}}} \cong
    q^{b+1} \NB{\tikz[font=\tiny, scale=0.6, xscale=1]{\input{\imagesfolder/sch/forktwist_ab_6}}}.\]
\end{proof}

\begin{prop}
    Let $a \in \N^*$, there are isomorphisms in the relative homotopy canopoly:
    \begin{equation} 
        \NB{\tikz[font=\tiny, scale=0.6,yscale=-1]{\input{\imagesfolder/sch/R1_col_1}}}  \cong t^a q^{a(2+a-N)-2} \NB{\tikz[font=\tiny, scale=0.6]{\input{\imagesfolder/sch/R1_col_4}}}  \qquad \text{and} \qquad \NB{\tikz[font=\tiny, scale=0.6, yscale=1]{\input{\imagesfolder/sch/R1_col_2}}} \cong t^{-a}q^{a(N-a-2)+2}\NB{\tikz[font=\tiny, scale=0.6]{\input{\imagesfolder/sch/R1_col_3}}}.
    \end{equation}
\end{prop}

\begin{proof}
    Let us construct the first isomorphism. We present the inductive step of the proof, the base case is given by Equation~\ref{uncol_reidem_1} of Proposition~\ref{uncol_reidem}.
    The first step consists of a full blist followed by a sequence of forkslide isomorphisms, and ends with a Reidemeister 1 move in the uncolored case on the innermost leg:
    \[ \NB{\tikz[font=\tiny, scale=0.6]{\input{\imagesfolder/sch/R1_blist_1}}} \cong \NB{\tikz[font=\tiny, scale=0.6]{\input{\imagesfolder/sch/R1_blist_3}}} \cong \NB{\tikz[font=\tiny, scale=0.6]{\input{\imagesfolder/sch/R1_blist_4}}}  \cong \NB{\tikz[font=\tiny, scale=0.6]{\input{\imagesfolder/sch/R1_blist_4p}}} \cong t^1q^{1-N}\NB{\tikz[font=\tiny, scale=0.6]{\input{\imagesfolder/sch/R1_blist_5}}}.\]
    The next step is to use the red dot slide isomorphisms in the uncolored case recalled in the Proposition \ref{uncol_sliding} and the Reidemeister 3 isomorphism \ref{reidem_3} to isolate the "unknotted leg":
    \[\NB{\tikz[font=\tiny, scale=0.6]{\input{\imagesfolder/sch/R1_blist_5}}} \cong \NB{\tikz[font=\tiny, scale=0.6]{\input{\imagesfolder/sch/R1_blist_2p}}} \cong \NB{\tikz[font=\tiny, scale=0.6]{\input{\imagesfolder/sch/R1_blist_2}}}.\]
    After that, we can resolve the innermost leg and move the red dot:
    \[ \NB{\tikz[font=\tiny, scale=0.6]{\input{\imagesfolder/sch/R1_blist_2}}} \cong t^aq^{a(2+a-N)-2} \NB{\tikz[font=\tiny, scale=0.6]{\input{\imagesfolder/sch/R1_blist_6}}} \cong t^aq^{a(2+a-N)-2} \NB{\tikz[font=\tiny, scale=0.6]{\input{\imagesfolder/sch/R1_blist_7}}}.\]
    Now we need to use the forktwists isomorphim \ref{forktwist} two times:
    \[\NB{\tikz[font=\tiny, scale=0.6]{\input{\imagesfolder/sch/R1_blist_7}}} \cong q^{2(a+1)} \NB{\tikz[font=\tiny, scale=0.5]{\input{\imagesfolder/sch/R1_blist_8}}}. \]
    To complete the proof, we need to use the red dot isomorphisms \ref{lemm_rd_3} and \ref{blist_lemm}: 
    \[\NB{\tikz[font=\tiny, scale=0.5]{\input{\imagesfolder/sch/R1_blist_8}}} \cong \NB{\tikz[font=\tiny, scale=0.5]{\input{\imagesfolder/sch/R1_blist_9}}} \cong \NB{\tikz[font=\tiny, scale=0.5]{\input{\imagesfolder/sch/R1_blist_10}}} \cong \NB{\tikz[font=\tiny, scale=0.5]{\input{\imagesfolder/sch/R1_blist_12}}}.\]
    After keeping track of the shifts, we obtain the following isomorphism: 
    \[\NB{\tikz[font=\tiny, scale=0.6]{\input{\imagesfolder/sch/R1_blist_1}}} \cong t^{a+1} q^{(a+1)(3+a-N)-2}\NB{\tikz[font=\tiny, scale=0.5]{\input{\imagesfolder/sch/R1_blist_12}}}.\]
\end{proof}

%% file: sch/forktwist_1.tex
\begin{scope}
	\draw[<-] (0.5,2) -- (0.5,1) node[pos=0,above] {$a+1$};
	\draw[-<] (1, 0.5) arc (0:180:0.5) -- (0,0.5) .. controls +(0,-0.5) and +(0, 0.5)
		.. (1, -0.5) -- (1,-1) node[pos=1,below]{$1$};
		\fill[white] (0.5, 0) circle (1mm);
	\draw[>-] (0, -1) -- (0, -0.5) node[pos=0,below]{$a$} .. controls +(0,0.5) and +(0, -0.5)
		.. (1, 0.5) ;
\end{scope}

%% file: sch/forktwist_2d.tex
\begin{scope}[]
	\draw[>-<] (1,-1) -- (1, 0.5) node[pos=0,below] {$1$} arc (0:180:0.5) -- (0,0.5) -- ( 0,-1) node[pos=1,below] {$a$};
	\draw[<-] (0.5,2) -- (0.5,1) node[pos=1, draw=\crd, fill= white, circle , inner sep=1.5pt]{} node[pos=1, yshift=5pt, right]{\textcolor{\crd}{$-1$}} node [pos=0,above] {$a+1$};
\end{scope}

%% file: sch/forktwist_1a.tex
\begin{scope}
	\draw[<-] (0.5,2) -- (0.5,1) node[pos=0,above] {$a+1$};
	\draw[-<] (1, 0.5) arc (0:180:0.5) -- (0,0.5) .. controls +(0,-0.5) and +(0, 0.5)
		.. (1, -0.5) -- (1,-1) node[pos=1,below]{$a$};
		\fill[white] (0.5, 0) circle (1mm);
	\draw[>-] (0, -1) -- (0, -0.5) node[pos=0,below]{$1$} .. controls +(0,0.5) and +(0, -0.5)
		.. (1, 0.5) ;
\end{scope}

%% file: sch/forktwist_2c.tex
\begin{scope}[]
	\draw[>-<] (1,-1) -- (1, 0.5) node[pos=0,below] {$1$} arc (0:180:0.5) -- (0,0.5) -- ( 0,-1) node[pos=1,below] {$a$};
	\draw[<-] (0.5,2) -- (0.5,1) node[pos=1, draw=\crd, fill= white, circle , inner sep=1.5pt]{} node[pos=1, yshift=5pt, right]{\textcolor{\crd}{$1$}} node [pos=0,above] {$a+1$};
\end{scope}

%% file: sch/forktwist_2b.tex
\begin{scope}[]
	\draw[>-<] (1,-1) -- (1, 0.5) node[pos=0,below] {$1$} arc (0:180:0.5) -- (0,0.5) -- ( 0,-1) node[pos=1,below] {$a$};
	\draw[<-] (0.5,2) -- (0.5,1) node [pos=0,above] {$a+1$};
\end{scope}

%% file: sch/forktwist_3.tex
\begin{scope}[font=\tiny]
	\coordinate (bl) at (0,  -1.5);
	\coordinate (br) at ( 1,  -1.5);
	\coordinate (bm) at (   0.5,-0.8);
	\coordinate (tl) at (0,   0.5);
	\coordinate (tr) at ( 1,   0.5);
	\coordinate (tm) at (   0.5, -0.2);
	\draw[>-]  (bl) .. controls +( 0, 0.5) and +(0,0) .. (bm) node[pos=0,below] {$a$};
	\draw[>-]  (br) .. controls +( 0, 0.5) and +(0,0) .. (bm)node[pos=0,below] {$1$};
	\draw[-<-]  (0.5,1) .. controls +( -0.5, 0) and +(0,0) ..  (tl) node[pos=1,left] {$1$} .. controls +( 0, -0.5) and +(0,0) .. (tm);
	\draw[-<-]  (0.5,1) .. controls +(0.5,0) and +(0,0) .. (tr) node[pos=0.3, draw=\crd, fill= white, circle , inner sep=1.5pt]{} node[pos=0.3, right]{\color{\crd}$\frac{\Bar{s}}{2}$} node[pos=1,right] {$a$} .. controls +( 0, -0.5) and +(0,0) .. (tm)
		node[pos= 0.35 , draw=\crd, fill=white,regular polygon, regular polygon sides=3,inner sep=0.9pt] {} node[pos=0.3, right]{\color{\crd}$-\lambda(a-1)-\frac{\Bar{s}}{2}$};
	\draw[] (bm) -- (tm)
		node[pos=0.25, draw=\crd, fill= white, circle , inner sep=1.5pt]{} node[pos=0.25, left]{\color{\crd}$-\frac{\Bar{s}}{2}$} 
        node[pos= 0.7 , draw=\crd, fill=white,regular polygon, regular polygon sides=3,inner sep=0.9pt] {} node[pos=0.7, right]{\color{\crd}$\lambda a+\frac{\Bar{s}}{2}$};
	\draw[<-] (0.5,2) -- (0.5,1) 
        node[pos=0,above] {$a+1$}
		node[pos=1, draw=\crd, fill= white, circle , inner sep=1.5pt]{} node[pos=1,yshift=5pt, right]{\color{\crd}$1$};
\end{scope}

%% file: sch/forktwist_1_p.tex
\begin{scope}
	\draw[<-] (0.5,2) -- (0.5,1) node[pos=0,above] {$a+b$};
	\draw[-<] (1, 0.5) arc (0:180:0.5) -- (0,0.5) .. controls +(0,-0.5) and +(0, 0.5)
		.. (1, -0.5) -- (1,-1) node[pos=1,below]{$b$};
		\fill[white] (0.5, 0) circle (1mm);
	\draw[>-] (0, -1) -- (0, -0.5) node[pos=0,below]{$a$} .. controls +(0,0.5) and +(0, -0.5)
		.. (1, 0.5) ;
\end{scope}

%% file: sch/forktwist_2d_p.tex
\begin{scope}[]
	\draw[>-<] (1,-1) -- (1, 0.5) node[pos=0,below] {$b$} arc (0:180:0.5) -- (0,0.5) -- ( 0,-1) node[pos=1,below] {$a$};
	\draw[<-] (0.5,2) -- (0.5,1) node[pos=1, draw=\crd, fill= white, circle , inner sep=1.5pt]{} node[pos=1, yshift=5pt, right]{\textcolor{\crd}{$-1$}} node [pos=0,above] {$a+b$};
\end{scope}

%% file: sch/forktwist_1a_p.tex
\begin{scope}
	\draw[<-] (0.5,2) -- (0.5,1) node[pos=0,above] {$a+b$};
	\draw[-<] (1, 0.5) arc (0:180:0.5) -- (0,0.5) .. controls +(0,-0.5) and +(0, 0.5)
		.. (1, -0.5) -- (1,-1) node[pos=1,below]{$a$};
		\fill[white] (0.5, 0) circle (1mm);
	\draw[>-] (0, -1) -- (0, -0.5) node[pos=0,below]{$b$} .. controls +(0,0.5) and +(0, -0.5)
		.. (1, 0.5) ;
\end{scope}

%% file: sch/forktwist_2c_p.tex
\begin{scope}[]
	\draw[>-<] (1,-1) -- (1, 0.5) node[pos=0,below] {$b$} arc (0:180:0.5) -- (0,0.5) -- ( 0,-1) node[pos=1,below] {$a$};
	\draw[<-] (0.5,2) -- (0.5,1) node[pos=1, draw=\crd, fill= white, circle , inner sep=1.5pt]{} node[pos=1, yshift=5pt, right]{\textcolor{\crd}{$1$}} node [pos=0,above] {$a+b$};
\end{scope}

%% file: sch/R1_col_1.tex
\begin{scope}[xshift = 2.5cm]
  \draw  (0, 1) -- (0, 0.5) .. controls +(0,-0.5) and +(0, 0.5)
  .. (1, -0.5) arc (180:360:0.5) -- (2,0);  
    \fill[white] (0.5, 0) circle (1mm);
  \draw[<-] (0, -1) -- (0, -0.5) .. controls +(0,0.5) and +(0, -0.5)
  .. (1, 0.5) arc (180:0:0.5) -- (2,0);
  \node at (0,1.3) {$a$};
\end{scope}

%% file: sch/R1_col_4.tex
\begin{scope}
  \draw [->] (0, -1) -- +(0,2) node[pos=0,below] {$a$} node[pos=0.8, draw=\crd, fill= white, circle , inner sep=1.5pt]{} node[pos=0.8, right]{\color{\crd}$\frac{1}{2}$}
  node[pos=0.5, draw= \crd, fill = \csrd, circle , inner sep=1.5pt]{} node[pos=0.5, right]{\color{\crd}$-\frac{1}{2}$}
  node[pos= 0.2 , draw=\crd, fill=white,regular polygon, regular polygon sides=3,inner sep=0.9pt] {} node[pos=0.2, right]{\color{\crd}$-\frac{1}{2}$};
\end{scope}

%% file: sch/R1_col_2.tex
\begin{scope}
  \draw [<-] (0, 1) -- (0, 0.5) .. controls +(0,-0.5) and +(0, 0.5) .. (1, -0.5) arc (180:360:0.5) -- (2,0) ;  
    \fill[white] (0.5, 0) circle (1mm);
  \draw (0, -1) -- (0, -0.5) .. controls +(0,0.5) and +(0, -0.5) .. (1, 0.5) arc (180:0:0.5) -- (2,0);
  \node at (0,-1.3) {$a$};
\end{scope}

%% file: sch/R1_col_3.tex
\begin{scope}
  \draw [->] (0, -1) -- +(0,2) node[pos=0,below] {$a$} node[pos=0.8, draw=\crd, fill= white, circle , inner sep=1.5pt]{} node[pos=0.8, right]{\color{\crd}$-\frac{1}{2}$}
  node[pos=0.5, draw= \crd, fill = \csrd, circle , inner sep=1.5pt]{} node[pos=0.5, right]{\color{\crd}$\frac{1}{2}$}
  node[pos= 0.2 , draw=\crd, fill=white,regular polygon, regular polygon sides=3,inner sep=0.9pt] {} node[pos=0.2, right]{\color{\crd}$\frac{1}{2}$};
\end{scope}

%% file: Z_sliding_g.tex
In \cite{GR25}, it is shown that all types of dots can slide along any crossing involving only $1$-colored strands.

\begin{prop}[{\cite[Cor~5.8, Lem~5.10]{GR25}}]
\label{uncol_sliding}
    Let $\omega$ and $\epsilon$ be elements of $\scalars$, there are isomorphisms in the relative homotopy canopoly:
    \begin{align}
        \label{line1}
        \NB{\tikz[font=\tiny,scale=1.2]{\input{\imagesfolder/sch/reidem_lemm_1}}} \cong \NB{\tikz[font=\tiny,scale=1.2]{\input{\imagesfolder/sch/reidem_lemm_2}}} \qquad \text{and} \qquad \NB{\tikz[font=\tiny,scale=1.2]{\input{\imagesfolder/sch/reidem_lemm_3}}} \cong \NB{\tikz[font=\tiny,scale=1.2]{\input{\imagesfolder/sch/reidem_lemm_4}}}, 
    \end{align}    
    \begin{align} 
        \NB{\tikz[font=\tiny,scale=1.2]{\input{\imagesfolder/sch/reidem_lemm_5}}} \cong \NB{\tikz[font=\tiny,scale=1.2]{\input{\imagesfolder/sch/reidem_lemm_6}}} \qquad \text{and} \qquad \NB{\tikz[font=\tiny,scale=1.2]{\input{\imagesfolder/sch/reidem_lemm_7}}} \cong \NB{\tikz[font=\tiny,scale=1.2]{\input{\imagesfolder/sch/reidem_lemm_8}}}.
    \end{align}
    % \begin{align}
    %     \label{line3}
    %     \ttkz[font=\tiny,scale=1.2]{sch/reidem_lemm_17} \cong \ttkz[font=\tiny,scale=1.2]{sch/reidem_lemm_18} \qquad \text{and} \qquad \ttkz[font=\tiny,scale=1.2]{sch/reidem_lemm_31} \cong \ttkz[font=\tiny,scale=1.2]{sch/reidem_lemm_32}.
    % \end{align}
\end{prop}

\begin{rmk}
    In line~(\ref{line1}), we can replace the circular hollow dot with the triangular one, since for a \(1\)-colored strand the circular hollow dot and the triangular dot coincide.
\end{rmk}

Let us now turn to the colored case. In the following, every proposition and lemma is stated for negative crossings, but the statements also hold for positive ones.

\begin{lem}
    \label{col_slid_2}
    Let $a \in \N^*$ and $\omega \in \scalars$, the following objects are isomorphic in the relative homotopy canopoly:
    \begin{align}
        \NB{\tikz[font=\tiny,xscale=-1,scale=1.2]{\input{\imagesfolder/sch/slid_lemm_19}}} \cong \NB{\tikz[font=\tiny,xscale=-1,scale=1.2]{\input{\imagesfolder/sch/slid_lemm_20}}} \qquad \text{and} \qquad \NB{\tikz[font=\tiny,xscale=-1,scale=1.2]{\input{\imagesfolder/sch/slid_lemm_21}}} \cong \NB{\tikz[font=\tiny,xscale=-1,scale=1.2]{\input{\imagesfolder/sch/slid_lemm_22}}}
        % \ttkz[font=\tiny,xscale=-1,scale=1.2]{sch/slid_lemm_13} \cong \ttkz[font=\tiny,xscale=-1,scale=1.2]{sch/slid_lemm_14} \qquad \text{and} \qquad \ttkz[font=\tiny,xscale=-1,scale=1.2]{sch/slid_lemm_15} \cong \ttkz[font=\tiny,xscale=-1,scale=1.2]{sch/slid_lemm_16}.
    \end{align}
\end{lem}

\begin{proof}
    The idea is to show that the dot can be slid along when the \(a\)-colored leg is full-blisted by applying the forkslide isomorphisms~(\ref{forkslide}) and the sliding lemma~(\ref{uncol_sliding}) in the uncolored case:
    \[
        \NB{\tikz[font=\tiny,xscale=-1,scale=1.2]{\input{\imagesfolder/sch/slid_pf1}}} \cong \NB{\tikz[font=\tiny,xscale=-1,scale=1.2]{\input{\imagesfolder/sch/slid_pf2}}} \cong \NB{\tikz[font=\tiny,xscale=-1,scale=1.2]{\input{\imagesfolder/sch/slid_pf3}}} \cong \NB{\tikz[font=\tiny,xscale=-1,scale=1.2]{\input{\imagesfolder/sch/slid_pf4}}}.
    \]
    Using this isomorphism, together with the blist hypothesis, we obtain the desired isomorphism.
\end{proof}

The proof of the following proposition is the purpose of Sections \ref{ext_can}, \ref{ext_f_w} and \ref{slid_or}. We assume that this proposition holds for the remainder of this section.

\begin{prop}
    \label{slid_rd}
    Let $a,b \in \N^*$ and $\omega, \varepsilon \in \scalars$, the following objects are isomorphic in the relative homotopy canopoly:
    \begin{equation}
         \NB{\tikz[font=\tiny,xscale=-1,scale=1.2]{\input{\imagesfolder/sch/slid_lemm_21_b}}} \cong \NB{\tikz[font=\tiny,xscale=-1,scale=1.2]{\input{\imagesfolder/sch/slid_lemm_22_b}}}. 
    \end{equation}
\end{prop}

\begin{lem}
    \label{blist_lemm}
    Let $a \in \N^*$ and $\omega \in \scalars$, the following objects are isomorphic in the relative homotopy canopoly: 
    \begin{align}
        \NB{\tikz[scale=1.2, font=\tiny]{\input{\imagesfolder/sch/slid_lemm_1}}} \cong \NB{\tikz[scale=1.2, font=\tiny]{\input{\imagesfolder/sch/slid_lemm_2}}} \qquad \text{and} \qquad \NB{\tikz[scale=1.2,font=\tiny]{\input{\imagesfolder/sch/slid_lemm_3}}} \cong \NB{\tikz[font=\tiny,scale=1.2]{\input{\imagesfolder/sch/slid_lemm_4}}}.
    \end{align}
\end{lem}

\begin{proof}
    We present the proof of the existence of the first isomorphism, the second one is similar. We assume that $\omega = 1$, without loss of generality. The proof of this lemma reduces to studying the decorations involved in the twists induced by the red dots. Recall that $p_k^{(a)}$ denotes the $k$th power sum symmetric polynomial in $Sym_a$. For all $1 \leq i \leq a$, we denote by $\epsilon_i$ the edges of thickness $1$. In particular, $p_k^{(\epsilon_i)} \in Sym_1$ is the $k$th power of the variable associated with $\epsilon_i$. For all $n \in \N$, 
    \begin{equation*}
        -\sum_{k=0}^n p_k^{(a)}p_{n-k}^{(a)} - \sum_{k=0}^n p_k^{(a)}\hat{p}_{n-k}^{(a)} = - \sum_{k=0}^n p^{(a)}_kP_{n-k} = - \sum_{i=1}^a \left(  \sum_{k=0}^np^{(\epsilon_i)}_k(P_{n-k}-p_{n-k}^{(\epsilon_i)})  + (n+1)p_n^{(\epsilon_i)}\right ).
    \end{equation*}
    We conclude by identifying each polynomial with its associated dot.
\end{proof}

The following lemma is a consequence of Lemma \ref{blist_lemm} and the blist hypothesis.

\begin{lem}
    \label{col_slid_1}
    Let $a,b \in \N^*$ and $\omega, \epsilon \in \scalars$, the following objects are isomorphic in the relative homotopy canopoly:
    \begin{align}
        \NB{\tikz[scale=1.2,font=\tiny]{\input{\imagesfolder/sch/slid_lemm_7}}} \cong \NB{\tikz[font=\tiny,scale=1.2]{\input{\imagesfolder/sch/slid_lemm_8}}} \qquad \text{and} \qquad \NB{\tikz[scale=1.2,font=\tiny]{\input{\imagesfolder/sch/slid_lemm_11}}} \cong \NB{\tikz[font=\tiny,scale=1.2]{\input{\imagesfolder/sch/slid_lemm_12}}}.
    \end{align}
    % \begin{align}
    %     \ttkz[scale=1.2, font=\tiny]{sch/slid_lemm_9} \cong \ttkz[scale=1.2, font=\tiny]{sch/slid_lemm_10} \qquad \text{and} \qquad \ttkz[scale=1.2, font=\tiny]{sch/slid_lemm_5} \cong \ttkz[scale=1.2, font=\tiny]{sch/slid_lemm_6}
    % \end{align}
\end{lem} 

The proof of the following proposition is a direct consequence of Proposition \ref{slid_rd} and the previous Lemma \ref{col_slid_1}.
\begin{prop}
    \label{col_slid_3}
    Let $a,b \in \N^*$ and $\omega, \epsilon \in \scalars$, the following objects are isomorphic in the relative homotopy canopoly:
    \begin{equation}
        \NB{\tikz[scale=1.2, font=\tiny]{\input{\imagesfolder/sch/slid_lemm_5}}} \cong \NB{\tikz[scale=1.2, font=\tiny]{\input{\imagesfolder/sch/slid_lemm_6}}}. 
    \end{equation}
\end{prop}

%% file: Z_ext_can.tex
In the proof of the sliding lemma \cite[Lemma~5.3]{GR25} in the uncolored case, 
an additional variable $z$, besides the variables of $\scalars_N$, is required to let a dot slide through a crossing. The purpose of this section is to construct an analogue of this additional variable by extending the canopoly $\catf$.

\begin{dfn}
    An \textit{extended web} in $\Disk$ is a web in $\Disk$ with dots in the interior of $\Disk$ that are disjoint from the web. This data is considered up to ambient isotopy that preserve the boundary of $\Disk$. The boundary of an extended web is the boundary of its web in $\Disk$.
\end{dfn}

\begin{figure}[h!]
    \centering
    \NB{\tikz[scale=.8]{\input{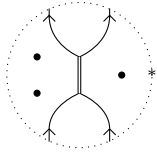}}}
    \caption{Example of an extended web.}
\end{figure}

\begin{dfn}
    A \textit{basic trajectory} in $\Disk \times \interval$ is one of the oriented $2$-dimensional CW-complexes shown in Figure~\ref{trajectory}, considered up to ambient isotopy that preserves the boundary $\partial \Disk \times \interval$.
    An \textit{extended foam} in $\Disk \times [0,1]$, is a finite collection of basic foams, traces of isotopy and basic trajectories that are glued together along their boundary points with respect to their labels. This data is considered up to ambient isotopy that preserves the boundary of $\Disk \times \interval$.
\end{dfn}

\begin{rmk}
    As for the foams, an extended foam has a degree which is additive under gluing. 
\end{rmk}

\begin{figure}[h!]
    \centering
    \begin{tikzpicture}[xscale=2.4, yscale =-3.2]
     \node (a) at (0,1) {\NB{\tikz[scale=1]{\input{\imagesfolder/sch/meteorite_3}}}};
     \node[yshift = -1.64cm] at (a) {Polynomial};
     \node[yshift = -2.1cm] at (a) {$\deg(P)$};
     \node (b) at (2,1) {\NB{\tikz[scale=1]{\input{\imagesfolder/sch/meteorite_2}}}};
     \node[yshift = -1.64cm] at (b) {Crash};
     \node[yshift = -2.1cm] at (b) {$0$};
     \node (c) at (4,1) {\NB{\tikz[scale=1]{\input{\imagesfolder/sch/meteorite_4}}}};
     \node[yshift = -1.64cm] at (c) {Cross};
     \node[yshift = -2.1cm] at (c) {$0$};
   \end{tikzpicture}
    \caption{The three basic trajectories with their degrees, where $P$ is a polynomial in $\scalars_N[z]$, where $z$ is a new variable distinct from all the variables of $\scalars_N$, and $a \in \N^*$ is the thickness of the considered facet.}
    \label{trajectory}
\end{figure}
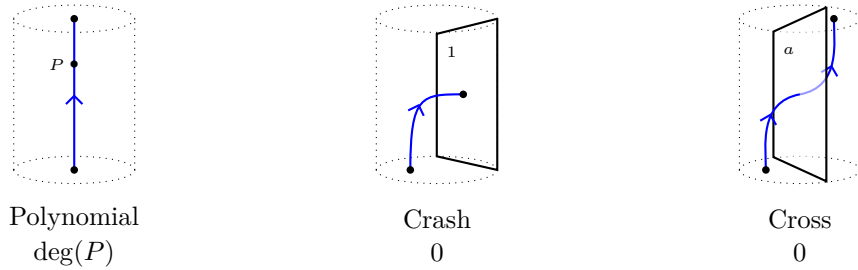

To understand the relation between these extended webs and foams, as well as the additional variables needed in the proof of the sliding lemma, we need to describe how the functor \(\ssp_N\) extends to these new objects and morphisms. This functor uses the Robert--Wagner evaluation and the universal construction to associate a \(\scalars_N\)-module to each closed web and a \(\scalars_N\)-linear map to each foam as in \cite{RW20a}.

Let $\widetilde{\web}$ be a closed extended web with $n$ dots, then 
\begin{equation}
    \ssp_N(\widetilde{\web}) := \ssp_N(\web) \otimes_{\scalars_N} \scalars_N[z_1,\ldots,z_n],
\end{equation}
where $\web$ is the closed web without the dots. For the morphisms, we only describe how \(\ssp_N\) behaves locally on the extra algebra $\scalars_N[z]$.
\begin{align}
    \ssp_N\left (\NB{\tikz[scale=.7]{\input{\imagesfolder/sch/meteorite_3}}} \right ) &:= \scalars_N[z] \xrightarrow{\times P(z)} \scalars_N[z], \\
    \label{eq:crash}
    \ssp_N\left (\NB{\tikz[scale=.7]{\input{\imagesfolder/sch/meteorite_2}}} \right ) &:= \ssp_N(\web) \otimes \scalars_N[z] \xrightarrow{``z \mapsto x''} \ssp_N(\web), \\
    \ssp_N\left (\NB{\tikz[scale=.7]{\input{\imagesfolder/sch/meteorite_4}}} \right ) &:= \ssp_N(\web) \otimes \scalars_N[z] \xrightarrow{\Id} \ssp_N(\web) \otimes \scalars_N[z],
\end{align}
where $x$ denotes the variable used for the decorations on the $1$-labeled facet where the trajectory crashed.The map~\eqref{eq:crash} sends all decorations along the trajectory to the $1$-labeled facet. This property is illustrated by the second equivalence in~\eqref{eq:ex_rel}.

This extension of the definition of $\ssp_N$ leads to an equivalence relation on extended foams. Let $\foam$ and $\foam'$ be two extended foams, we say that these two foams are equivalent, denoted $\foam \approx \foam'$, if they bound the same extended webs and their images under $\ssp_N$ coincide. For example this extended foams are equivalent: 

\begin{equation}
    \label{eq:ex_rel}
    \NB{\tikz[scale=1]{\input{\imagesfolder/sch/meteorite_5}}} \quad \approx \quad \NB{\tikz[scale=1]{\input{\imagesfolder/sch/meteorite_6}}} \hspace{2cm} \NB{\tikz[scale=1]{\input{\imagesfolder/sch/meteorite_7}}} \quad \approx \quad \NB{\tikz[scale=1]{\input{\imagesfolder/sch/meteorite_8}}}.  
\end{equation}

\begin{dfn}
    For a given $\bnd \in \Z_{\neq 0}^{n_\bnd}$, the category $\widehat{\extcatf}(\bnd)$ has objects that are extended webs in $\Disk$ with boundary describe by $\bnd$ and morphisms that are foams defined as follows: 
    \[
    \text{Hom}_{\extcatf}(\web_0, \web_1) =  \{\text{extended foams}~\foam \text{ in } \Disk \times \interval \text{ with } \foam_0 = -\web_0 \text{ and } \foam_1 = \web_1 \} / \approx ,
    \]
    where $-\web_0$ is the same extended web as $\web_0$ but with all edges oriented in the opposite direction. 
\end{dfn}

\begin{rmk}
    As $\widehat{\catf}(\bnd)$, we extend the category $\widehat{\extcatf}(\bnd)$ by adding formal grading shifts, direct sums, and $\scalars_N$-linear combinations of morphisms to form the graded, additive, $\scalars_N$-linear category $\extcatf(\bnd)$.
\end{rmk}

\begin{rmk}
    For a given extended web $\widetilde{\web}$ with $n$ dots, the algebra of decorations is \[D_{\tilde{\web}} := D_{\web} \otimes \scalars_N[z_1,\ldots,z_n],\] where $\web$ denotes the web obtained from $\widetilde{\web}$ by removing all dots.
    % In other words by using the same notations as Section \ref{sec:MCfoam}, if $e$ is a dot of $\widetilde{\web}$, then $D_e = D_{e}' := \scalars_N[z_i]$.
\end{rmk}

The proof of Proposition~\ref{prop:ext_canop} is straightforward, following the ideas in Example~\ref{ex:foam_can}.

\begin{prop}
    \label{prop:ext_canop}    
    The collection $(\extcatf(\bnd))_\bnd$ is a colored canopoly denoted $\extcatf$. 
\end{prop}

%% file: sch/meteorite_3.tex
\begin{scope}[font=\tiny]
\draw[dotted] 
  (0,2) arc[start angle=180,end angle=360,x radius=.8,y radius=0.18];
\draw[dotted] 
  (0,0) arc[start angle=180,end angle=360,x radius=.8,y radius=0.18];
\draw[dotted] 
  (0,2) arc[start angle=180,end angle=0,x radius=.8,y radius=0.18];
\draw[dotted] 
  (0,0) arc[start angle=180,end angle=0,x radius=.8,y radius=0.18];
\draw[dotted] (0,0)--(0,2);
\draw[dotted] (1.6,0)--(1.6,2);
\draw[->-,thick,blue] (0.8,0) -- (.8,2) node[pos=.7]{\color{black}$\bullet$}node[pos=.7,left]{\color{black}$P$};
\fill (.8,0) circle (0.05);
\fill (.8,2) circle (0.05);
\end{scope}

%% file: sch/meteorite_2.tex
\begin{scope}[xscale=-1,font=\tiny]
\draw[dotted] 
  (0,2) arc[start angle=180,end angle=360,x radius=.8,y radius=0.18];
\draw[dotted] 
  (0,0) arc[start angle=180,end angle=360,x radius=.8,y radius=0.18];
\draw[dotted] 
  (0,2) arc[start angle=180,end angle=0,x radius=.8,y radius=0.18];
\draw[dotted] 
  (0,0) arc[start angle=180,end angle=0,x radius=.8,y radius=0.18];
\draw[dotted] (0,0)--(0,2);
\draw[dotted] (1.6,0)--(1.6,2);
\draw[thick] (0,2.01) -- (0,0) -- (.8,0.18) -- (.8,1.8) node[pos=.85,right]{$1$} -- (0,2);
\draw[-<-,thick,blue] (0.45,1).. controls +(0.55,0) and +(0,1.0).. (1.15,0);
\fill (1.15,0) circle (0.05);
\fill (0.45,1) circle (0.05);
\end{scope}

%% file: sch/meteorite_4.tex
\begin{scope}[xscale=-1,font=\tiny]
\draw[dotted] 
  (0,2) arc[start angle=180,end angle=360,x radius=.8,y radius=0.18];
\draw[dotted] 
  (0,0) arc[start angle=180,end angle=360,x radius=.8,y radius=0.18];
\draw[dotted] 
  (0,2) arc[start angle=180,end angle=0,x radius=.8,y radius=0.18];
\draw[dotted] 
  (0,0) arc[start angle=180,end angle=0,x radius=.8,y radius=0.18];
\draw[dotted] (0,0)--(0,2);
\draw[dotted] (1.6,0)--(1.6,2);
\draw[->-,thick,blue] (0.8,1).. controls +(-0.55,.1) and +(0,-.3).. (0.35,2);
\fill[white,opacity=.6] (.45,0.17) rectangle (1.15,1.83);
\draw[thick] (.45,2.16) -- (.45,-.15) -- (1.15,0.17) -- (1.15,1.83) node[pos=.85,right]{$a$} -- (.45,2.15);
\draw[-<-,thick,blue] (0.8,1).. controls +(0.55,-.1) and +(0,.3).. (1.25,0);
\fill (1.25,0) circle (0.05);
\fill (0.35,2) circle (0.05);
\end{scope}

%% file: sch/meteorite_6.tex
\begin{scope}[xscale=-1,font=\tiny]
\draw[dotted] 
  (0,2) arc[start angle=180,end angle=360,x radius=.8,y radius=0.18];
\draw[dotted] 
  (0,0) arc[start angle=180,end angle=360,x radius=.8,y radius=0.18];
\draw[dotted] 
  (0,2) arc[start angle=180,end angle=0,x radius=.8,y radius=0.18];
\draw[dotted] 
  (0,0) arc[start angle=180,end angle=0,x radius=.8,y radius=0.18];
\draw[dotted] (0,0)--(0,2);
\draw[dotted] (1.6,0)--(1.6,2);
\fill[white,opacity=.6] (.45,0.17) rectangle (1.15,1.83);
\draw[thick] (.45,2.16) -- (.45,-.15) -- (1.15,0.17) -- (1.15,1.83) node[pos=.85,right]{$a$} -- (.45,2.15);
\draw[-<-,thick,blue] (1.25,2) -- (1.25,0);
\fill (1.25,0) circle (0.05);
\fill (1.25,2) circle (0.05);
\end{scope}

%% file: sch/meteorite_7.tex
\begin{scope}[xscale=-1,font=\tiny]
\draw[dotted] 
  (0,2) arc[start angle=180,end angle=360,x radius=.8,y radius=0.18];
\draw[dotted] 
  (0,0) arc[start angle=180,end angle=360,x radius=.8,y radius=0.18];
\draw[dotted] 
  (0,2) arc[start angle=180,end angle=0,x radius=.8,y radius=0.18];
\draw[dotted] 
  (0,0) arc[start angle=180,end angle=0,x radius=.8,y radius=0.18];
\draw[dotted] (0,0)--(0,2);
\draw[dotted] (1.6,0)--(1.6,2);
\draw[thick] (0,2.01) -- (0,0) -- (.8,0.18) -- (.8,1.8) node[pos=.85,right]{$1$} -- (0,2);
\draw[-<-,thick,blue] (0.45,1).. controls +(0.55,0) and +(0,1.0).. (1.15,0);
\fill (1.15,0) circle (0.05);
\fill (0.45,1) circle (0.05);
\fill (1.13,.5) circle (.05) node[left]{$P$}; 
\end{scope}

%% file: sch/meteorite_8.tex
\begin{scope}[xscale=-1,font=\tiny]
\draw[dotted] 
  (0,2) arc[start angle=180,end angle=360,x radius=.8,y radius=0.18];
\draw[dotted] 
  (0,0) arc[start angle=180,end angle=360,x radius=.8,y radius=0.18];
\draw[dotted] 
  (0,2) arc[start angle=180,end angle=0,x radius=.8,y radius=0.18];
\draw[dotted] 
  (0,0) arc[start angle=180,end angle=0,x radius=.8,y radius=0.18];
\draw[dotted] (0,0)--(0,2);
\draw[dotted] (1.6,0)--(1.6,2);
\draw[thick] (0,2.01) -- (0,0) -- (.8,0.18) -- (.8,1.8) node[pos=.85,right]{} -- (0,2);
\draw[-<-,thick,blue] (0.45,1).. controls +(0.55,0) and +(0,1.0).. (1.15,0);
\fill (1.15,0) circle (0.05);
\fill (0.45,1) circle (0.05);
\fill (.45,1.4) circle (.05) node[right]{$P$}; 
\end{scope}

%% file: Z_witt_ext.tex
We can now extend the operators defined on basic foams in section \ref{sec_action} denoted $\Wgen_n$: 

\begin{align}
    \Wgen_n\left (  \NB{\tikz[scale=1]{\input{\imagesfolder/sch/meteorite_3}}} \right ) &= \NB{\tikz[scale=1]{\input{\imagesfolder/sch/meteorite_9}}}, \\
    \Wgen_n\left (  \NB{\tikz[scale=1]{\input{\imagesfolder/sch/meteorite_2}}} \right ) = 0,  \hspace{1cm}&\hspace{1cm} \Wgen_n\left (  \NB{\tikz[scale=1]{\input{\imagesfolder/sch/meteorite_4}}} \right ) = 0. \\
\end{align}

This proposition is a direct corollary of Proposition $\ref{action}$. 
\begin{prop}
    Let $\web, \web'$ be two extended webs with the same boundary. Mapping $\wgen_n$ to $\Wgen_n$ for all $n \in \N_{-1}$ defines a $\Witt$-action on the $\scalars_N$-module of $\Hom_{\extcatf}(\web,\web')$. Moreover, for any extended foam $\foam$:
    \[\deg_N(\wgen_n \cdot \foam) = 2n + \deg_N(\foam).\]
\end{prop}

Let $\extalgtop(\bnd)$ be the subcategory of $\extcatf(\bnd)$ whose objects are the same and whose morphisms are $\scalars_N$-linear combinations of extended foams $\foam$ such that for all $n \in \N_{-1}$:
    \[\wgen_n \cdot \foam = 0.\]
Such an extended foam is called \textit{$\Witt$-equivariant}.

As for webs, we can define red-dotted extended webs. Here is a diagrammatic example of a Maurer--Cartan element:
\begin{equation}
    \wgen_n * \, \NB{\tikz[font=\small, scale=1]{\input{\imagesfolder/sch/meteorite_12}}} \, =
    \wgen_n \cdot \NB{\tikz[font=\small, scale=1]{\input{\imagesfolder/sch/meteorite_11}}}
    -\omega (n+1) \, \NB{\tikz[font=\small, scale=1]{\input{\imagesfolder/sch/meteorite_10}}}, 
\end{equation}
associated to the Maurer--Cartan element: 
\begin{equation}
    \wgen_n \mapsto -\omega (n+1)z^n.
\end{equation}
The proof that this is indeed a Maurer--Cartan element is left to the reader.

Let $\extralgtop(\bnd)$ be the category whose objects are pairs \((\web, \alpha)\) with \(\web\) in \( \extcatf(\bnd)\) and \(\alpha\) an associated Maurer--Cartan element, and whose morphisms are $\scalars_N$-linear combinations of extended foams $\foam$ such that for all $n \in \N_{-1}$:
    \[\wgen_n * \foam = 0.\]

As in Section~\ref{RHC}, we consider the relative homotopy category associated with $\extralgtop(\bnd)$, denoted by $\extrelcat{\bnd}$.

\begin{prop}
    The collection $(\extrelcat{\bnd})_\bnd$ is a colored canopoly denoted $\extrelcatt$. 
\end{prop}

%% file: sch/meteorite_9.tex
\begin{scope}[font=\tiny]
\draw[dotted] 
  (0,2) arc[start angle=180,end angle=360,x radius=.8,y radius=0.18];
\draw[dotted] 
  (0,0) arc[start angle=180,end angle=360,x radius=.8,y radius=0.18];
\draw[dotted] 
  (0,2) arc[start angle=180,end angle=0,x radius=.8,y radius=0.18];
\draw[dotted] 
  (0,0) arc[start angle=180,end angle=0,x radius=.8,y radius=0.18];
\draw[dotted] (0,0)--(0,2);
\draw[dotted] (1.6,0)--(1.6,2);
\draw[->-,thick,blue] (0.85,0) -- (.85,2) node[pos=.7]{\color{black}$\bullet$}node[pos=.7,left,xshift=.4mm]{\color{black}$\wgen_n\cdot P$};
\fill (.85,0) circle (0.05);
\fill (.85,2) circle (0.05);
\end{scope}

%% file: sch/meteorite_12.tex
\begin{scope}
\draw[dotted] 
  (0,2) arc[start angle=180,end angle=360,x radius=.8,y radius=0.18];
\draw[dotted] 
  (0,0) arc[start angle=180,end angle=360,x radius=.8,y radius=0.18];
\draw[dotted] 
  (0,2) arc[start angle=180,end angle=0,x radius=.8,y radius=0.18];
\draw[dotted] 
  (0,0) arc[start angle=180,end angle=0,x radius=.8,y radius=0.18];
\draw[dotted] (0,0)--(0,2);
\draw[dotted] (1.6,0)--(1.6,2);
\draw[->-,thick,blue] (0.8,0) -- (.8,2);
\fill (.8,0) circle (0.05);
\filldraw[draw= \crd, fill = white] (.8,2) circle (1mm) node[left,xshift=-.5mm] {$\color{\crd}\omega$};
\fill (.8,2) circle (0.05);
\end{scope}

%% file: sch/meteorite_11.tex
\begin{scope}[font=\tiny]
\draw[dotted] 
  (0,2) arc[start angle=180,end angle=360,x radius=.8,y radius=0.18];
\draw[dotted] 
  (0,0) arc[start angle=180,end angle=360,x radius=.8,y radius=0.18];
\draw[dotted] 
  (0,2) arc[start angle=180,end angle=0,x radius=.8,y radius=0.18];
\draw[dotted] 
  (0,0) arc[start angle=180,end angle=0,x radius=.8,y radius=0.18];
\draw[dotted] (0,0)--(0,2);
\draw[dotted] (1.6,0)--(1.6,2);
\draw[->-,thick,blue] (0.8,0) -- (.8,2);
\fill (.8,0) circle (0.05);
\fill (.8,2) circle (0.05);
\end{scope}

%% file: sch/meteorite_10.tex
\begin{scope}
\draw[dotted] 
  (0,2) arc[start angle=180,end angle=360,x radius=.8,y radius=0.18];
\draw[dotted] 
  (0,0) arc[start angle=180,end angle=360,x radius=.8,y radius=0.18];
\draw[dotted] 
  (0,2) arc[start angle=180,end angle=0,x radius=.8,y radius=0.18];
\draw[dotted] 
  (0,0) arc[start angle=180,end angle=0,x radius=.8,y radius=0.18];
\draw[dotted] (0,0)--(0,2);
\draw[dotted] (1.6,0)--(1.6,2);
\draw[->-,thick,blue] (0.8,0) -- (.8,2) node[pos=.7]{\color{black}$\bullet$}node[pos=.7,left]{\color{black}$z^{n}$};
\fill (.8,0) circle (0.05);
\fill (.8,2) circle (0.05);
\end{scope}

%% file: Z_sliding_orange.tex
To prove that the dots are sliding through any crossings we need to introduce a new Maurer--Cartan element.

\begin{rmk}
    If $e$ is an $a$-labeled edge or a dot of an extended web, then $p_k^{(e)}$ denotes the extended foam $\widetilde{\web} \times \interval$ decorated by $p_k$ on $e \times \interval$. Moreover, $p_k \in Sym_a$ if $e$ is an $a$-labeled edge, while $p_k = z_i^k \in \scalars[z_i]$ if $e$ is a dot.
\end{rmk}

\begin{dfn}
    Let $e_1$ and $e_2$ be edges or dots of an extended web $\widetilde{\web}$.
    A double orange dot $\dod$ between $e$ and $\widetilde{e}$, labeled by $\omega \in \scalars$, is associated with the map
    \begin{equation} 
    \label{orange_tw}
    \wgen_n \mapsto - \omega \sum_{i=0}^n p_i^{(e_1)} p_{n-i}^{(e_2)}.
    \end{equation}
\end{dfn}

The proof of the following lemma proceeds in the same way as the proof of Lemma~4.5 in \cite{GR25}.

\begin{lem}
    The map \eqref{orange_tw} defines a Maurer--Cartan element. 
\end{lem}

% \begin{exa}
%     We provide an example of an orange dotted web:
%     \[\ttkz[font=\tiny,scale=1]{sch/ex_orange_dot}\]
% \end{exa}

\begin{rmk}
    It is important to note that a double orange dot may relate two edges that are not necessarily close to each other.
\end{rmk}

The purpose of this double dot is that it allows us to decompose the hollow dot:

\begin{lem}
    \label{decomp_hollow}
    Let $a,b,c\in \N^*$ and $\omega \in \scalars$, the following objects are isomorphic in the relative homotopy canopoly:
    \begin{equation}
        \NB{\tikz[font=\tiny,yscale=.7]{\input{\imagesfolder/sch/split_hollow}}} \cong \NB{\tikz[font=\tiny,yscale=.7]{\input{\imagesfolder/sch/split_hollow_2}}} \quad \text{and} \quad \NB{\tikz[font=\tiny,yscale=.7]{\input{\imagesfolder/sch/split_orange_1}}} \cong \NB{\tikz[font=\tiny,yscale=.7]{\input{\imagesfolder/sch/split_orange_2}}}.
    \end{equation}
    The orientation of the $c$-colored part of the tangle is not specified because the statement holds for both orientations.
\end{lem}

\begin{proof}
     Without loss of generality, we assume that $\omega = 1$. The proof of this lemma reduces to studying the decorations involved in the twists induced by the red dots. We denote by $p_k^{(a)}$ the $k$th power sum symmetric polynomials in $Sym_a$.  For all $n \in \N$, 
    \begin{equation*}
        -\sum_{k=0}^n p_k^{(a+b)}p_{n-k}^{(a+b)} = - \sum_{k=0}^n (p_k^{(a)}+p_k^{(b)})(p_{n-k}^{(a)}+ p_{n-k}^{(b)}) =  -\sum_{k=0}^n p_k^{(a)}p_{n-k}^{(a)} - 2 \sum_{k=0}^n p_k^{(a)}p_{n-k}^{(b)}  -\sum_{k=0}^n p_k^{(b)}p_{n-k}^{(b)} .
    \end{equation*}
    We conclude by identifying each polynomial with its associated dot. The proof of the existence of the second isomorphism is left to the reader.
\end{proof}

\begin{lem}
    \label{slid_bigproof}
    Let $a \in \N^*$, the following objects are isomorphic in the relative homotopy canopoly: 
    \[\NB{\tikz[font=\tiny,scale=1.2]{\input{\imagesfolder/sch/orange_slid_1}}} \cong \NB{\tikz[font=\tiny,scale=1.2]{\input{\imagesfolder/sch/orange_slid_2}}} \quad \text{and} \quad \NB{\tikz[font=\tiny,scale=1.2]{\input{\imagesfolder/sch/orange_slid_3}}} \cong \NB{\tikz[font=\tiny,scale=1.2]{\input{\imagesfolder/sch/orange_slid_4}}}. \]
    The orientation of the $a$-colored part of the tangle is not specified because the statement remains the same for both orientations.
\end{lem}

The idea of the proof of \ref{slid_bigproof} is originally from \cite{KR16}. See also \cite{QRSW21,QRSW23,GR25,ES25} for other adaptations of this lemma.

\begin{proof}
Let \(x\) and \(y\) be the variables in the polynomial ring of decorations associated with the two labeled $1$-facets of \(\mapunzip\). We only need extended webs with one dot, we denote by \(z\) the variable associated with this dot.

Let \(\web\) be an extended web, and let \(\web^{\star_x^z}\) denote the extended web equipped with the associated Maurer--Cartan element given by
\[
\wgen_n \mapsto -h_n(x,z)\Id_\web .
\]
We use the following notations: 
\begin{equation}
    \mathrm{X}_z^x = \NB{\tikz[font=\tiny,scale=1.5]{\input{\imagesfolder/sch/X_sliding_lemma}}} \hspace{2cm} \mathrm{H}_z^x = \NB{\tikz[font=\tiny,scale=1.5]{\input{\imagesfolder/sch/H_sliding_lemma}}}
\end{equation}
We first prove, using the following short exact sequence of complexes, that $\Coker(x-z) \cong \Cone(x-z)$ in the relative homotopy canopoly:
\begin{equation}
\NB{
  \begin{tikzpicture}[xscale =4, yscale=2]
    \node (CT) at (1.65,3) {\(\Coker(x-z)\)};
    \node (Ct) at (1.65,1.9) {\(\Cone(x-z)\)};
    \node (Cb) at (1.65,0.8) {\(\Cone(\Id)\)};
    \node (zeroT) at (1.65,3.5) {$0$};
    \node (zerob) at (1.65,0.3) {$0$};
    \draw[densely dotted] (CT) -- (1.34, 3);
    \draw[densely dotted] (Ct) -- (1.34, 1.9);
    \draw[densely dotted] (Cb) -- (1.34, 0.8);
    \draw[densely dotted, rounded corners] (-0.22, 3.4) rectangle (1.34, 2.59);
    \draw[densely dotted, rounded corners] (-1.45, 2.3) rectangle (1.34, 1.49);
    \draw[densely dotted, rounded corners] (-1.45, 1.2) rectangle (1.34,0.35);
    \draw [-to] (CT) -- (zeroT);
    \draw [-to] (Ct) -- (CT);
    \draw [-to] (Cb) -- (Ct);
    \draw [-to] (zerob) -- (Cb);
    \node (BT) at (0,3) {$\NB{\tikz[font=\tiny,scale=0.7]{\input{\imagesfolder/sch/reidem_lemm_11}}}$};
    \node (AT) at (1.1,3) {$\NB{\tikz[font=\tiny,scale=0.7]{\input{\imagesfolder/sch/reidem_lemm_13}}}$};
    \node (Zt) at (-1.1,1.9) {$\NB{\tikz[font=\tiny,scale=0.7]{\input{\imagesfolder/sch/reidem_lemm_12_xz}}}$};
    \node (Bt) at (0,1.9) {$\NB{\tikz[font=\tiny,scale=0.7]{\input{\imagesfolder/sch/reidem_lemm_12}}}\oplus \NB{\tikz[font=\tiny,scale=0.7]{\input{\imagesfolder/sch/reidem_lemm_14_xz}}}$};
    \node (At) at (1.1,1.9) {$\NB{\tikz[font=\tiny,scale=0.7]{\input{\imagesfolder/sch/reidem_lemm_14_z}}}$};
    \node (Zb) at (-1.1,0.8) {$\NB{\tikz[font=\tiny,scale=0.7]{\input{\imagesfolder/sch/reidem_lemm_12_xz}}}$};
    \node (Bb) at (0,0.8) {$\NB{\tikz[font=\tiny,scale=0.7]{\input{\imagesfolder/sch/reidem_lemm_12_xz}}}\oplus \NB{\tikz[font=\tiny,scale=0.7]{\input{\imagesfolder/sch/reidem_lemm_14_xz}}}$};
    \node (Ab) at (1.1,0.8) {$\NB{\tikz[font=\tiny,scale=0.7]{\input{\imagesfolder/sch/reidem_lemm_14_xz}}}$};
    \draw[-to] (BT) -- (AT) node[pos=0.5, above, scale = 0.7] {$\mapunzip$};
    \draw[-to] (Bt) -- (At) node[pos=0.5, above, scale = 0.7] {$\left(\mapunzip \,\, x-z\right)$};
    \draw[-to] (Bb) -- (Ab) node[pos=0.5, above, scale = 0.7] {$\left(\mapunzip \,\, 1 \right)$};
    \draw[-to] (Zt) -- (Bt) node[pos=0.5, above, scale = 0.7]
    {$\left( \begin{array}{c}z-x \\ \mapunzip\end{array}\right)$};
    \draw[-to] (Zb) -- (Bb) node[pos=0.5, above, scale = 0.7]
    {$\left( \begin{array}{c} 1 \\ -\mapunzip\end{array}\right)$};
    \draw[-to] (Zb) -- (Zt) node[pos=0.5, left, scale=0.7] {$-\Id$};
    \draw[-to] (Ab) -- (At) node[pos=0.5, left, scale=0.7] {$x-z$};
    \draw[-to] (Bb) -- (Bt) node[pos=0.5, left, scale=0.6]
    {$\begin{pmatrix}  x-z & 0 \\ 0 & 1 \end{pmatrix}$};
    \draw[-to] (Bt) -- (BT) node[pos=0.5, left, scale=0.7]
    {$\left(\mathrm{X}_z^x \,\,\, 0 \right)$};
    \draw[-to] (At) -- (AT) node[pos=0.5, left, scale=0.7]
    {$\mathrm{H}_z^x$};
    \end{tikzpicture}}
\ ,
\end{equation}

We proceed similarly for $\Coker(y-z) \cong \Cone(y-z)$:

\begin{equation}
\NB{
  \begin{tikzpicture}[xscale =4, yscale=2]
    \node (CT) at (1.65,3)   {\(\Coker(y-z)\)};
    \node (Ct) at (1.65,1.9) {\(\Cone(y-z)\)};
    \node (Cb) at (1.65,0.8) {\(\Cone(\Id)\)};
    \node (zeroT) at (1.65,3.5) {$0$};
    \node (zerob) at (1.65,0.3) {$0$};
    \draw[densely dotted] (CT) -- (1.34, 3);
    \draw[densely dotted] (Ct) -- (1.34, 1.9);
    \draw[densely dotted] (Cb) -- (1.34, 0.8);
    \draw[densely dotted, rounded corners] (-0.22, 3.4) rectangle (1.34, 2.59);
    \draw[densely dotted, rounded corners] (-1.45, 2.36) rectangle (1.34, 1.49);
    \draw[densely dotted, rounded corners] (-1.45, 1.2) rectangle (1.34,0.35);
    \draw [-to] (CT) -- (zeroT);
    \draw [-to] (Ct) -- (CT);
    \draw [-to] (Cb) -- (Ct);
    \draw [-to] (zerob) -- (Cb);
    \node (BT) at (0,3) {$\NB{\tikz[font=\tiny,scale=0.7]{\input{\imagesfolder/sch/reidem_lemm_15}}}$};
    \node (AT) at (1.1,3) {$\NB{\tikz[font=\tiny,scale=0.7]{\input{\imagesfolder/sch/reidem_lemm_16}}}$};
    \node (Zt) at (-1.1,1.9) {$\NB{\tikz[font=\tiny,scale=0.7]{\input{\imagesfolder/sch/reidem_lemm_12_xz}}}$};
    \node (Bt) at (0,1.9) {$\NB{\tikz[font=\tiny,scale=0.7]{\input{\imagesfolder/sch/reidem_lemm_12}}}\oplus \NB{\tikz[font=\tiny,scale=0.7]{\input{\imagesfolder/sch/reidem_lemm_14_xz}}}$};
    \node (At) at (1.1,1.9) {$\NB{\tikz[font=\tiny,scale=0.7]{\input{\imagesfolder/sch/reidem_lemm_14_z}}}$};
    \node (Zb) at (-1.1,0.8) {$\NB{\tikz[font=\tiny,scale=0.7]{\input{\imagesfolder/sch/reidem_lemm_12_xz}}}$};
    \node (Bb) at (0,0.8) {$\NB{\tikz[font=\tiny,scale=0.7]{\input{\imagesfolder/sch/reidem_lemm_12_yz}}}\oplus \NB{\tikz[font=\tiny,scale=0.7]{\input{\imagesfolder/sch/reidem_lemm_14_xz}}}$};
    \node (Ab) at (1.1,0.8) {$\NB{\tikz[font=\tiny,scale=0.7]{\input{\imagesfolder/sch/reidem_lemm_14_yz}}}$};
    \draw[-to] (BT) -- (AT) node[pos=0.5, above, scale = 0.7] {$\mapunzip$};
    \draw[-to] (Bt) -- (At) node[pos=0.5, above, scale = 0.7] {$\left(\mapunzip \,\, y -z\right)$};
    \draw[-to] (Bb) -- (Ab) node[pos=0.5, above, scale = 0.7] {$\left(\mapunzip \,\, 1 \right)$};
    \draw[-to] (Zt) -- (Bt) node[pos=0.5, above, scale = 0.7]
    {$\left( \begin{array}{c}z-y \\ \mapunzip\end{array}\right)$};
    \draw[-to] (Zb) -- (Bb) node[pos=0.5, above, scale = 0.7]
    {$\left( \begin{array}{c} 1 \\ -\mapunzip\end{array}\right)$};
    \draw[-to] (Zb) -- (Zt) node[pos=0.5, left, scale=0.7] {$-\Id$};
    \draw[-to] (Ab) -- (At) node[pos=0.5, left, scale=0.7] {$y-z$};
    \draw[-to] (Bb) -- (Bt) node[pos=0.5, left, scale=0.6]
    {$\begin{pmatrix}  y-z & 0 \\ 0 & 1 \end{pmatrix}$};
    \draw[-to] (Bt) -- (BT) node[pos=0.5, left, scale=0.7]
    {$\left(z\mapsto y \,\, 0 \right)$};
    \draw[-to] (At) -- (AT) node[pos=0.5, left, scale=0.7]
    {$z\mapsto y$};
    \draw[-to, \crd] ($(Bt)+(0.04,0.25)$) .. controls +(-0.05,0.05) and +(0.05, 0.05) .. +(-0.15, 0) node[scale=0.7,\crd, pos=0.8, above] {(1)};
    \draw[-to, \crd] ($(Bb)+(0.08, -0.22)$) .. controls +(-0.07,-0.07) and +(0.07, -0.07) .. +(-0.2, 0) node[scale=0.7,\crd, pos=0.2, below] {(2)};
    \end{tikzpicture}}
\ .
\end{equation}

The red arrows indicate additional Maurer--Cartan elements on the direct sum objects defined by:
\begin{align}
&\textcolor{\crd}{(1)} : 
L_n \mapsto
(h_n(x,z) - h_n(x,y))
\left(
\begin{array}{cc}
    0 & 0 \\
    0 & \mapzip
\end{array}
\right)\\
&\textcolor{\crd}{(2)} :
L_n \mapsto
\frac{h_n(y,z) - h_n(x,z)}{x-y}
\left(
\begin{array}{cc}
    0 & 0 \\
    0 & \mapzip
\end{array}
\right)
\end{align}

The additional twists on $\Cone(y-z)$ make the following isomorphism to $\Cone(x-y)$ equivariant:
\begin{equation}
\NB{
  \begin{tikzpicture}[xscale=4, yscale=2]
    \node (Zb) at (-1.1,0.8) {$\NB{\tikz[font=\tiny,scale=0.7]{\input{\imagesfolder/sch/reidem_lemm_12_xz}}}$};
    \node (Bb) at (0,0.8) {$\NB{\tikz[font=\tiny,scale=0.7]{\input{\imagesfolder/sch/reidem_lemm_12}}}\oplus \NB{\tikz[font=\tiny,scale=0.7]{\input{\imagesfolder/sch/reidem_lemm_14_xz}}}$};
    \node (Ab) at (1.1,0.8) {$\NB{\tikz[font=\tiny,scale=0.7]{\input{\imagesfolder/sch/reidem_lemm_14_z}}}$};
    \node (Zt) at (-1.1,1.9) {$\NB{\tikz[font=\tiny,scale=0.7]{\input{\imagesfolder/sch/reidem_lemm_12_xz}}}$};
    \node (Bt) at (0,1.9) {$\NB{\tikz[font=\tiny,scale=0.7]{\input{\imagesfolder/sch/reidem_lemm_12}}}\oplus \NB{\tikz[font=\tiny,scale=0.7]{\input{\imagesfolder/sch/reidem_lemm_14_xz}}}$};
    \node (At) at (1.1,1.9) {$\NB{\tikz[font=\tiny,scale=0.7]{\input{\imagesfolder/sch/reidem_lemm_14_z}}}$};
    \draw[-to] (Bt) -- (At) node[pos=0.5, above, scale = 0.7] {$\left(\mapunzip \,\, x-z\right)$};
    \draw[-to] (Bb) -- (Ab) node[pos=0.5, above, scale = 0.7] {$\left(\mapunzip \,\, y-z \right)$};
    \draw[-to] (Zt) -- (Bt) node[pos=0.5, above, scale = 0.7]
    {$\left( \begin{array}{c}z-x \\ \mapunzip\end{array}\right)$};
    \draw[-to] (Zb) -- (Bb) node[pos=0.5, above, scale = 0.7]
    {$\left( \begin{array}{c} z-y \\ \mapunzip\end{array}\right)$};
    \draw[-to] (Zb) -- (Zt) node[pos=0.5, left, scale=0.7] {$\Id$};
    \draw[-to] (Ab) -- (At) node[pos=0.5, left, scale=0.7] {$\Id$};
    \draw[-to] (Bb) -- (Bt) node[pos=0.65, left, scale=0.6]
    {$\begin{pmatrix}  1 & \mapzip \\ 0 & 1 \end{pmatrix}$};
     \draw[-to, \crd] ($(Bb)+(0.08,-0.25)$) .. controls
     +(-0.05,-0.05) and +(0.05, -0.05) .. +(-0.15, 0)
     node[scale=0.7,\crd, pos=0.2, below] {$(1)$};
    \end{tikzpicture}}
. 
\end{equation} 
\end{proof}

The strategy for proving the following lemma is the same as the one used in the proof of Lemma \ref{col_slid_2}.
\begin{lem}
    \label{col_slid_orange}
    Let $a,b \in \N^*$, the following objects are isomorphic in the relative homotopy canopoly: 
    \begin{equation}\NB{\tikz[font=\tiny,scale=1.2]{\input{\imagesfolder/sch/orange_slid_5}}} \cong \NB{\tikz[font=\tiny,scale=1.2]{\input{\imagesfolder/sch/orange_slid_6}}} \quad \text{and} \quad \NB{\tikz[font=\tiny,scale=1.2]{\input{\imagesfolder/sch/orange_slid_7}}} \cong \NB{\tikz[font=\tiny,scale=1.2]{\input{\imagesfolder/sch/orange_slid_8}}}.
    \end{equation}
\end{lem}

Using the blist hypothesis, the previous lemma, and Lemma \ref{decomp_hollow}, an induction on \(c\) is enough to prove the following lemma.

\begin{lem}
    \label{col_slid_orange_gen}
    Let $a,b \in \N^*$, the following objects are isomorphic in the relative homotopy canopoly: 
    \begin{equation}\NB{\tikz[font=\tiny,scale=1.2]{\input{\imagesfolder/sch/orange_slid_9}}} \cong \NB{\tikz[font=\tiny,scale=1.2]{\input{\imagesfolder/sch/orange_slid_10}}} \quad \text{and} \quad \NB{\tikz[font=\tiny,scale=1.2]{\input{\imagesfolder/sch/orange_slid_11}}} \cong \NB{\tikz[font=\tiny,scale=1.2]{\input{\imagesfolder/sch/orange_slid_12}}}.
    \end{equation}
\end{lem}

We can now prove the Proposition \ref{slid_rd}. 
\begin{proof}
    We only detail the proof of the following isomorphism: 
    \begin{equation}
        \NB{\tikz[font=\tiny,xscale=-1,scale=1.2]{\input{\imagesfolder/sch/slid_lemm_19_b}}} \cong \NB{\tikz[font=\tiny,xscale=-1,scale=1.2]{\input{\imagesfolder/sch/slid_lemm_20_b}}}.
    \end{equation}
   We proceed by induction on \(a\). The base case is ensured by Lemma \ref{col_slid_2}. For the inductive step, we blist the \(a\)-colored leg of the tangle, use Lemma \ref{decomp_hollow} to decompose the hollow dot, apply the forkslide isomorphism~(\ref{forkslide}), and slide the orange and hollow red dots using the inductive hypothesis. The conclusion follows from the blist hypothesis. 
\end{proof}

%% file: sch/split_hollow.tex
\begin{scope}
    \draw[->] (0,0.2) -- (0,1) node[pos=1,above]{$a+b$} node[pos=0.5, draw=\crd, fill= white, circle , inner sep=1.5pt]{} node[pos=0.5, left]{$\color{\crd}\omega$};
    \draw[] (-.6,-1) -- (0,0.2) node[pos=0,below]{$a$};
    \draw[] (.6,-1) -- (0,0.2) node[pos=0,below]{$b$};
\end{scope}

%% file: sch/reidem_lemm_12_xz.tex
\begin{scope}
  \draw[orange] (.6,0) -- (1.1,0)  node[pos=0.7, below]{$\color{orange}\omega$};
  \coordinate (bl) at (-0.5, -1);
  \coordinate (br) at ( 0.5, -1);
  \coordinate (bm) at (  0,-0.3);
  \coordinate (tl) at (-0.5,  1);
  \coordinate (tr) at ( 0.5,  1);
  \coordinate (tm) at (  0, 0.3);
  \draw[>-]  (bl) .. controls +( 0, 0.5) and +(0,0) .. (bm);
  \draw[>-]  (br) .. controls +( 0, 0.5) and +(0,0) .. (bm);
  \draw[<-]  (tl) .. controls +( 0, -0.5) and +(0,0) .. (tm) coordinate[pos = 0.25] node[pos= 0.35 , draw=\crd, fill=white,regular polygon, regular polygon sides=3,inner sep=0.9pt] {} node [pos=0.3, left]{$\color{\crd}\lambda$};
  \draw[<-]  (tr) .. controls +( 0, -0.5) and +(0,0) .. (tm)
  coordinate[pos = 0.25] node[pos= 0.35 , draw=\crd, fill=white,regular polygon, regular polygon sides=3,inner sep=0.9pt] {} node[pos=0.3, right]{$\color{\crd}\lambda$};
  \draw [double] (bm) -- (tm) node[pos=0, draw=\crd, fill= white, circle , inner sep=1.5pt]{} node[pos=0, left]{$\color{\crd}s$};
  \draw[] (.8,-.4)  ..controls +(-.3,0.2) and +(-.3,-.2) .. (.8,.4)  node[pos=0.5, draw=orange, fill= white, circle , inner sep=1.5pt]{} node[pos=0, below]{$a$};
  \node at  (1.15,0) [draw=orange, fill= white, circle , inner sep=1.5pt]{};
  \fill (1.15,0) circle (0.05);
  \node at (1.15,.8){$\star_x^z$};
\end{scope}

%% file: sch/reidem_lemm_14_xz.tex
\begin{scope}
    \draw[orange] (.6,0) -- (1.2,0)  node[pos=0.7, below]{$\color{orange}\omega$};
    \coordinate (bl) at (-0.45, -1);
    \coordinate (br) at ( 0.4, -1);
    \coordinate (tl) at (-0.45,  1);
    \coordinate (tr) at ( 0.4,  1);
    \draw[->] (bl) -- (tl);
    \draw[->] (br) -- (tr);
    \draw[] (.9,-.4)  ..controls +(-.3,0.2) and +(-.3,-.2) .. (.9,.4)  node[pos=0.5, draw=orange, fill= white, circle , inner sep=1.5pt]{} node[pos=0, below]{$a$};
    \node at  (1.25,0) [draw=orange, fill= white, circle , inner sep=1.5pt]{};
    \node at (1.25,.8) {$\star_x^z$};
    \fill (1.25,0) circle (0.05);
\end{scope}

%% file: sch/reidem_lemm_14_z.tex
\begin{scope}
    \draw[orange] (.6,0) -- (1.2,0)  node[pos=0.7, below]{$\color{orange}\omega$};
    \coordinate (bl) at (-0.45, -1);
    \coordinate (br) at ( 0.4, -1);
    \coordinate (tl) at (-0.45,  1);
    \coordinate (tr) at ( 0.4,  1);
    \draw[->] (bl) -- (tl);
    \draw[->] (br) -- (tr);
    \draw[] (.9,-.4)  ..controls +(-.3,0.2) and +(-.3,-.2) .. (.9,.4)  node[pos=0.5, draw=orange, fill= white, circle , inner sep=1.5pt]{} node[pos=0, below]{$a$};
    \node at  (1.25,0) [draw=orange, fill= white, circle , inner sep=1.5pt]{};
    \fill (1.25,0) circle (0.05);
\end{scope}

%% file: sch/reidem_lemm_12_yz.tex
\begin{scope}
  \draw[orange] (.6,0) -- (1.1,0)  node[pos=0.7, below]{$\color{orange}\omega$};
  \coordinate (bl) at (-0.5, -1);
  \coordinate (br) at ( 0.5, -1);
  \coordinate (bm) at (  0,-0.3);
  \coordinate (tl) at (-0.5,  1);
  \coordinate (tr) at ( 0.5,  1);
  \coordinate (tm) at (  0, 0.3);
  \draw[>-]  (bl) .. controls +( 0, 0.5) and +(0,0) .. (bm);
  \draw[>-]  (br) .. controls +( 0, 0.5) and +(0,0) .. (bm);
  \draw[<-]  (tl) .. controls +( 0, -0.5) and +(0,0) .. (tm) coordinate[pos = 0.25] node[pos= 0.35 , draw=\crd, fill=white,regular polygon, regular polygon sides=3,inner sep=0.9pt] {} node [pos=0.3, left]{$\color{\crd}\lambda$};
  \draw[<-]  (tr) .. controls +( 0, -0.5) and +(0,0) .. (tm)
  coordinate[pos = 0.25] node[pos= 0.35 , draw=\crd, fill=white,regular polygon, regular polygon sides=3,inner sep=0.9pt] {} node[pos=0.3, right]{$\color{\crd}\lambda$};
  \draw [double] (bm) -- (tm) node[pos=0, draw=\crd, fill= white, circle , inner sep=1.5pt]{} node[pos=0, left]{$\color{\crd}s$};
  \draw[] (.8,-.4)  ..controls +(-.3,0.2) and +(-.3,-.2) .. (.8,.4)  node[pos=0.5, draw=orange, fill= white, circle , inner sep=1.5pt]{} node[pos=0, below]{$a$};
  \node at  (1.15,0) [draw=orange, fill= white, circle , inner sep=1.5pt]{};
  \node at (1.4,.8){$\star_y^z$};
  \fill (1.15,0) circle (0.05);
\end{scope}

%% file: sch/reidem_lemm_14_yz.tex
\begin{scope}
    \draw[orange] (.6,0) -- (1.2,0)  node[pos=0.7, below]{$\color{orange}\omega$};
    \coordinate (bl) at (-0.45, -1);
    \coordinate (br) at ( 0.4, -1);
    \coordinate (tl) at (-0.45,  1);
    \coordinate (tr) at ( 0.4,  1);
    \draw[->] (bl) -- (tl);
    \draw[->] (br) -- (tr);
    \draw[] (.9,-.4)  ..controls +(-.3,0.2) and +(-.3,-.2) .. (.9,.4)  node[pos=0.5, draw=orange, fill= white, circle , inner sep=1.5pt]{} node[pos=0, below]{$a$};
    \node at  (1.25,0) [draw=orange, fill= white, circle , inner sep=1.5pt]{};
    \node at (1.5,.8) {$\star_y^z$};
    \fill (1.25,0) circle (0.05);
\end{scope}

%% file: Z_functoriality.tex
\begin{dfn}
    Let $\TD$ be the subcanopoly of $\WT$ defined by: 
    \begin{itemize}
        \item The objects are colored oriented tangle diagrams in $\Disk$.
        \item The morphisms are two dimensional oriented framed colored cobordisms between tangles, embedded in $\Disk \times \interval \times \interval$. 
    \end{itemize} 
\end{dfn}

\begin{thm}[{\cite[Th.4.5]{ETW17}}]
    The construction of the colored Khovanov--Rozansky complex defines a canopoly morphism
    \[\colcom_N : \TD \rightarrow \mathcal{C}(\catf).\]
\end{thm}

By introducing a $\Witt$-action on colored Khovanov--Rozansky homology, one might expect the target canopoly of our canopoly functor to be a collection of category of $\Witt$-modules. However, this is not the case, since the morphisms induced by cobordisms are not, in general, $\Witt$-equivariant. In the following, we define a source canopoly whose image under the canopoly morphism $\colcom_N$ consists of $\Witt$-equivariant morphisms.

In order for the image of a cobordism in the canopoly $\TD$ to be $\Witt$-equivariant, we introduce red dots on tangles. These dots compensate for the $\Witt$-action on cups, caps, and saddles in two-dimensional cobordisms. Using the sliding lemmas, we can gather the red dots (hollow, solid and triangular) on each connected component of the tangle. This observation motivates the following definition:

\begin{dfn}
    Let us define the canopoly $\rTD$:
    \begin{itemize}
        \item The objects are quadruples \((T,(a_i),(b_i),(c_i))\), where \(T\) is an object of \(\TD\) and \((a_i),(b_i),(c_i) \in \scalars^{h}\) are respectively the labels of the hollow, solid, and triangular red dots on each component of \(T\). Here, \(h\) denotes the number of components of \(T\).
        \item A morphism $C$ from $(T,(a_i),(b_i),(c_i))$ to $(T',(a'_i),(b'_i),(c'_i))$ is a morphism in $\TD$ from $T$ to $T'$ such that for all connected component of $C$ denoted $C_j$:
        \begin{equation}
            \label{id_1}
            \chi(C_j) - k_j = 2 \left (\sum_{\partial_1 C_i}b'_i - \sum_{\partial_0 C_i}b_i \right ), \quad \text{and} \quad k_j = - 2 \left (\sum_{\partial_1 C_i}a'_i - \sum_{\partial_0 C_i}a_i \right ) = 2 \left (\sum_{\partial_1 C_i}c'_i - \sum_{\partial_0 C_i}c_i \right ).      
        \end{equation}
        where $\chi$ is the Euler characteristic and $k_j$ is the framing of $C_j$.
    \end{itemize} 
\end{dfn}

\begin{rmk}
    The identity \eqref{id_1} comes from the dots needed for a Reidemeister~1 isomorphism, a saddle, a cup, or a cap.
\end{rmk}

\begin{thm}
    The construction of the colored Khovanov--Rozansky complex defines a canopoly morphism
    \[\colcom_N : \rTD \rightarrow \mathcal{C}^{\Witt}(\mathrm{r}\extcatf),\]
    where the \(\Witt\)-action on the image of a cobordism is trivial.
\end{thm}

If we restrict the previous canopoly morphism to the category of tangles without boundary (\(\bnd = \varnothing\)), that is, to the category of colored oriented links, we can compose this functor with the TQFT-like functor \(\ssp_N\). This gives the following statement:

\begin{cor}
    The construction of the colored Khovanov--Rozansky complex defines a functor
    \[\colcom_N : \rTD(\varnothing) \rightarrow \mathcal{C}^{\Witt}(\scalars_N\text{-proj}_{\text{gr}}),\]
    where the image of a cobordism is a $\Witt$-equivariant map. By taking the homology of the complex defining the functor,
    \[\mathrm{cHKRW}_N : \rTD(\varnothing) \rightarrow \scalars_N \# \Witt\text{-proj}_{\text{bigr}},\]
    where the target category is a category of $\scalars_N$ and $\Witt$ bigraded projective modules.
\end{cor}

%% file: biblio.bib
@misc{QRSW23,
  author        = {Qi, You and Robert, Louis-Hadrien and Sussan, Joshua and Wagner, Emmanuel},
  title         = {Symmetries of equivariant {K}hovanov--{R}ozansky homology},
  year          = {2023},
  eprint        = {2306.10729},
  archivePrefix = {arXiv},
  primaryClass  = {math.GT},
  url           = {https://arxiv.org/abs/2306.10729},
}

@book{neeman2014triangulated,
  author    = {Neeman, Amnon},
  title     = {Triangulated categories},
  series    = {Annals of Mathematics Studies},
  volume    = {148},
  publisher = {Princeton University Press},
  address   = {Princeton, NJ},
  year      = {2001},
  pages     = {viii+449},
  isbn      = {0-691-08685-0; 0-691-08686-9},
  doi       = {10.1515/9781400837212},
  url       = {https://doi.org/10.1515/9781400837212},
}

@article{qi_p-differential_2022,
  author    = {Qi, You and Sussan, Joshua},
  title     = {On some {$p$}-differential graded link homologies},
  journal   = {Forum Math. Pi},
  volume    = {10},
  year      = {2022},
  pages     = {Paper No. e26, 58},
  doi       = {10.1017/fmp.2022.19},
  url       = {https://doi.org/10.1017/fmp.2022.19},
}

@article{QRSW24,
  author        = {Qi, You and Robert, Louis-Hadrien and Sussan, Joshua and Wagner, Emmanuel},
  title         = {Symmetries of {$\mathfrak{gl}_N$}-foams},
  journal       = {Quantum Topol.},
  volume        = {16},
  year          = {2025},
  doi           = {10.4171/QT/215},
  eprint        = {2212.10106},
  archivePrefix = {arXiv},
  primaryClass  = {math.GT},
}

@article{RW20a,
  author    = {Robert, Louis-Hadrien and Wagner, Emmanuel},
  title     = {A closed formula for the evaluation of foams},
  journal   = {Quantum Topol.},
  volume    = {11},
  number    = {3},
  year      = {2020},
  pages     = {411--487},
  doi       = {10.4171/qt/139},
  url       = {https://doi.org/10.4171/qt/139},
}

@misc{QRSW21,
  author        = {Qi, You and Robert, Louis-Hadrien and Sussan, Joshua and Wagner, Emmanuel},
  title         = {A categorification of the colored Jones polynomial at a root of unity},
  year          = {2021},
  eprint        = {2111.13195},
  archivePrefix = {arXiv},
  primaryClass  = {math.QA},
}

@article{ETW17,
  author    = {Ehrig, Michael and Tubbenhauer, Daniel and Wedrich, Paul},
  title     = {Functoriality of colored link homologies},
  journal   = {Proc. Lond. Math. Soc. (3)},
  volume    = {117},
  number    = {5},
  year      = {2018},
  pages     = {996--1040},
  doi       = {10.1112/plms.12154},
  url       = {https://doi.org/10.1112/plms.12154},
}

@article{BHMV92,
  author    = {Blanchet, C. and Habegger, N. and Masbaum, G. and Vogel, P.},
  title     = {Three-manifold invariants derived from the {K}auffman bracket},
  journal   = {Topology},
  volume    = {31},
  number    = {4},
  year      = {1992},
  pages     = {685--699},
  doi       = {10.1016/0040-9383(92)90002-Y},
  url       = {https://doi.org/10.1016/0040-9383(92)90002-Y},
}

@article{Kho0,
  author    = {Khovanov, Mikhail},
  title     = {A categorification of the {J}ones polynomial},
  journal   = {Duke Math. J.},
  volume    = {101},
  number    = {3},
  year      = {2000},
  pages     = {359--426},
  doi       = {10.1215/S0012-7094-00-10131-7},
  url       = {https://doi.org/10.1215/S0012-7094-00-10131-7},
}

@article{KR08a,
  author    = {Khovanov, Mikhail and Rozansky, Lev},
  title     = {Matrix factorizations and link homology},
  journal   = {Fund. Math.},
  volume    = {199},
  number    = {1},
  year      = {2008},
  pages     = {1--91},
  doi       = {10.4064/fm199-1-1},
  url       = {https://doi.org/10.4064/fm199-1-1},
}

@article{KR08b,
  author    = {Khovanov, Mikhail and Rozansky, Lev},
  title     = {Matrix factorizations and link homology. II},
  journal   = {Geom. Topol.},
  volume    = {12},
  number    = {3},
  year      = {2008},
  pages     = {1387--1425},
  doi       = {10.2140/gt.2008.12.1387},
  url       = {https://doi.org/10.2140/gt.2008.12.1387},
}

@article{KR16,
  author    = {Khovanov, Mikhail and Rozansky, Lev},
  title     = {Positive half of the {W}itt algebra acts on triply graded link homology},
  journal   = {Quantum Topol.},
  volume    = {7},
  number    = {4},
  year      = {2016},
  pages     = {737--795},
  doi       = {10.4171/QT/84},
  url       = {https://doi.org/10.4171/QT/84},
}

@article{GOR13,
  author    = {Gorsky, Eugene and Oblomkov, Alexei and Rasmussen, Jacob},
  title     = {On stable {K}hovanov homology of torus knots},
  journal   = {Exp. Math.},
  volume    = {22},
  number    = {3},
  year      = {2013},
  pages     = {265--281},
  doi       = {10.1080/10586458.2013.798553},
  url       = {https://doi.org/10.1080/10586458.2013.798553},
}

@article{GHM24,
  author    = {Gorsky, Eugene and Hogancamp, Matthew and Mellit, Anton},
  title     = {Tautological classes and symmetry in {K}hovanov--{R}ozansky homology},
  journal   = {Duke Math. J.},
  volume    = {173},
  number    = {13},
  year      = {2024},
  pages     = {2481--2561},
  doi       = {10.1215/00127094-2023-0063},
  url       = {https://doi.org/10.1215/00127094-2023-0063},
}

@article{CG24,
  author    = {Chandler, Alex and Gorsky, Eugene},
  title     = {Structures in {HOMFLY}-{PT} homology},
  journal   = {Exp. Math.},
  volume    = {33},
  number    = {4},
  year      = {2024},
  pages     = {817--842},
  doi       = {10.1080/10586458.2023.2300412},
  url       = {https://doi.org/10.1080/10586458.2023.2300412},
}

@article{GLW18,
  author    = {Grigsby, J. Elisenda and Licata, Anthony M. and Wehrli, Stephan M.},
  title     = {Annular {K}hovanov homology and knotted {S}chur--{W}eyl representations},
  journal   = {Compos. Math.},
  volume    = {154},
  number    = {3},
  year      = {2018},
  pages     = {459--502},
  doi       = {10.1112/S0010437X17007540},
  url       = {https://doi.org/10.1112/S0010437X17007540},
}

@article{Kho04,
  author    = {Khovanov, Mikhail},
  title     = {${\mathfrak{sl}_3}$ link homology},
  journal   = {Algebr. Geom. Topol.},
  volume    = {4},
  year      = {2004},
  pages     = {1045--1081},
  doi       = {10.2140/agt.2004.4.1045},
  url       = {https://doi.org/10.2140/agt.2004.4.1045},
}

@article{KR07,
  author    = {Khovanov, Mikhail and Rozansky, Lev},
  title     = {Topological {L}andau--{G}inzburg models on the world-sheet foam},
  journal   = {Adv. Theor. Math. Phys.},
  volume    = {11},
  number    = {2},
  year      = {2007},
  pages     = {233--259},
  doi       = {10.4310/atmp.2007.v11.n2.a2},
  url       = {https://doi.org/10.4310/atmp.2007.v11.n2.a2},
}

@article{QR16,
  author    = {Queffelec, Hoel and Rose, David E. V.},
  title     = {The {$\mathfrak{sl}_n$} foam 2-category: a combinatorial formulation of {K}hovanov--{R}ozansky homology via categorical skew {H}owe duality},
  journal   = {Adv. Math.},
  volume    = {302},
  year      = {2016},
  pages     = {1251--1339},
  doi       = {10.1016/j.aim.2016.07.027},
  url       = {https://doi.org/10.1016/j.aim.2016.07.027},
}

@article{Bar05,
  author    = {Bar-Natan, Dror},
  title     = {Khovanov's homology for tangles and cobordisms},
  journal   = {Geom. Topol.},
  volume    = {9},
  year      = {2005},
  pages     = {1443--1499},
  doi       = {10.2140/gt.2005.9.1443},
  url       = {https://doi.org/10.2140/gt.2005.9.1443},
}

@article{Wang24,
  author    = {Wang, Joshua},
  title     = {On {$\mathfrak{sl}(N)$} link homology with mod {$N$} coefficients},
  journal   = {Quantum Topol.},
  volume    = {15},
  number    = {1},
  year      = {2024},
  pages     = {87--121},
  doi       = {10.4171/qt/194},
  url       = {https://doi.org/10.4171/qt/194},
}

@article{Jon21,
  author    = {Jones, V. F. R.},
  title     = {Planar algebras, I},
  journal   = {New Zealand J. Math.},
  volume    = {52},
  year      = {2021},
  pages     = {1--107},
  doi       = {10.53733/172},
  url       = {https://doi.org/10.53733/172},
}

@misc{GR25,
  author        = {Gu{\'e}rin, Alexis and Roz, Felix},
  title         = {An action of the {W}itt algebra on {K}hovanov--{R}ozansky homology},
  year          = {2025},
  eprint        = {2501.19096},
  archivePrefix = {arXiv},
  primaryClass  = {math.GT},
  url           = {https://arxiv.org/abs/2501.19096},
}

@article{W07,
  author    = {Webster, Ben},
  title     = {Khovanov--{R}ozansky homology via a canopolis formalism},
  journal   = {Algebr. Geom. Topol.},
  volume    = {7},
  year      = {2007},
  pages     = {673--699},
  doi       = {10.2140/agt.2007.7.673},
  url       = {https://doi.org/10.2140/agt.2007.7.673},
}

@article{MSV09,
  author    = {Mackaay, Marco and Sto{\v s}i{\'c}, Marko and Vaz, Pedro},
  title     = {{$\mathfrak{sl}(N)$}-link homology ({$N\geq4$}) using foams and the {K}apustin--{Li} formula},
  journal   = {Geom. Topol.},
  volume    = {13},
  number    = {2},
  year      = {2009},
  pages     = {1075--1128},
  doi       = {10.2140/gt.2009.13.1075},
  url       = {https://doi.org/10.2140/gt.2009.13.1075},
}

@article{MOY98,
  author    = {Murakami, Hitoshi and Ohtsuki, Tomotada and Yamada, Shuji},
  title     = {HOMFLY polynomial via an invariant of colored plane graphs},
  journal   = {Enseign. Math. (2)},
  volume    = {44},
  number    = {3--4},
  year      = {1998},
  pages     = {325--360},
}

@article{MN08,
  author    = {Morrison, Scott and Nieh, Ari},
  title     = {On {K}hovanov's cobordism theory for {$\mathfrak{su}_3$} knot homology},
  journal   = {J. Knot Theory Ramifications},
  volume    = {17},
  number    = {9},
  year      = {2008},
  pages     = {1121--1173},
  doi       = {10.1142/S0218216508006555},
  url       = {https://doi.org/10.1142/S0218216508006555},
}

@incollection{Ca14,
  author    = {Caprau, Carmen},
  title     = {A cohomology theory for colored tangles},
  booktitle = {Knots in Poland. III. Part 1},
  series    = {Banach Center Publ.},
  volume    = {100},
  pages     = {13--25},
  publisher = {Polish Acad. Sci. Inst. Math.},
  address   = {Warsaw},
  year      = {2014},
  doi       = {10.4064/bc100-0-1},
  url       = {https://doi.org/10.4064/bc100-0-1},
}

@misc{ES25,
  author        = {Ebert, Mark and Schelstraete, L{\'e}o},
  title         = {A local {$\mathfrak{gl}_{1|1}$}-action on odd Khovanov homology},
  year          = {2025},
  eprint        = {2511.00947},
  archivePrefix = {arXiv},
  primaryClass  = {math.GT},
}

@article{KS25,
  author        = {Khovanov, Mikhail and Sano, Taketo},
  title         = {Symmetries of equivariant Khovanov homology},
  year          = {2025},
  eprint        = {2509.03785v2},
  archivePrefix = {arXiv},
  primaryClass  = {math.GT},
}

@article{Ru74,
  author    = {Rudakov, A. N.},
  title     = {Irreducible representations of infinite-dimensional Lie algebras of Cartan type},
  journal   = {Math. USSR-Izv.},
  volume    = {8},
  number    = {4},
  year      = {1974},
  pages     = {836--866},
  doi       = {10.1070/IM1974v008n04ABEH002129},
  url       = {https://doi.org/10.1070/IM1974v008n04ABEH002129},
}

@article{Shumakovitch_2014,
  author    = {Shumakovitch, Alexander N.},
  title     = {Torsion of {K}hovanov homology},
  journal   = {Fund. Math.},
  volume    = {225},
  number    = {1},
  year      = {2014},
  pages     = {343--364},
  doi       = {10.4064/fm225-1-16},
  url       = {https://doi.org/10.4064/fm225-1-16},
}

@article{CK08,
  author    = {Cautis, Sabin and Kamnitzer, Joel},
  title     = {Knot homology via derived categories of coherent sheaves. I. The {$\mathfrak{sl}(2)$}-case},
  journal   = {Duke Math. J.},
  volume    = {142},
  number    = {3},
  year      = {2008},
  pages     = {511--588},
  doi       = {10.1215/00127094-2008-012},
  url       = {https://doi.org/10.1215/00127094-2008-012},
}

@article{CK09,
  author    = {Cautis, Sabin and Kamnitzer, Joel},
  title     = {Knot homology via derived categories of coherent sheaves. II. {$\mathfrak{sl}_m$} case},
  journal   = {Invent. Math.},
  volume    = {174},
  number    = {1},
  year      = {2008},
  pages     = {165--232},
  doi       = {10.1007/s00222-008-0138-6},
  url       = {https://doi.org/10.1007/s00222-008-0138-6},
}

@article{Man07,
  author    = {Manolescu, Ciprian},
  title     = {Link homology theories from symplectic geometry},
  journal   = {Adv. Math.},
  volume    = {211},
  number    = {1},
  year      = {2007},
  pages     = {363--416},
  doi       = {10.1016/j.aim.2006.09.007},
  url       = {https://doi.org/10.1016/j.aim.2006.09.007},
}

@article{SS06,
  author    = {Seidel, Paul and Smith, Ivan},
  title     = {A link invariant from the symplectic geometry of nilpotent slices},
  journal   = {Duke Math. J.},
  volume    = {134},
  number    = {3},
  year      = {2006},
  pages     = {453--514},
  doi       = {10.1215/S0012-7094-06-13432-4},
  url       = {https://doi.org/10.1215/S0012-7094-06-13432-4},
}

@article{Yo11,
  author    = {Yonezawa, Yasuyoshi},
  title     = {Quantum {$({\mathfrak{sl}_n}, \wedge V_n)$} link invariant and matrix factorizations},
  journal   = {Nagoya Math. J.},
  volume    = {204},
  year      = {2011},
  pages     = {69--123},
  doi       = {10.1215/00277630-1431840},
  url       = {https://doi.org/10.1215/00277630-1431840},
}

@article{Wu07,
  author    = {Wu, Hao},
  title     = {A colored {$\mathfrak{sl}(N)$}-homology for links in {$S^3$}},
  journal   = {Dissertationes Math.},
  volume    = {499},
  year      = {2009},
  doi       = {10.4064/dm499-0-1},
  url       = {https://doi.org/10.4064/dm499-0-1},
}
